\documentclass{amsart}

\newtheorem{thm}{Theorem}[section]
\newtheorem{cor}[thm]{Corollary}
\newtheorem{lem}[thm]{Lemma}
\newtheorem{pro}[thm]{Proposition}
\newtheorem{dfn}[thm]{Definition}
\newtheorem{hypothesis}[thm]{Hypothesis}
\newtheorem{rmk}[thm]{Remark}

\newtheorem{conjecture}[thm]{Conjecture}

\numberwithin{equation}{section}

\usepackage{fullpage, amsmath}
\usepackage{amsthm,amsfonts,url,amssymb}
\usepackage{mathrsfs}
\usepackage{amssymb,amscd}
\usepackage{color}
\usepackage{chemarrow}
\usepackage{pictexwd,dcpic}
\usepackage{extarrows}
\usepackage{enumitem}
\usepackage{lipsum}
\usepackage[all,cmtip]{xy}
\usepackage{leftidx}
\usepackage[version=3]{mhchem}
\usepackage{bm}
\usepackage{stmaryrd}

\newcommand{\hooklongrightarrow}{\lhook\joinrel\longrightarrow}
\newcommand{\twoheadlongrightarrow}{\relbar\joinrel\twoheadrightarrow}

\newcommand{\ana}{{\rm an}}

\newcommand{\Art}{{\rm Art}}

\newcommand{\alge}{{\rm alg}}

\newcommand{\bdr}{{B^+_{\mathrm{dR}}}}
\newcommand{\bersteineigenvarpik}{{\mathcal{E}_{\omepik,\fp,\bm{\lambda}}^\infty(\overline{\rho})}}

\newcommand{\ccyc}{{\epsilon}}

\newcommand{\cokerr}{\rm coker}

\newcommand{\df}{{\mathbf{DF}}}
\newcommand{\dr}{{\rm dR}}
\newcommand{\Dpik}{{\mathbf{D}}}

\newcommand{\defvar}{{ X_{\omepik,\mathbf{{h}}}(\overline{r})}}
\newcommand{\defvarrho}{{ X_{\omepik,\mathbf{{h}}}(\overline{r})}}

\newcommand{\EndO}{{\rm End}}
\newcommand{\ext}{{\rm Ext}}
\newcommand{\EX}{{E^\times}}

\newcommand{\fil}{{\rm Fil}}

\newcommand{\gal}{{\rm Gal}}

\newcommand{\GLN}{{\rm GL}}
\newcommand{\GL}{{\rm GL}}

\newcommand{\gr}{{\rm gr}}

\newcommand{\homo}{{\rm Hom}}
\newcommand{\hH}{{\mathrm{H}}}
\newcommand{\hpi}{{{\mathbf{h}}}}

\newcommand{\ind}{{\rm Ind}}
\newcommand{\im}{{\rm Im}}

\newcommand{\lamoverx}{{L({\bm\lambda})}}

\newcommand{\lrr}{{\langle r\rangle}}
\newcommand{\lalg}{{\rm lalg}}
\newcommand{\lra}{\longrightarrow}

\newcommand{\lp}{{\rm lp}}

\newcommand{\MSpec}{{\rm MaxSpec}}

\newcommand{\op}{{\overline{{\mathbf{P}}}}}
\newcommand{\ob}{{\overline{{\mathbf{B}}}}}
\newcommand{\on}{{\overline{{\mathbf{N}}}}}

\newcommand{\obp}{\overline{\bQ}_p^\times}
\newcommand{\omepik}{\Omega_{\beta}}

\newcommand{\pol}{{\rm pol}}

\newcommand{\padelc}{{(\Delta_{\Dpik}^z)}}
\newcommand{\padel}{{(\Delta_{\Dpik}^0)}}
\newcommand{\pr}{{\rm pr}}
\newcommand{\pst}{{\rm pst}}

\newcommand{\red}{{\rm red}}

\newcommand{\rig}{{\rm rig}}
\newcommand{\ra}{{\longrightarrow}}
\newcommand{\rec}{{\rm rec}_L}

\newcommand{\res}{{\rm Res}}
\newcommand{\rigch}{{\mathcal{Z}_{{\beta},\mathcal{O}_L}}}
\newcommand{\rigchl}{{\mathcal{Z}_{{\beta},L}}}
\newcommand{\rigerch}{{\mathcal{Z}_{{\beta},\mathcal{O}_L}}}
\newcommand{\rigerchl}{{\mathcal{Z}_{{\beta},L}}}

\newcommand{\Spec}{{\rm Spec}}
\newcommand{\Spf}{{\mathrm{Spf}}}

\newcommand{\st}{{\rm St}}
\newcommand{\soc}{{\rm soc}}
\newcommand{\sm}{{\rm sm}}

\newcommand{\sbanpik}{{\big(\mathrm{Spec}\hspace{2pt}\mathfrak{Z}_{\omepik}\big)^{\mathrm{rig}}}}
\newcommand{\sbanpi}{{\left(\Spec\mathfrak{Z}_{\Omega_{r}}\right)^{\mathrm{rig}}}}

\newcommand{\tee}{{{\otimes}_{\cR_{E,L}}}}
\newcommand{\te}{{\otimes_E}}

\newcommand{\ur}{{\rm ur}}
\newcommand{\unr}{{\rm unr}}
\newcommand{\unra}{{\rm unr}}
\newcommand{\univer}{{\rm univ}}
\newcommand{\undelram}{{\delta^0}}

\newcommand{\ul}{\underline}

\newcommand{\val}{\mathrm{val}}

\newcommand{\wdre}{{\mathbf{r}}}
\newcommand{\wt}{{\rm wt}}

\newcommand{\wdr}{{W^+_{\mathrm{dR}}}}

\newcommand{\xfpx}{X^{X^\fp-\mathrm{aut}}_{\omepik}(y)}

\newcommand{\zuni}{{\widehat{Z_{\beta,\varpi_L}}}}

\newcommand{\BA}{{\mathbb{A}}}

\newcommand{\BC}{{\mathbb{C}}}

\newcommand{\BG}{{\mathbb{G}}}

\newcommand{\BQ}{{\mathbb{Q}}}

\newcommand{\BU}{{\mathbb{U}}}

\newcommand{\BZ}{{\mathbb{Z}}}

\newcommand{\bB}{{\mathbf{B}}}

\newcommand{\bD}{{\mathbf{D}}}

\newcommand{\bG}{{\mathbf{G}}}

\newcommand{\bI}{{\mathbf{I}}}

\newcommand{\bL}{{\mathbf{L}}}

\newcommand{\bN}{{\mathbf{N}}}

\newcommand{\bP}{{\mathbf{P}}}
\newcommand{\bQ}{{\mathbf{Q}}}

\newcommand{\bT}{{\mathbf{T}}}

\newcommand{\bW}{{\mathbf{W}}}

\newcommand{\bZ}{{\mathbf{Z}}}

\newcommand{\bv}{{\mathbf{v}}}
\newcommand{\br}{{\mathbf{r}}}
\newcommand{\bh}{{\mathbf{h}}}
\newcommand{\bk}{{\mathbf{k}}}

\newcommand{\cL}{\mathcal L}
\newcommand{\co}{\mathcal O}

\newcommand{\cR}{\mathcal R}

\newcommand{\cC}{\mathcal C}

\newcommand{\cW}{\mathcal W}

\newcommand{\cM}{\mathcal M}
\newcommand{\cF}{\mathcal F}

\newcommand{\cE}{\mathcal E}

\newcommand{\cO}{\mathcal O}

\newcommand{\cZ}{\mathcal Z}

\newcommand{\FI}{{\mathfrak{I}}}

\newcommand{\FX}{{\mathfrak{X}}}

\newcommand{\FZ}{{\mathfrak{Z}}}

\newcommand{\fb}{{\mathfrak{b}}}

\newcommand{\fd}{{\mathfrak{d}}}

\newcommand{\fg}{{\mathfrak{g}}}
\newcommand{\fh}{{\mathfrak{h}}}

\newcommand{\fl}{{\mathfrak{l}}}
\newcommand{\fm}{{\mathfrak{m}}}
\newcommand{\fn}{{\mathfrak{n}}}

\newcommand{\fp}{{\mathfrak{p}}}

\newcommand{\ft}{{\mathfrak{t}}}

\newcommand{\fz}{{\mathfrak{z}}}

\newcommand{\sE}{\mathscr E}

\newcommand{\sF}{\mathscr F}

\newcommand{\sW}{\mathscr W}
\newcommand{\sL}{\mathscr L}

\begin{document}

\title{Parabolic Simple $\sL$-invariants}

\author{Yiqin He}
\address{School of Mathematical Sciences, Peking University,\;No. 5 Yiheyuan Road, Haidian District, Beijing 100871, P.R. China}
\curraddr{Morningside Center of Mathematics, Chinese Academy of Sciences,\;No. 55 Zhongguancun East Road, Haidian District, Beijing 100190, P.R. China}
\email{heyiqin@amss.ac.cn}
\thanks{This work was supported by NSFC Grants No. 8200905010 and No. 8200800065.}





\begin{abstract}
Let $L$ be a finite extension of $\bQ_p$, and let $\rho_L$ be a potentially semistable, non-crystalline $p$-adic Galois representation whose associated $F$-semisimple Weil--Deligne representation is absolutely indecomposable.\;We study the parabolic simple Fontaine--Mazur $\sL$-invariants of $\rho_L$, which have previously been understood only in the trianguline case.\;Building on \cite{2022ext1hyq}, which introduced Breuil's parabolic simple $\sL$-invariants, we attach to $\rho_L$ a locally $\bQ_p$-analytic representation $\Pi(\rho_L)$ of $\GLN_n(L)$ that encodes these invariants.\;When $\rho_L$ arises from a patched automorphic representation of $\bG(\BA_{F^+})$, where $\bG$ is a definite unitary group over a totally real field $F^+$ that is compact at the infinite places and isomorphic to $\GLN_n$ at the $p$-adic places, we prove, under mild hypotheses, that $\Pi(\rho_L)$ embeds into the associated Hecke-isotypic subspace of the Banach space of patched $p$-adic automorphic forms on $\bG(\BA_{F^+})$.\;Equivalently, Breuil's parabolic simple $\sL$-invariants agree with the parabolic simple Fontaine--Mazur $\sL$-invariants.\;
\end{abstract}

\maketitle

\tableofcontents

\setlength{\baselineskip}{14pt}

\section{Introduction and notation}

Let $L$ and $E$ be finite extensions of $\bQ_p$.\;Assume that $E$ is sufficiently large to contain the images of all embeddings $\Sigma_L:=\{\sigma:L\hookrightarrow\overline{\bQ}_p\}$ of $L$ into $\overline{\bQ}_p$.\;Let $\rho_L:\gal_L\rightarrow\GLN_n(E)$ be a regular  de Rham $p$-adic Galois representation, where $\gal_L$ denotes the absolute Galois group of $L$.\;Fontaine's theory associates with $\rho_L$ an $n$-dimensional Weil--Deligne representation $\wdre_{\rho_L}$ and, via the classical local Langlands correspondence, an irreducible smooth representation $\pi_{\mathrm{sm}}(\rho_L)$ of $\GLN_n(L)$ over $E$.\;Let $\bh:=(\hpi_{1,\tau}>\hpi_{2,\tau}>\cdots>\hpi_{n,\tau})_{\tau\in\Sigma_L}$ denote the Hodge--Tate weights of $\rho_L$, and put $\theta=(0,-1,\ldots,1-n)_{\tau\in\Sigma_L}$.\;The locally $\bQ_p$-algebraic representation $\pi_{\lalg}(\rho_L):=\pi_{\mathrm{sm}}(\rho_L)\otimes_E L(\bh-\theta)$ is expected to be the locally $\bQ_p$-algebraic subrepresentation of the conjectural locally analytic representation $\pi_{\ana}(\rho_L)$ attached to $\rho_L$ by the $p$-adic local Langlands correspondence; here, $L(\bh-\theta)$ denotes the algebraic representation of $\GLN_n(L)$ with highest weight $\bh-\theta$.\;The representation $\pi_{\lalg}(\rho_L)$ records only the $F$-semisimplification of $\wdre_{\rho_L}$ and the Hodge--Tate weights $\bh$, and thus loses the information contained in the Hodge filtration of $\rho_L$.\;A fundamental problem in the locally analytic $p$-adic local Langlands program is to recover this filtration from locally $\bQ_p$-analytic representations of $\GLN_n(L)$.\;The $p$-adic Langlands correspondence for $\GLN_2(\bQ_p)$ and its $p$-adic local-global compatibility were established through the work of many authors, notably Breuil's foundational work \cite{Breuil03a,Breuil03b}, Colmez \cite{PADINGL2COL}, and Emerton \cite{localglobalgl2}.\;Current research focuses on higher-dimensional representations and on base fields other than $\bQ_p$; the case $n\geq3$ is substantially more difficult.\;

The so-called $\sL$-\textit{invariants} also play a role in the study of the $p$-adic Langlands correspondence.\;Their study began with the work of Mazur--Tate--Teitelbaum, in which the $\sL$-invariant $\cL(f)$ describes the derivative of the $p$-adic $L$-function of a modular form $f$ at an exceptional zero.\;Several other definitions were subsequently proposed, and these invariants were eventually shown to agree in the modular-form setting.\;More precisely, let $\rho_f$ denote the $\gal(\overline{\bQ}/\bQ)$-representation attached to $f$ by the Langlands correspondence, and let $\rho_{f,p}:=\rho_f|_{\gal_{\bQ_p}}$. Then $\cL(f)$ can be recovered explicitly from a Hodge-theoretic parameter in the admissible Hodge filtration on $D_{\mathrm{st}}(\rho_{f,p})$; we call this parameter the Fontaine--Mazur $\sL$-invariant.\;From the perspective of the $p$-adic Langlands program, Breuil constructs an explicit finite-length locally analytic representation $\Pi_{f,p}$ of $\GLN_2(\bQ_p)$ whose isomorphism class determines $\cL(f)$.\;He also proves a $p$-adic local-global compatibility result by embedding $\Pi_{f,p}$ into the $f$-isotypic component of the completed cohomology of modular curves.\;We use the term \textit{Breuil's $\sL$-invariants} for invariants defined via this representation-theoretic construction.\;


The guiding philosophy of the locally analytic $p$-adic local Langlands program, inspired by Breuil's ideas on $\sL$-invariants, is to generalize these invariants to $\GLN_n(L)$ and thereby recover information about the Hodge filtration of $\rho_L$ from locally $\bQ_p$-analytic representations of $\GLN_n(L)$.\;In \cite{2019DINGSimple}, Yiwen Ding extends the theory of $\sL$-invariants to higher-rank groups.\;He adapts Breuil's approach (see \cite{BCLINV2}) to patched $p$-adic automorphic forms on certain definite unitary groups, defines Breuil's simple $\sL$-invariants, and proves that they agree with the Fontaine--Mazur $\sL$-invariants of two-dimensional subquotients of the trianguline $(\varphi,\Gamma)$-module associated with a semistable local $p$-adic Galois representation.\;

\textit{Although trianguline $p$-adic Galois representations have been studied extensively, substantially fewer results are available in the non-trianguline setting.}\;In this paper, we develop the theory of \textit{parabolic simple Fontaine--Mazur and Breuil $\sL$-invariants} for certain potentially semistable, non-crystalline Galois representations $\rho_L$ that need not be trianguline and whose associated smooth representations arise from Zelevinsky segments attached to cuspidal representations.\;Our work extends Ding's results \cite{2019DINGSimple} to the parabolic setting and provides further evidence for the $p$-adic Langlands correspondence beyond the trianguline case.\;Our consideration is motivated by the following observations:
\begin{itemize}
	\item[(a)] Zelevinsky segments play a central role in the Bernstein--Zelevinsky classification of irreducible smooth representations of $\GLN_n(L)$.\;It is therefore natural to study their analogues on the automorphic side of the $p$-adic Langlands correspondence.\;More specifically, we study locally $\bQ_p$-analytic generalized Steinberg representations of $\GLN_n(L)$ attached to a Zelevinsky segment and compute the $\ext^1$-groups between them; these constructions underlie the parabolic simple Breuil $\sL$-invariants in \cite{2022ext1hyq}.
	\item[(b)] We expect our parabolic simple $\sL$-invariants to appear in derivatives of $p$-adic adjoint $L$-functions of certain \textit{non-finite-slope} automorphic forms at trivial zeros.\;The exceptional Euler-type factors at places above $p$ suggest that standard $p$-adic $L$-functions may not be the appropriate objects here, unlike in the Borel case; see \cite{expectionalSteinberg} for recent work in the Borel--Steinberg setting.\;Our parabolic simple $\sL$-invariants are also closely related to the Greenberg--Benois $\sL$-invariants of adjoint-square Galois representations, leading to an explicit formula relevant to the adjoint trivial-zero conjecture; see Appendix II.\;This direction remains largely open.\;Hida studied the $\GLN_2$ case in \cite{hida2004greenberg}; see also Section \ref{introfurthercomment}.
	\item[(c)] Since non-trianguline $p$-adic Galois representations are not covered by the classical theory of eigenvarieties,\;we use the framework of \textit{Bernstein eigenvarieties} developed by Christophe Breuil and Yiwen Ding \cite{Ding2021}.\;Further arithmetic applications of Bernstein eigenvarieties are expected.\;For example, they may provide a framework for conjectural multivariable $p$-adic $L$-functions in families of non-finite-slope $p$-adic automorphic forms.
\end{itemize}

We define parabolic simple Fontaine--Mazur $\sL$-invariants for $\rho_L$ and establish local-global compatibility results using the geometry of Bernstein eigenvarieties.\;More precisely, we realize the correspondence between these invariants and the parabolic simple Breuil $\sL$-invariants introduced in \cite{2022ext1hyq} in spaces of $p$-adic automorphic forms on certain definite unitary groups.\;


Let $\bT$ (resp.,\;$\bB$) be the torus of diagonal matrices (resp.,\;the Borel subgroup of upper-triangular matrices) in $\GLN_n$.\;Let $\Delta:=\{1,\ldots,n-1\}$ be the set of simple roots of $\GLN_n$.\;For $I\subseteq\Delta$, let $\bL_I\supseteq\bT$ (resp.,\;$\bP_I\supseteq\bL_I\bB$) be the corresponding standard Levi subgroup (resp.,\;standard parabolic subgroup), and let $\op_I$ be the parabolic subgroup opposite to $\bP_I$.\;Fix positive integers $k$ and $r$ such that $n=kr$, and consider the partition $\beta=(r,\ldots,r)$ of $n$ with $k$ parts.\;Put $\Delta_\beta:=\{r,2r,\ldots,(k-1)r\}\subseteq\Delta$ and $\Delta^\beta:=\Delta\backslash\Delta_\beta$, and define
$\bL_{\beta,I}:=\bL_{\Delta^\beta\cup I}$ and $\bP_{\beta,I}:=\bP_{\Delta^\beta\cup I}$.\;Write $\bL_\beta:=\bL_{\beta,\emptyset}\cong (\GLN_r)^k$  and  $\bP_\beta:=\bP_{\beta,\emptyset}$ for simplicity.\;Fix a cuspidal Bernstein component $\Omega_r$ of $\GLN_r(L)$, and let $\omepik:=\prod_{i=1}^k\Omega_r$ be the corresponding cuspidal Bernstein component of $\bL_\beta(L)$.\;Let $\FZ_{\Omega_r}$ (resp.,\;$\FZ_{\omepik}$) be the (rational) Bernstein center of $\Omega_r$ (resp.,\;$\omepik$) over $E$.\;Then $\FZ_{\omepik}\cong\FZ_{\Omega_r}^{\otimes k}$.\;We fix an irreducible smooth cuspidal representation ${\pi}\in\Omega_r$ over $E$.\;

By the normalized classical local Langlands correspondence \cite{scholze2013local}, the representation ${\pi}$ corresponds to an irreducible Weil--Deligne representation $\br_0:=\rec({\pi})$ of $W_L$.\;Berger's theory \cite{berger2008equations} then yields an absolutely irreducible rank-$r$ $(\varphi,\Gamma)$-module $\Delta_0$ over $\cR_{E,L}$, the $p$-adic differential equation associated with $\br_0$.\;This module is de Rham with constant Hodge--Tate weight $0$. After forgetting the Hodge filtration, $D_{\pst}(\Delta_0)$ is isomorphic to the absolutely irreducible Deligne--Fontaine module attached to $\br_0$ by Fontaine \cite[Proposition 4.1]{breuil2007first}.\;For $\alpha\in E^\times$, write $\unr(\alpha)$ for the unramified character of $L^\times$ that sends a uniformizer to $\alpha$.\;

Let $\rho_L:\gal_L\rightarrow\GLN_n(E)$ be a potentially semistable, non-crystalline $p$-adic Galois representation, and let $\Dpik=D_{\rig}(\rho_L)$ be its associated rank-$n$ $(\varphi,\Gamma)$-module over $\cR_{E,L}$.\;Let $\bh:=(\hpi_{1,\tau}>\hpi_{2,\tau}>\cdots>\hpi_{n,\tau})_{\tau\in\Sigma_L}$ denote the Hodge--Tate weights of $\Dpik$, and put $\hpi_i=(\hpi_{i,\tau})_{\tau\in\Sigma_L}$ for $1\leq i\leq n$.\;Assume that $\Dpik$ admits a \textit{non-critical special $\omepik$-filtration} (Definition \ref{dfnnoncriticalspecial}), a parabolic analogue of a triangulation consisting of saturated $(\varphi,\Gamma)$-submodules
\[\cF=\fil_{\bullet}^{\cF} \Dpik: \ 0 =\fil_0^{\cF} \Dpik \subsetneq \fil_1^{\cF}\Dpik \subsetneq \cdots \subsetneq \fil_{k}^{\cF} \Dpik=\Dpik,\]
with the following properties (write $q_L:=|k_L|$, where $k_L$ is the residue field of $L$):
\begin{itemize}
	\item[(1)] For each $1\leq i\leq k$, there is an injection of rank-$r$ $(\varphi,\Gamma)$-modules over $\cR_{E,L}$
	\begin{equation}\label{introDpikinjection}
		\mathbf{I}_i:\gr_{i}^{\cF}\Dpik:=\fil_i^{\cF}\Dpik/\fil_{i-1}^{\cF}\Dpik \hookrightarrow \Delta_0\tee \cR_{E,L}(\unr(\alpha q_L^{i-k})z^{\hpi_{ir}}),
	\end{equation}
	for some $\alpha\in E^\times$, where $z^{\hpi_{ir}}:=\prod_{\tau\in  \Sigma_L}\tau(z)^{\hpi_{ir,\tau}}$.\;
	\item[(2)] The \textit{non-criticality} condition requires that the Hodge--Tate weights of $\fil_i^{\cF}\Dpik$ (resp.,\;$\gr_i^{\cF}\Dpik$) be
	\[\{\hpi_{1,\tau},\hpi_{2,\tau},\ldots,\hpi_{ir,\tau}\}_{\tau\in\Sigma_L}\quad (\text{resp.,}\;\{\hpi_{(i-1)r+1,\tau},\hpi_{(i-1)r+2,\tau},\ldots,\hpi_{ir,\tau}\}_{\tau\in\Sigma_L}).\]
	\item[(3)] The $F$-semisimple Weil--Deligne representation $\br_{\rho_L}^{\mathrm{F-ss}}$ is absolutely indecomposable (i.e.,\;$N^{k-1}\neq 0$).
\end{itemize}

\subsection{Parabolic Fontaine-Mazur simple $\sL$-invariants}
We associate with $\rho_L$ its parabolic simple Fontaine--Mazur $\sL$-invariants $\sL(\rho_L)$.\;These invariants capture information about the successive extensions $\Dpik_i^{i+1}$ of $\gr_{i+1}^{\cF}\Dpik$ by $\gr_i^{\cF}\Dpik$ inside $\Dpik$, for $1\leq i\leq k-1$.\;When $r=1$, this construction recovers the one in \cite[Page 7994]{2019DINGSimple}.\;

Let $\homo(L^\times,E)$ denote the $(d_L+1)$-dimensional $E$-vector space of $E$-valued additive characters of $L^\times$.\;Note that $\dim_E\ext^1_{(\varphi,\Gamma)}(\gr_{i+1}^{\cF}\Dpik,\gr_i^{\cF}\Dpik)=1+d_Lr^2$.\;The injections $\mathbf{I}_i$ and $\mathbf{I}_{i+1}$ in (\ref{introDpikinjection}) give rise to the following pairing:\;
\begin{pro}(Proposition \ref{pairingdef})
There is a pairing that is non-degenerate in the middle term:
\begin{equation}\label{introperfectpairing}
	\begin{aligned}
		\ext^1_{(\varphi,\Gamma)}(\gr_{i+1}^{\cF} \Dpik,\gr_{i}^{\cF}\Dpik)  \;\times\;   \hH^1_{(\varphi,\Gamma)}(\cR_{E,L}) \;\xrightarrow{\cup}\;
		E.
	\end{aligned}
\end{equation}
Note that $\hH^1_{(\varphi,\Gamma)}(\cR_{E,L}) \cong \homo(L^\times,E)$.\;
\end{pro}
\begin{dfn} (Parabolic Fontaine--Mazur simple $\sL$-invariants) For $1\leq i\leq k-1$, let
	\[\sL(\rho_L)_{ir}\subset \homo(L^\times,E)\]
	be the $E$-subspace orthogonal to $[\Dpik_i^{i+1}]\in\ext^1_{(\varphi,\Gamma)}(\gr_{i+1}^{\cF}\Dpik,\gr_i^{\cF}\Dpik)$ under (\ref{introperfectpairing}).\;This subspace has dimension $d_L$ over $E$.\;Set $\sL(\rho_L):=\prod_{i=1}^{k-1}\sL(\rho_L)_{ir}$.\;
\end{dfn}

\subsection{Parabolic simple Breuil $\sL$-invariants \cite{2022ext1hyq}}
Following \cite{av1980induced2}, consider the (normalized) Zelevinsky segment
\[\bm\beta:=\{{\pi}\nu^{k-1+\frac{1-n}{2}},\;\cdots,\;{\pi}\nu,{\pi}\nu^{\frac{1-n}{2}}\}.\]
where $\nu$ is the character of $\GLN_r(L)$ defined by $g\mapsto |\det(g)|_L$.\;Set ${\bm\lambda}:=(\hpi_{i,\tau}+i-1)_{\tau\in\Sigma_L,1\leq i\leq n}$; this is a dominant weight of $\mathrm{Res}_{L/\bQ_p}\GL_n$ with respect to $\mathrm{Res}_{L/\bQ_p}\bB$.\;Let $G:=\GL_n(L)$.\;To $\bm\beta$ we associate the following locally $\bQ_p$-analytic (resp., locally $\bQ_p$-algebraic) generalized Steinberg representations of $G$ (see Section 3):
\[\{v^{\ana}_{I}(\bm\beta,{\bm\lambda})\}_{I\subseteq\Delta_\beta}\quad (\text{resp.,\;}\{v^{\lalg}_{I}(\bm\beta,{\bm\lambda})\}_{I\subseteq\Delta_\beta}).\]
We write $\st_n^{\lalg}(\bm\beta,{\bm\lambda}):=v^{\lalg}_{\emptyset}(\bm\beta,{\bm\lambda})$ (resp., $\st_n^{\ana}(\bm\beta,{\bm\lambda}):=v^{\ana}_{\emptyset}(\bm\beta,{\bm\lambda})$) for the locally $\bQ_p$-algebraic (resp., locally $\bQ_p$-analytic) Steinberg representations.\;For $\alpha\in E^\times$, define $\st_n^{\ast}(\alpha,\bm\beta,{\bm\lambda}):=\unr(\alpha)\circ\det\otimes_E\st_n^{\ast}(\bm\beta,{\bm\lambda})$ and $v^{\ast}_{I}(\alpha,\bm\beta,{\bm\lambda}):=\unr(\alpha)\circ\det\otimes_Ev^{\ast}_{I}(\bm\beta,{\bm\lambda})$ for $\ast\in\{\lalg,\ana\}$.\;Let $\st_n^{\ana}(\alpha,\bm\beta,{\bm\lambda})^1$ denote the subrepresentation formed by the first two layers of the socle filtration of $\st_n^{\ana}(\alpha,\bm\beta,{\bm\lambda})$; Orlik--Strauch describe its structure explicitly in \cite[Theorem]{orlik2015jordan}.\;The main theorem of \cite{2022ext1hyq} and the discussion in Sections \ref{reviewlocallyanaST} and \ref{reviewEXTlocallyanaST} give, for each $1\leq i\leq k-1$, the following isomorphisms of $(d_L+1)$-dimensional $E$-vector spaces:
\begin{equation}\label{thmintro1}
	\homo(L^\times,E)\xrightarrow{\sim }\ext^1_{G}\big(v_{ir}^{\lalg}(\alpha,\bm\beta,{\bm\lambda}), \st_n^{\ana}(\alpha,\bm\beta,{\bm\lambda})^1\big)\xrightarrow{\sim }\ext^1_{G}\big(v_{ir}^{\lalg}(\alpha,\bm\beta,{\bm\lambda}), \st_n^{\ana}(\alpha,\bm\beta,{\bm\lambda})\big).
\end{equation}
Let $\Sigma(\alpha,\bm\beta,{\bm\lambda},\sL(\rho_L))$ be the  extension of $\bigoplus_{1\leq i\leq k-1}v_{ir}^{\lalg}(\alpha,\bm\beta,{\bm\lambda})\otimes_E\sL(\rho_L)_{ir}$ by $\st_n^{\ana}(\alpha',\bm\beta,{\bm\lambda})^1$.\;The parabolic simple Breuil $\sL$-invariants contain counterparts of the parabolic simple Fontaine--Mazur $\sL$-invariants using the following rule:
\[\rho_L \rightsquigarrow \sL(\rho_L)\rightsquigarrow \Sigma(\alpha,\bm\beta,{\bm\lambda},\sL(\rho_L)).\]

\subsection{Local-global compatibility} We establish new local-global compatibility results in the $p$-adic Langlands program by analyzing the geometry of patched Bernstein eigenvarieties and Bernstein paraboline varieties \cite{Ding2021}.\;Under the standard Taylor--Wiles hypotheses, the construction in \cite{PATCHING2016} produces a continuous Banach representation $\Pi_\infty$ of $G$, equipped with a continuous action of the patched Galois deformation ring $R_\infty$ that commutes with the $G$-action.\;In Section \ref{BENVARPARAVAR}, we adapt the arguments of \cite[Sections 3.3 and 4.2]{Ding2021} to our setting.\;We construct a patched Bernstein eigenvariety $\bersteineigenvarpik$ associated with $\omepik$ and the weight ${\bm\lambda}$, where $\overline{\rho}$ is a suitable global residual Galois representation.\;If $\rho_L$ corresponds to a maximal ideal $\fm_\rho$ of $R_\infty[1/p]$ and lies on $\bersteineigenvarpik$ (see Section \ref{appearpoint}), we associate with it the admissible unitary Banach representation $\widehat{\Pi}(\rho_L):=\Pi_\infty[\fm_\rho]$ of $G$.\;A parabolic analogue of global triangulation theory (see Lemma \ref{detlameda}) then yields an injection of $G$-representations
\[\st_n^{\infty}(\alpha,\bm\beta,{\bm\lambda})\hookrightarrow \widehat{\Pi}(\rho_L).\]
The locally $\bQ_p$-analytic representation $\Sigma(\alpha,\bm\beta,{\bm\lambda},\sL(\rho_L))$ therefore encodes the Weil--Deligne representation $\wdre_{\rho_L}$, the Hodge--Tate weights of $\rho_L$, and the parabolic simple $\sL$-invariants $\sL(\rho_L)$.\;Our main result is the following (see Theorem \ref{thm: lgln-main}).
\begin{thm}\label{introthm: lgln-main} The injection $\st_n^{\infty}(\alpha,\bm\beta,{\bm\lambda})\hookrightarrow \widehat{\Pi}(\rho_L)$ extends uniquely to an injection of $G$-representations
	\[\Sigma(\alpha,\bm\beta,{\bm\lambda}, \sL(\rho_L))\hookrightarrow \widehat{\Pi}(\rho_L).\]
	Moreover, $\sL(\rho_L)$ can be recovered from $\widehat{\Pi}(\rho_L)$.\;

\end{thm}

We outline the proof, following the strategy in \cite[pp.\;7996--7997]{2019DINGSimple}.\;It is enough to construct an injection $\Sigma_{ir}(\alpha,\bm\beta,{\bm\lambda},\psi)\hookrightarrow\widehat{\Pi}(\rho_L)$ for each $\psi\in\sL(\rho_L)_{ir}$.\;The patched Bernstein eigenvariety $\bersteineigenvarpik\hooklongrightarrow(\Spf\;R_\infty)^{\rig}\times\sbanpik\times\rigch$ is constructed from the Jacquet--Emerton module $J_{\bP_\beta}(\Pi_\infty^{R_\infty-\ana})$, where $\Pi_\infty^{R_\infty-\ana}$ denotes the subspace of locally $R_\infty$-analytic vectors in $\Pi_\infty$ (see \cite[Section 3.1]{breuil2017interpretation}).\;Here, $\rigch$ denotes the rigid analytic space over $E$ parameterizing continuous characters of $\bZ_\beta(\cO_L)$, where $\bZ_\beta$ is the center of $\bL_\beta$.\;The representation $\rho_L$ determines a point $x:=(\rho,{\pi}_{\beta,x},\chi)\in\bersteineigenvarpik$.\;Let $\omega:\bersteineigenvarpik\rightarrow\sbanpik\times\rigch$ be the weight map, and consider its tangent map at $x$:
\begin{equation}
	d\omega_x:T_{\bersteineigenvarpik,x}\longrightarrow T_{\sbanpik\times\rigch,\omega(x)}\cong\homo(\bZ_\beta(L),E)=\prod_{i=1}^k\homo(L^\times,E).\;
\end{equation}
To analyze the local geometry of $\bersteineigenvarpik$ at $x$, we use the paraboline deformation variety $\defvarrho$ of \cite{Ding2021}, which serves as its local analogue.\;Define
$\kappa_L:\homo(\bZ_\beta(L),E)\rightarrow\prod_{i=1}^{k-1}\homo(L^\times,E)$ by sending $(\psi_1,\psi_2,\ldots,\psi_k)$ to $(\psi_i-\psi_{i+1})_{1\leq i\leq k-1}$.\;

\begin{pro}\label{introcoreigentangentmap2}$\bersteineigenvarpik$ is smooth at $x$.\;We have a surjective map  $\kappa_L\circ d\omega_x:	T_{\bersteineigenvarpik,x}\twoheadlongrightarrow \sL(\rho_L)$.\;
\end{pro}

It follows that there is some $\Psi\in\mathrm{Im}(d\omega_x)$ with $\kappa_L(\Psi)=\psi$, which is the key step in the proof.\;The construction of $\bersteineigenvarpik$ then yields an injection of $\bL_\beta(L)$-representations
\[\pi_{\bm\beta}\te(\mathrm{1}_\Psi\circ\mathrm{det}_{\bL_{\beta}(L)})\te\lamoverx_{\Delta^{\beta}} \hooklongrightarrow J_{\bP_{\beta}}\big(\Pi_{\infty}^{R_{\infty}-\ana}[\fm_x^\infty]\big),\]
Here, $[\fm_x^\infty]$ denotes the generalized $\fm_x$-eigenspace, $\pi_{\bm\beta}$ is the normalized irreducible smooth cuspidal representation of $\bL_\beta(L)$ associated with the Zelevinsky segment $\bm\beta$, $\mathrm{1}_\Psi$ is the self-extension of the trivial character attached to $\Psi$, and $\lamoverx_{\Delta^{\beta}}$ is the irreducible $\bQ_p$-algebraic representation of $\bL_\beta(L)$ with highest weight ${\bm\lambda}$.\;The non-criticality hypothesis on $\rho_L$ implies that this injection is balanced in the sense of \cite[Definition 0.8]{emerton2007jacquet}.\;Emerton's adjunction formula \cite[Theorem 0.13]{emerton2007jacquet} therefore gives a nonzero morphism
\[I^G_{\op_{\beta}(L)}\Big(\pi_{\bm\beta}\te(\mathrm{1}_\Psi\circ\mathrm{det}_{\bL_{\beta}(L)})\te\lamoverx_{\Delta^{\beta}}\Big) \longrightarrow\Pi_{\infty}^{R_{\infty}-\ana}[\fm_x^\infty].\]
This morphism, in turn, induces the desired injection $\Sigma_{ir}(\alpha,\bm\beta,{\bm\lambda},\psi)\hookrightarrow\widehat{\Pi}(\rho_L)$.\;

\subsection{Further comments and developments}\label{introfurthercomment}
In \cite{AutoLinv}, Gehrmann developed a theory of automorphic $\sL$-invariants for a cuspidal automorphic representation $\pi$ of a semi-simple reductive group $\bG$ over  a number field $F$,\;which is cohomological with respect to the trivial coefficient system and $\pi_\fp$ is the Borel-Steinberg representation of $\bG(F_{\fp})$  at one $p$-place $\fp$.\;More generally, the results of \cite{expectionalSteinberg} suggest that derivatives of $p$-adic $L$-functions at trivial zeros are more directly related to automorphic $\sL$-invariants than to simple Fontaine--Mazur $\sL$-invariants.\;We expect a parabolic analogue of Gehrmann's construction for a cuspidal automorphic representations of $\bG$ that are (parabolic)-Steinberg (associated to a  general Zelevinsky segment) at $\fp$.\;We also expect the parabolic automorphic $\sL$-invariants to agree with the parabolic simple Fontaine--Mazur $\sL$-invariants.\;See \cite{ShalikaLFun,liuDongwen} for constructions of nearly ordinary standard $p$-adic $L$-functions and studies of their exceptional zeros for automorphic representations of symplectic type on $\GLN_{2n}$, or equivalently, representations admitting Shalika models.\;We may ask if these methods may also apply to the Borel-Steinberg and (parabolic)-Steinberg cases.\;

\subsection{Acknowledgments}
This work forms part of the author's PhD thesis.\;The author is deeply grateful to his advisor, Yiwen Ding, for introducing him to this subject, for many helpful discussions and suggestions, and for commenting on earlier drafts of this paper.\;

\subsection{General notation}\label{notationger}

Let $L$ and $E$ be finite extensions of $\bQ_p$, with rings of integers $\co_L$ and $\co_E$ and uniformizers $\varpi_L$ and $\varpi_E$, respectively.\;Assume that $E$ is sufficiently large to contain the images of all embeddings of $L$ into $\overline{\bQ}_p$.\;Set $\Sigma_L:=\{\sigma:L\hookrightarrow\overline{\bQ}_p\}=\{\sigma:L\hookrightarrow E\}$.\;Let $k_E$ be the residue field of $E$, and set $d_L:=[L:\bQ_p]=|\Sigma_L|$.\;Let $e_L$ and $f_L$ denote the ramification index and residue degree of $L/\bQ_p$, respectively; thus, $d_L=e_Lf_L$ and $q_L:=p^{f_L}=|\co_L/\varpi_L\co_L|$.\;Let $\val_L:L^\times\rightarrow\BZ$ be the normalized valuation with $\val_L(\varpi_L)=1$.\;


Let $X$ be either a scheme locally of finite type over $E$ or a locally noetherian formal scheme over $\cO_E$ whose reduction is locally of finite type over $k_E$.\;Write $X^{\mathrm{rig}}$ for its associated rigid analytic space over $E$.\;If $X$ is a scheme locally of finite type over $E$ or a rigid analytic space over $E$, let $X^{\red}$ denote the associated reduced Zariski-closed subspace.\;For a point $x\in X$, write $\cO_{X,x}$ and $k(x)$ for its local ring and residue field, respectively.\;Let $\widehat{\cO}_{X,x}$ be the $\fm_{\cO_{X,x}}$-adic completion of $\cO_{X,x}$, and set $\widehat{X}_x:=\mathrm{Spf}\;\widehat{\cO}_{X,x}$.\;If $x$ is closed, then $\widehat{\cO}_{X,x}$ is a noetherian complete local $k(x)$-algebra with residue field $k(x)$.\;

For a Lie algebra $\fh$ over $L$ and $\sigma\in\Sigma_L$, define the $E$-Lie algebra $\fh_\sigma:=\fh\otimes_{L,\sigma}E$.\;For each $J\subseteq\Sigma_L$, define $\fh_J:=\prod_{\sigma\in J}\fh_\sigma$.\;In particular, $\fh_{L}:=\fh_{\Sigma_{L}}\cong\fh\otimes_{\bQ_p}E$.\;We write $U(\fh_L)$ for the universal enveloping algebra of $\fh_L$ over $E$.\;

	Let $\GL_n$ be the general linear group over $L$, and let $\Delta:=\Delta_n$ be its set of simple roots with respect to the Borel subgroup $\bB$ of upper-triangular matrices.\;We identify $\Delta$ with $\{1,\ldots,n-1\}$ so that $i$ corresponds to the simple root $\alpha_i:(x_1,\ldots,x_n)\mapsto x_i-x_{i+1}$, where $\ft$ is the $L$-Lie algebra of the diagonal torus $\bT$.\;For each $I\subseteq\Delta$, let $\bL_I$ be the standard Levi subgroup containing $\bT$ whose set of simple roots is $I$, and let $\bP_I:=\bL_I\bB$ be the corresponding standard parabolic subgroup.\;Thus $\bP_\Delta=\GL_n$ and $\bP_\emptyset=\bB$.\;Let $\overline{\bP}_I$ (resp.,\;$\ob$) be the parabolic (resp.,\;Borel) subgroup opposite to $\bP_I$ (resp.,\;$\bB$), and let $\bN_I$ (resp.,\;$\overline{\bN}_I$) be the unipotent radical of $\bP_I$ (resp.,\;$\overline{\bP}_I$).\;Then $\bP_I=\bL_I\bN_I$ and $\overline{\bP}_I=\bL_I\overline{\bN}_I$, and $\Delta\backslash I$ is the set of simple roots occurring in $\bN_I$.\;Let $\bZ$ (resp.,\;$\bZ_I$) be the center of $\GL_n$ (resp.,\;$\bL_I$).\;We denote the $L$-Lie algebras of $\GL_n$, $\bB$, $\ob$, $\bP_I$, $\op_I$, $\bN_I$, $\bL_I$, $\overline{\bN}_I$, $\bZ_I$, and $\bT$ by $\fg$, $\fb$, $\overline{\fb}$, $\fp_I$, $\overline{\fp}_I$, $\fn_I$, $\fl_I$, $\overline{\fn}_I$, $\fz_I$, and $\ft$, respectively.\;Put $G:=\GL_n(L)$.\;

Let $\sW_n\cong S_n$ be the Weyl group of $\GLN_n$, and let $s_i$ be the simple reflection corresponding to $i\in\Delta$.\;For $I\subseteq\Delta$, let $\sW_I$ be the subgroup generated by $\{s_i:i\in I\}$.\;Put $\sW_{n,L}:=\prod_{\sigma\in\Sigma_L}\sW_{n,\sigma}$, where $\sW_{n,\sigma}\cong\sW_n$ is the $\sigma$-copy of $\sW_n$.\;For $i\in\Delta$ and $\sigma\in\Sigma_L$, let $s_{i,\sigma}\in\sW_{n,\sigma}$ denote the corresponding simple reflection.\;

Let $\cR_L:=\bB_{\rig,L}^{\dagger}$ denote the Robba ring associated with $L$.\;If $A$ is a $\bQ_p$-affinoid algebra, set $\cR_{A,L}:=\cR_L\widehat{\otimes}_{\bQ_p}A$; if $X$ is a rigid analytic space, write $\cR_{X,L}$ for the corresponding Robba ring with coefficients in $\cO_X$ (see \cite[Definition 6.2.1]{XLROBBA}).\;For a continuous character $\delta_A:L^\times\rightarrow A^\times$ (resp.,\;$\delta_X:L^\times\rightarrow \Gamma(X,\cO_X)^\times$), let $\cR_{A,L}(\delta_A)$ (resp.,\;$\cR_{X,L}(\delta_X)$) denote the associated rank-one $(\varphi,\Gamma)$-module over $\cR_{A,L}$ (resp.,\;$\cR_{X,L}$).\;For $\bullet\in \{A,X\}$ and a  $(\varphi,\Gamma)$-module $D$ over $\cR_{\bullet,L}$,\;write $D(\delta_{\bullet}):=D\otimes_{\cR_{\bullet,L}}\cR_{\bullet,L}(\delta_{\bullet})$ for simplicity.\;

Let $m\geq1$, and let $\pi$ be an irreducible smooth admissible representation of $\GLN_m(L)$.\;Write $\rec(\pi)$ for the $F$-semisimple Weil--Deligne representation attached to $\pi$ by the normalized local Langlands correspondence (see \cite{scholze2013local}).\;We normalize the reciprocity isomorphism $\rec:L^\times\rightarrow W_L^{\mathrm{ab}}$ so that a uniformizer $\varpi_L$ maps to geometric Frobenius; here, $W_L^{\mathrm{ab}}$ is the abelianization of the Weil group $W_L\subset\gal_L$.\;

For $\alpha\in E^\times$, write $\unr(\alpha)$ for the unramified character of $L^\times$ that sends $\varpi_L$ to $\alpha$.\;For a character $\chi$ of $\cO_L^\times$, denote by $\chi_{\varpi_L}$ the character of $L^\times$ characterized by $\chi_{\varpi_L}|_{\cO_L^\times}=\chi$ and $\chi_{\varpi_L}(\varpi_L)=1$.\;For a character $\delta$ of $L^\times$, set $\delta_0:=\delta|_{\cO_L^\times}$.\;For $\bk=(\bk_\tau)_{\tau\in\Sigma_L}\in\BZ^{\Sigma_L}$, define $z^\bk:=\prod_{\tau\in\Sigma_L}\tau(z)^{\bk_\tau}$.\;Let $\ccyc:\gal_L\rightarrow\bZ_p^\times$ be the $p$-adic cyclotomic character, characterized by $g(\epsilon_n)=\epsilon_n^{\ccyc(g)}$.\;Local class field theory gives $\ccyc\circ\rec=\unr(q_L^{-1})z^{\mathbf{1}}$, where $\mathbf{1}=(1)_{\tau\in\Sigma_L}$.\;We use the convention that the Hodge--Tate weights of a de Rham $p$-adic representation of $\gal_L$ are the negatives of the jumps in the filtration on the covariant de Rham functor; thus, the cyclotomic character has Hodge--Tate weight $1$.\;

For a topological commutative group $M$, we use $\homo(M,E)$ (resp.\;$\homo_{\sm}(M,E)$) to denote the $E$-space of continuous (resp.\;locally constant) additive $E$-valued characters on $M$.\;If $M$ is a locally $L$-analytic group, denote by $\homo_\sigma(M,E)$ the $E$-vector space of locally $\sigma$-analytic characters on $M$ (i.e., the continuous characters which are locally $\sigma$-analytic $E$-valued functions on $M$).\;We obtain  a bijection $\homo(L^\times,E)\cong \homo(\gal_L,E)$ by the local class theory.\;

 Let $\bG$ be a split connected reductive group over $L$, and let $\bP$ be a parabolic subgroup of $\bG$.\;For a continuous representation $V$ of $\bG(L)$ over $E$, write $V^{\bQ_p-\ana}$ for its subspace of locally $\bQ_p$-analytic vectors.\;If $V$ is locally $\bQ_p$-analytic, write  $V^{\mathrm{lalg}}$ for its locally $\bQ_p$-algebraic subrepresentation; see \cite{schneider2002banach,Emerton2007summary}.\;For a continuous (resp., locally $\bQ_p$-analytic, resp., smooth) representation $\pi'$ of $\bP(L)$, define the corresponding unnormalized continuous (resp., locally $\bQ_p$-analytic, resp., smooth) parabolic inductions as follows:
\begin{equation}\label{smoothadj2}
\begin{aligned}
	&(\mathrm{Ind}_{\bP(L)}^{\bG(L)}\pi')^{\cC^0} (\text{resp.,\;}(\mathrm{Ind}_{\bP(L)}^{\bG(L)}\pi')^{\bQ_p-\ana},\;\text{resp.,\;}(\mathrm{Ind}_{\bP(L)}^{\bG(L)}\pi')^{\infty})\\
	:=&\;\{f:\bG(L)\rightarrow \pi' \text{\;continuous} \;(\text{resp.,\;locally $\bQ_p$-analytic,\;resp.,\;smooth}),\;f(pg)=p\cdot f(g),\ \forall p\in \bP(L)\}.
\end{aligned}
\end{equation}
The group $\bG(L)$ acts by right translation: $(g\cdot f)(g')=f(g'g)$.\;

Throughout the paper, we use the notation $\bullet-\bullet$ for an extension of two objects (such as $(\varphi,\Gamma)$-modules or $G$-representations), with the first object as the subobject and the second as the quotient.\;

\section{Parabolic Fontaine--Mazur simple {$\sL$}-invariants}\label{sectionparasimlinv}

In Section \ref{noncrispecfildfn}, we study potentially semistable $p$-adic Galois representations $\rho_L$ that admit a non-critical special $\omepik$-filtration associated with a Zelevinsky segment.\;In Section \ref{dfnFMparalINVSECTION}, we define the corresponding parabolic simple Fontaine--Mazur $\sL$-invariants, which encode information about successive extension subquotients of the associated $(\varphi,\Gamma)$-module over $\cR_{E,L}$.\;In Sections \ref{dfndeftyoeomega} and \ref{cgsformulaSection}, we explain how to detect these invariants in first-order deformations, yielding the parabolic Colmez--Greenberg--Stevens formula.\;

\subsection{Preliminaries}\label{BPFORPAPAIINV}

In this subsection, we recall the notion of an $\Omega$-filtration from \cite{Ding2021}.\;Let $D$ be a $(\varphi,\Gamma)$-module over $\cR_{E,L}$.\;It is called trianguline if it is a successive extension of rank-one $(\varphi,\Gamma)$-modules.\;Breuil and Ding introduced the more flexible notion of an $\Omega$-filtration,\;where $\Omega$ is a Bernstein component \cite[Definition 4.1.6]{Ding2021}.\;

Fix a cuspidal Bernstein component $\Omega_r$ of $\GLN_r(L)$, and let $\mathfrak{Z}_{\Omega_r}$ be its (rational) Bernstein center over $E$,\;here we assume that $E$ is sufficiently large for $\mathfrak{Z}_{\Omega_r}$ to be defined over $E$; see \cite[Section 3.13]{PATCHING2016}.\;Throughout this paper,\;we fix an irreducible smooth representation ${\pi}\in\Omega_r$ of $\GLN_r(L)$ over $E$.\;Then $	\eta_{\Omega_r}:=\{a\in\EX:\pi_{\Omega_r}\cong \pi_{\Omega_r}\otimes_{E}\unra(a)\circ\det\}\cong \mu_{r_0}$ for an integer $r_0|r$,\;where $\mu_{r_0}$ is the group of $r_0$-th roots of unity in $\obp$.\;By \cite[pp. 242--243 and Lemma 3.24]{PATCHING2016} and \cite[Section 2.2]{Ding2021},\;$\mathfrak{Z}_{\Omega_r}\cong E[z,z^{-1}]^{\mu_{r_0}}\cong E[z^{r_0},z^{-r_0}]$.\;We further assume that $\mu_{r_0}\subset E$.\;

To each closed point $x\in\Spec\hspace{2pt}\FZ_{\Omega_r}$, we associate an irreducible smooth cuspidal representation $\pi_x$ of $\GLN_r(L)$ over $k(x)$.\;The normalized classical local Langlands correspondence gives an $F$-semisimple Weil--Deligne representation $\br_x:=\rec(\pi_x)$ of $W_L$ over $k(x)$ \cite{scholze2013local}.\;Under Fontaine's equivalence \cite[Proposition 4.1]{breuil2007first}, $\br_x$ corresponds to a Deligne--Fontaine module $\df_x$; Berger's theory \cite{berger2008equations} then associates with $\df_x$ a rank-$r$ $p$-adic differential equation $\Delta_x$ over $\cR_{E,L}$.\;This module is de Rham with constant Hodge--Tate weight $0$, and $D_{\pst}(\Delta_x)$, after forgetting the Hodge filtration, is isomorphic to $\df_x$.\;If $L'/L$ is finite and $\br_x|_{W_{L'}}$ is unramified, then $\df_x$ is naturally a $(\varphi,N,\gal(L'/L))$-module.\;

We now fix notation for Zelevinsky segments.\;Let $k,r\geq1$ with $n=kr$, and let $\beta=(r,\ldots,r)$ be the partition of $n$ with $k$ parts.\;Put $\Delta_\beta:=\{r,2r,\ldots,(k-1)r\}\subseteq\Delta$ and $\Delta^\beta:=\Delta\backslash\Delta_\beta$.\;For a subset $I\subset \Delta_{\beta}$, we use the shortened notation $S^{\beta}(I):=\Delta^{\beta}\cup I$ and 
\begin{equation}\label{lrropbL}
	\begin{aligned}
		&\bL_{\beta,I}:=\bL_{S^{\beta}(I)}, \bP_{\beta,I}:=\bP_{S^{\beta}(I)}, \op_{\beta,I}:=\op_{S^{\beta}(I)},\bN_{\beta,I}:=\bN_{S^{\beta}(I)}, \on_{\beta,I}:=\on_{S^{\beta}(I)}, \bZ_{\beta,I}:=\bZ_{S^{\beta}(I)}.
	\end{aligned}
\end{equation}
Let $\fl_{\beta,I}$, $\overline{\fl}_{\beta,I}$, $\fp_{\beta,I}$, $\fn_{\beta,I}$, $\overline{\fp}_{\beta,I}$, $\overline{\fn}_{\beta,I}$, $\fz_{\beta,I}$, $\overline{\fz}_{\beta,I}$ be the $L$-Lie algebras of $\bL_{\beta,I}$, $\bL_{\beta,I}/\bZ$, $\bP_{\beta,I}$, $\bN_{\beta,I}$, $\op_{\beta,I}$, $\on_{\beta,I}$ and $\bZ_{\beta,I}$, $\bZ_{\beta,I}/\bZ$ respectively.\;We omit $I$ when $I=\emptyset$.\;For example,
\[\bL_{\beta}:=\left(\begin{array}{cccc}
	\GLN_r & 0 & \cdots & 0 \\
	0 & \GLN_r & \cdots & 0 \\
	\vdots & \vdots & \ddots & 0 \\
	0 & 0 & 0 & \GLN_r \\
\end{array}\right),\quad\op_{\beta}:=\left(\begin{array}{cccc}
	\GLN_r & 0 & \cdots & 0 \\
	\ast & \GLN_r & \cdots & 0 \\
	\vdots & \vdots & \ddots & 0 \\
	\ast & \ast & \cdots & \GLN_r \\
\end{array}\right)\]
The parabolic subgroups of $\GLN_n$ that contain $\op_\beta$ are exactly $\{\op_{\beta,I}\}_{I\subseteq\Delta_\beta}$.\;Set
\[\omepik=\prod_{i=1}^k\Omega_{r,i},\quad\Omega_{r,i}=\Omega_{r},\]
which is a cuspidal Bernstein component of $\bL_\beta(L)$.\;Its (rational) Bernstein center is $\FZ_{\omepik}\cong\bigotimes_{i=1}^k\FZ_{\Omega_{r,i}}\cong \FZ_{\Omega_r}^{\otimes k}$.\;A closed point $\underline{x}=(x_i)_{1\leq i\leq k}\in \sbanpik$ determines irreducible smooth cuspidal representations $(\pi_{x_i})_{1\leq i\leq k}$ of $\GLN_r(L)$, $F$-semisimple Weil--Deligne representations $(\br_{x_i})_{1\leq i\leq k}$, and $p$-adic differential equations $(\Delta_{x_i})_{1\leq i\leq k}$. It therefore determines an irreducible smooth cuspidal representation $\pi_{\underline{x}}:=\boxtimes_{i=1}^k\pi_{x_i}$ of $\bL_\beta(L)$ and an $n$-dimensional $F$-semisimple Weil--Deligne representation $\br_{\underline{x}}:=\bigoplus_{i=1}^k\br_{x_i}$ of $W_L$.\;


We recall the definition of an $\omepik$-filtration on a $(\varphi,\Gamma)$-module; see \cite[Section 4.1.2]{Ding2021}.\;Recall that  $\bZ_{\beta}$ is the center of $\bL_\beta$.\;The embedding $\BZ\hookrightarrow L^\times$, $i\mapsto\varpi_L^i$, induces an embedding $\bigoplus_{i=1}^k\BZ\hookrightarrow\bZ_{\beta}(L)$.\;Write $Z_{\beta,\varpi_L}$ for its image, so that $\bZ_{\beta}(L)\cong Z_{\beta,\varpi_L}\times\bZ_{\beta}(\cO_L)$.\;Let $\rigerch:=\widehat{\bZ_{\beta}(\cO_L)}$, $\rigerchl:=\widehat{\bZ_{\beta}(L)}$, and $\widehat{Z_{\beta,\varpi_L}}$ denote the rigid analytic spaces over $E$ parameterizing continuous characters of the indicated groups.\;Thus, $\rigerchl\cong\widehat{Z_{\beta,\varpi_L}}\times\rigerch$.\;We now give the following definition, following \cite[Section 4.1.2 and Definition 4.1.6]{Ding2021}.
\begin{dfn}\label{dfnomegafil}An $\omepik$-filtration on a rank-$n$ $(\varphi,\Gamma)$-module $D$ over $\cR_{E,L}$ is an increasing $\bP_{\beta}$-parabolic filtration by saturated $(\varphi,\Gamma)$-submodules
\[\cF=\fil_{\bullet}^{\cF} D: \ 0 =\fil_0^{\cF} D \subsetneq \fil_1^{\cF} D \subsetneq \cdots \subsetneq \fil_{k}^{\cF} D=D,\]
such that, for $1\leq i\leq k$,
\begin{itemize}
	\item $\gr_i^{\cF}D$ has rank $r$;
	\item there is a pair consisting of an $E$-point $x_i\in \sbanpik$ and a continuous character $\delta_i:L^\times\rightarrow E^\times$ such that
	$\gr^{\cF}_{i}D\tee \cR_{E,L}(\delta_{i}^{-1})\hookrightarrow \Delta_{x_i}$.\;
\end{itemize}
Set $\underline{x}:=(x_i)_{1\leq i\leq k}$ and $\delta:=\otimes_{i=1}^k\delta_i$.\;
\begin{itemize}
	\item We call $(\underline{x},\delta)\in\sbanpik\times\rigerchl$ a parameter of $\cF$ provided that, for every $i$ and $\tau\in\Sigma_L$, the minimal $\tau$-Hodge--Tate weight of $\gr_i^{\cF}D\tee\cR_{E,L}(\delta_i^{-1})$ is $0$.
	\item We call $(\underline{x},\delta^0)\in\sbanpik\times\rigerch$ a parameter of $\cF$ if, writing $\delta^0=\otimes_{i=1}^k\delta_i^0$, the pair $(\underline{x},(\delta^0)_{\varpi_L})$ is a parameter in $\Spec(\FZ_{\omepik})\times\rigerchl$, where $(\delta^0)_{\varpi_L}:=\otimes_{i=1}^k(\delta_i^0)_{\varpi_L}$.
\end{itemize}
\end{dfn}

\subsection{{Non-critical special $\omepik$-filtration}}\label{noncrispecfildfn}

We now consider a special class of $\omepik$-filtrations, which we call \textit{non-critical special $\omepik$-filtrations}.\;For $x\in\Spec\hspace{2pt}\FZ_{\Omega_r}(E)$, choose $\alpha_x\in E^\times$ such that $\pi_x\cong{\pi}\otimes_E\unr(\alpha_x)\circ\det$; the scalar $\alpha_x$ is unique up to multiplication by an element of $\eta_{\Omega_r}\cong \mu_{r_0}$.\;

Let $\Dpik$ be a potentially semistable rank-$n$ $(\varphi,\Gamma)$-module over $\cR_{E,L}$.\;Choose a finite Galois extension $L'/L$ over which $\Dpik$ is semistable.\;The associated Deligne--Fontaine module is
\[\df_{\Dpik}=(\varphi,N,\gal(L'/L),D_{\pst}(\Dpik))\]
where $D_{\pst}(\Dpik):=D_{\mathrm{st}}^{L'}(\Dpik\otimes_{\cR_{E,L}}\cR_{E,L'})$ is a finite free $L'_0\otimes_{\bQ_p}E$-module of rank $n$, and $L'_0$ denotes the maximal unramified subextension of $L'/\bQ_p$.\;Its $(\varphi,N)$-action is induced by the action on $B_{\mathrm{st}}$, while its $\gal(L'/L)$-action comes from the residual action of $\gal_L$.\;Fontaine's equivalence of categories \cite[Proposition 4.1]{breuil2007first} associates with $\df_{\Dpik}$ an $n$-dimensional Weil--Deligne representation $\br_{\Dpik}$ of $W_L$ over $E$.\;

We say that $\br_{\Dpik}$ admits a special $\omepik$-filtration $\cF$ if it is equipped with an increasing filtration by Weil--Deligne subrepresentations
\[\cF=\fil_{\bullet}^{\cF} \br_{\Dpik}: \ 0 =\fil_0^{\cF} \br_{\Dpik} \subsetneq \fil_1^{\cF} \br_{\Dpik} \subsetneq \cdots \subsetneq \fil_{k}^{\cF} \br_{\Dpik}=\br_{\Dpik},\]
with the following properties:
\begin{itemize}
\item[(i)] For each $1\leq i\leq k$, there exists $x\in\Spec\hspace{2pt}\FZ_{\Omega_r}(E)$ such that $\gr_i^{\cF}\br_{\Dpik}\cong\br_x\otimes_E|\rec^{-1}|^{k-i}$;
\item[(ii)] For $2\leq i\leq k$, the monodromy operator $N$ maps $\gr^{\cF}_i\br_{\Dpik}$ to $\gr^{\cF}_{i-1}\br_{\Dpik}$ by either the zero map or the identity on $\br_x$; it maps $\gr^{\cF}_1\br_{\Dpik}$ to zero.\;
\end{itemize}
Under Fontaine's equivalence \cite[Proposition 4.1]{breuil2007first}, the filtration $\cF$ on $\br_{\Dpik}$ corresponds to a special $\omepik$-filtration of $\df_{\Dpik}$ by Deligne--Fontaine submodules
\[\cF=\fil_{\bullet}^{\cF}\df_{\Dpik}: \ 0 =\fil_0^{\cF} \df_{\Dpik} \subsetneq \fil_1^{\cF} \df_{\Dpik} \subsetneq \cdots \subsetneq \fil_{k}^{\cF} \df_{\Dpik}=\df_{\Dpik},\]
such that $\fil_i^{\cF}\df_{\Dpik}$ corresponds to $\fil_i^{\cF}\br_{\Dpik}$ and the graded piece is   $\gr_i^{\cF}\df_{\Dpik}\cong(p^{i-k}\varphi_x,0,\gal(L'/L),\df_x)$
for each $1\leq i\leq k$.\;In particular,\;we have 
 \[\df^{\mathrm{ss}}_{\Dpik}=\bigoplus_{i=1}^k(p^{i-k}\varphi_x,0,\gal(L'/L),\df_{x}).\]
The operator $N$ maps $(p^{i-k}\varphi_x,0,\gal(L'/L),\df_x)$ to $(p^{i-1-k}\varphi_x,0,\gal(L'/L),\df_x)$ by either the zero map or the identity on $\df_x$, for $2\leq i\leq k$, and vanishes on $(p^{1-k}\varphi_x,0,\gal(L'/L),\df_x)$.\;In particular, $N^k=0$.\;$\df_{\Dpik}$ is absolutely indecomposable if and only if $N$ has maximal monodromy rank, i.e.\;$N^{k-1}\neq0$.\;

Let $\Delta_\Dpik$ denote the $p$-adic differential equation over $\cR_{E,L}$ associated with $\df_\Dpik$.\;The special $\omepik$-filtration on $\df_\Dpik$ induces  a special $\omepik$-filtration on $\fil_\bullet^{\cF}\Delta_\Dpik$ by saturated $(\varphi,\Gamma)$-submodules, with $\fil_i^{\cF}\Delta_\Dpik$ corresponding to $\fil_i^{\cF}\df_\Dpik$.\;In particular, $\gr_i^{\cF}\Delta_\Dpik\cong \Delta_x(\unr(q_L^{i-k}))=\Delta_x\tee\cR_{E,L}(\unr(q_L^{i-k}))$ for $1\leq i\leq k$.\;Since $\cM_\Dpik:=\Dpik[1/t]\cong\Delta_\Dpik[1/t]$, localizing at $t$ yields a filtration $\fil_i^{\cF}\cM_\Dpik:=\fil_i^{\cF}\Delta_\Dpik[1/t]$ by $(\varphi,\Gamma)$-submodules over $\cR_{E,L}[1/t]$.\;Intersecting this filtration with $\Dpik$ gives a filtration
\[\cF=\fil_{\bullet}^{\cF}\Dpik: \ 0 =\fil_0^{\cF}\Dpik \subsetneq \fil_1^{\cF}\Dpik \subsetneq \cdots \subsetneq \fil_{k}^{\cF} \Dpik=\Dpik,\;\fil_i^{\cF}\Dpik=(\fil_{i}^{\cF}\cM_\Dpik)\cap \Dpik,\]
This filtration consists of saturated $(\varphi,\Gamma)$-submodules of $\Dpik$.\;Since $\Dpik$ is potentially semistable, it is de Rham. Its de Rham module is
\[D_{\dr}(\Dpik)\cong(D_{\pst}(\Dpik)\otimes_{L'_0}L')^{\gal(L'/L)}\]
This is a free $L\otimes_{\bQ_p}E$-module of rank $n$, equipped with its Hodge filtration.\;Assume that $D_{\dr}(\Dpik)$ has distinct Hodge--Tate weights $\bh:=(\hpi_{1,\tau}>\hpi_{2,\tau}>\cdots>\hpi_{n,\tau})_{\tau\in\Sigma_L}$, and put $\hpi_i:=(\hpi_{i,\tau})_{\tau\in\Sigma_L}$.\;By Berger's equivalence, $\fil_i^{\cF}\Dpik$ corresponds to the filtered Deligne--Fontaine module $\fil_i^{\cF}\df_\Dpik$ with the induced Hodge filtration.\;We call $\cF$ \textit{non-critical} if the Hodge--Tate weights of $\fil_i^{\cF}\Dpik$ and $\gr_i^{\cF}\Dpik$ are, respectively,
\[\{\hpi_{1,\tau},\hpi_{2,\tau},\ldots,\hpi_{ir,\tau}\}_{\tau\in\Sigma_L}\quad\text{and}\quad\{\hpi_{(i-1)r+1,\tau},\hpi_{(i-1)r+2,\tau},\ldots,\hpi_{ir,\tau}\}_{\tau\in\Sigma_L}.\]
In this case, Berger's equivalence \cite[Theorem A]{berger2008equations} and comparison of Hodge--Tate weights give an injection (see also \cite[(2.4)]{Ding2021})
\begin{equation}\label{Dpikinjection}
	\mathbf{I}_{i}:	\gr_{i}^{\cF} \Dpik \hookrightarrow\gr^{\cF}_i{\Delta_\Dpik}\tee \cR_{E,L}(z^{\hpi_{ir}}),
\end{equation}
for $i=1,\ldots,k$.\;Equivalently, we have an isomorphism of $(\varphi,\Gamma)$-modules over $\cR_{E,L}[1/t]$:
$\gr_{i}^{\cF}\Dpik[1/t]\cong\gr^{\cF}_i{\Delta_\Dpik}\tee \cR_{E,L}(z^{\hpi_{ir}})[1/t]$; equivalently, there is an injection of $(\varphi,\Gamma)$-modules over $\cR_{E,L}$:
\[\mathbf{I}_{i}': \gr^{\cF}_i{\Delta_\Dpik}  \hookrightarrow \gr_{i}^{\cF}\Dpik\tee \cR_{E,L}(z^{\hpi_{(i-1)r+1}}).\]
Thus $\Dpik$ admits a non-critical $\omepik$-filtration $\cF$ with one of the following two parameters:
\begin{itemize}
	\item  $(\underline{x}_{\beta},z_{\beta}^\bh)\in \sbanpik\times\rigchl$:
	\begin{equation*}
		\begin{aligned}
			\underline{x}_{\beta}=(\underline{x}_{\beta,i})_{1\leq i\leq k},\quad\pi_{\underline{x}_{\beta,i}}\cong\pi_x\te\unr(q_L^{i-k})\circ\det,\quad z_{\beta}^\bh:=(z^{\bh}_{\beta,i}:=z^{\hpi_{ir}})_{1\leq i\leq k};
		\end{aligned}
	\end{equation*}
	\item $(\underline{x}_{\beta,\bh},z_{\beta}^\bh)\in \sbanpik\times\rigch$:
	\begin{equation*}
		\begin{aligned}
			\underline{x}_{\beta,\bh}=(\underline{x}_{\beta,\bh,i})_{1\leq i\leq k},\quad\pi_{\underline{x}_{\beta,\bh,i}}\cong\pi_x\te\unr(q_L^{i-k}\varpi_L^{\hpi_{ir}})\circ\det.
		\end{aligned}
	\end{equation*}
\end{itemize}
For $1\leq i\leq j\leq k$, put $\Dpik_i^j:=\fil_j^{\cF}\Dpik/\fil_{i-1}^{\cF}\Dpik$.\;We also write $\cF_i^j$ for the induced non-critical special filtration on $\Dpik_i^j$.\;For $1\leq i\leq k$, put
\[\Dpik_i:=\gr_i^{\cF}\Dpik,\qquad \Delta_{x,i}:=\Delta_x(\unr(q_L^{i-k})z^\bh_{\beta,i})=\Delta_x\tee\cR_{E,L}(\unr(q_L^{i-k})z^\bh_{\beta,i}).\]
The filtration $\cF$ therefore expresses $\Dpik$ as an iterated extension of its graded pieces $\{\Dpik_i\}_{1\leq i\leq k}$:
\begin{equation}\label{standardparaboline}
	\Dpik=[\Dpik_{1}-\Dpik_{2}-\cdots-\Dpik_{k}].\;
\end{equation}
We use the following notation:
\[\EndO_x:=\EndO(\Delta_x):=\Delta_x\tee\Delta_x^\vee,\quad \mathrm{Ad}_x:=\mathrm{Ad}(\Delta_x):=\EndO(\Delta_x)/\cR_{E,L}\]
Thus, $\EndO_x=\mathrm{Ad}_x\oplus\cR_{E,L}$.\;

\begin{rmk}For $1\leq i\leq k-1$,\;we say more about $\Dpik_i^{i+1}$.\;The injections $\bI_j$ and $\bI'_{j}$ for $j\in\{i,i+1\}$ induce the following injections:
	\begin{equation}
		\begin{aligned}
			f_i:\Dpik_i^{i+1}\hookrightarrow \EndO_x(\unr(q_L)z^{\hpi_{(i+1)r}-\hpi_{(i-1)r+1}}),\quad
			f_i':\EndO_x(\unr(q_L)z^{\hpi_{ir+1}-\hpi_{ir}})\hookrightarrow \Dpik_i^{i+1}.\;
		\end{aligned}
	\end{equation}
We therefore define
\[(\Dpik_i^{i+1})^z:=f_i^{-1}\big(\cR_{E,L}(\unr(q_L)z^{\hpi_{(i+1)r}-\hpi_{(i-1)r+1}})\big),\quad (\Dpik_i^{i+1})^0:=f_i^{-1}\big(\mathrm{Ad}_x(\unr(q_L)z^{\hpi_{(i+1)r}-\hpi_{(i-1)r+1}})\big).\]
In particular, the following are injections of (generalized) $(\varphi,\Gamma)$-modules:
\begin{equation}
	\begin{aligned}
		\cR_{E,L}(\unr(q_L)z^{\hpi_{ir+1}-\hpi_{ir}})\hookrightarrow &\;(\Dpik_i^{i+1})^z\hookrightarrow  \cR_{E,L}(\unr(q_L)z^{\hpi_{(i+1)r}-\hpi_{(i-1)r+1}}),\\
	\mathrm{Ad}_x(\unr(q_L)z^{\hpi_{ir+1}-\hpi_{ir}})\hookrightarrow &\;(\Dpik_i^{i+1})^0\hookrightarrow  \mathrm{Ad}_x(\unr(q_L)z^{\hpi_{(i+1)r}-\hpi_{(i-1)r+1}}).
	\end{aligned}
\end{equation}
\end{rmk}

\begin{dfn}[Non-critical special $\omepik$-filtration]\label{dfnnoncriticalspecial} A non-critical $\omepik$-filtration $\cF$ on $\Dpik$ with parameter $(\underline{x}_{\beta},z_{\beta}^\bh)\in\sbanpik\times\rigchl$ (resp.,\;$(\underline{x}_{\beta,\bh},z_{\beta}^\bh)\in\sbanpik\times\rigch$) is called special if $\hH^2_{(\varphi,\Gamma)}(\Dpik_s,\Dpik_1^s)=0$ for any $1\leq s\leq k$.\; 
\end{dfn}
\begin{rmk}
When $r=1$ and $\Dpik_i^{i+1}$ is non-split for every $1\leq i\leq n$, the proof of \cite[Proposition 3.6]{2019DINGSimple} implies that $\hH^2_{(\varphi,\Gamma)}(\Dpik_s,\Dpik_1^s)=0$ for every $1\leq s\leq n$.\;See Remark \ref{splitcondeplain} for further discussion of the vanishing of $\hH^2_{(\varphi,\Gamma)}(\Dpik_s,\Dpik_1^s)$.
\end{rmk}

\subsection{Deformation of type {$\omepik$}}\label{defdfnomepik}\label{dfndeftyoeomega}

Following \cite[Section 4.1]{Ding2021}, we study paraboline deformations of $(\varphi,\Gamma)$-modules that admit an $\omepik$-filtration.

We begin by recalling the setup in \cite[Section 4.1.1]{Ding2021}.\;Let $\Delta$ be an irreducible rank-$r$ $(\varphi,\Gamma)$-module over $\cR_{E,L}$ that is de Rham with constant Hodge--Tate weight $0$.\;Let $D$ be a rank-$r$ $(\varphi,\Gamma)$-module for which there exist a continuous character $\delta:L^\times\rightarrow E^\times$ and an injection $D\otimes_{\cR_{E,L}}\cR_{E,L}(\delta^{-1})\hooklongrightarrow\Delta$.\;The authors define the functor
\[F_{D}^0:\Art(E):=\{\text{Artinian local $E$-algebra with residue field $E$}\}\longrightarrow \{\text{sets}\}\]
which assigns to each $A$ the set of isomorphism classes of triples $(D_A,\pi_A,\delta_A)$, where
\begin{description}
\item[(1)] $D_A$ is a rank-$r$ $(\varphi,\Gamma)$-module over $\cR_{A,L}$, and $\pi_A:D_A\otimes_AE\xrightarrow{\sim}D$;
\item[(2)] ${\delta}_{A}:L^\times\rightarrow A^\times$ is a continuous character such that ${\delta}_{A}\equiv {\delta}\mod \mathfrak{m}_A$;
\item[(3)] there is an injection of rank-$r$ $(\varphi,\Gamma)$-modules $D_A\hooklongrightarrow\Delta\otimes_{\cR_{E,L}}\cR_{A,L}(\delta_A)$.
\end{description}


Fix a $(\varphi,\Gamma)$-module $\Dpik$ equipped with a non-critical special $\omepik$-filtration $\cF$ of parameter $(\underline{x}_\beta,z_\beta^\bh)\in\sbanpik\times\rigchl$ (Definition \ref{dfnnoncriticalspecial}).\;Let $\Art(E)$ be the category of Artinian local $E$-algebras with residue field $E$.\;Define $F_{\Dpik,\cF}$ to be the deformation functor that assigns to each $A\in\Art(E)$ the set of isomorphism classes of triples $(D_A,\pi_A,\cF_A)$, where
\begin{itemize}
	\item[(1)] $D_A$ is a rank-$n$ $(\varphi,\Gamma)$-module over $\cR_{A,L}$, and $\pi_A:D_A\otimes_AE\xrightarrow{\sim}\Dpik$;
	\item[(2)] $\cF_A=\fil_\bullet D_A$ is an increasing filtration by $(\varphi,\Gamma)$-submodules such that each $\fil_iD_A$ is a direct summand of $D_A$ as an $\cR_{A,L}$-module and $\pi_A(\fil_iD_A)=\fil_i^\cF\Dpik$.
\end{itemize}
Let $F_{\Dpik,\cF}^0$ be the subfunctor parametrizing $\omepik$-deformations.\;It assigns to each $A$ the set of isomorphism classes of quadruples $(D_A,\pi_A,\cF_A,\delta_A)$ satisfying (1)--(2) and the following conditions:
\begin{itemize}
\item[(3)] $\delta_A=(\delta_{A,i})_{1\leq i\leq k}$, where each $\delta_{A,i}:L^\times\rightarrow A^\times$ is continuous and $\delta_{A,i}\equiv z^\bh_{\beta,i}\pmod{\mathfrak{m}_A}$;
\item[(4)] for every $i$, there is an injection of rank-$r$ $(\varphi,\Gamma)$-modules
\begin{equation}\label{injectiondefF0deforma}
	\gr^{\cF_A}_iD_A\hooklongrightarrow \Delta_x(\unr(q_L^{i-k}))\otimes_{\cR_{E,L}}\cR_{A,L}({\delta}_{A,i}).
\end{equation}
\end{itemize}
Forgetting $\delta_A$ gives a natural morphism $F_{\Dpik,\cF}^0\rightarrow F_{\Dpik,\cF}$.\;Passing to the graded pieces induces natural morphisms
\begin{equation}\label{FmorphismGRD}
\Upsilon^0_{\Dpik}:F_{\Dpik,\cF}^0\longrightarrow \prod_{i=1}^kF_{\Dpik_i}^0\;\;\; (\text{resp.,\;}\Upsilon_{\Dpik}:F_{\Dpik,\cF}\longrightarrow \prod_{i=1}^kF_{\Dpik_i})
\end{equation}
Consequently,
\[F_{\Dpik,\cF}^0=\sF_{\Dpik,\cF}\times_{\prod_{i=1}^k F_{\Dpik_i}}\prod_{i=1}^k F_{\Dpik_i}^0.\]
The functor
\[A\in \Art_E\mapsto \{\delta_A:\bZ_{\beta}(L)\rightarrow A^\times,\;{\delta}_{A,i}\equiv z^{\bh}_{\beta,i}(\mathrm{mod}\;\mathfrak{m}_A)\}\]
is pro-represented by the completion $\widehat{(\cZ_{\beta,L})}_{z_{\beta}^\bh}$ of $\cZ_{\beta,L}$ at $z_{\beta}^\bh$.\;By \cite[Lemma 6.2.2]{Ding2021}, this gives a morphism of functors
\begin{equation}\label{omegadeltapinon}
	\omega_{\Dpik}:F_{\Dpik,\cF}^0\rightarrow \widehat{(\cZ_{{\beta},L})}_{z_{\beta}^\bh},\;\;(D_A,\pi_A,\cF_A,\delta_A)\mapsto \delta_A,
\end{equation}
Evaluating this morphism on the dual numbers yields the following $E$-linear map:
\begin{equation}\label{kappakappaL}
\begin{aligned}
	\kappa:F_{\Dpik,\cF}^0(E[\epsilon]/\epsilon^2)\xrightarrow{\omega_{\Dpik}(E[\epsilon]/\epsilon^2)} &\widehat{(\cZ_{\bL_{\beta},L})}_{z_{\beta}^\bh}(E[\epsilon]/\epsilon^2)\xrightarrow{\sim}\prod_{i=1}^k\homo(L^\times,E)\xrightarrow{\kappa_L} \prod_{i=1}^{k-1}\homo(L^\times,E),
\end{aligned}
\end{equation}
where the middle identification is given by
$((D_{E[\epsilon]/\epsilon^2,i},\pi_{E[\epsilon]/\epsilon^2,i},\delta_{E[\epsilon]/\epsilon^2,i}))_{1\leq i\leq k} 
\mapsto((
\delta_{E[\epsilon]/\epsilon^2,i}(z^{\bh}_{\beta,i})^{-1}-1)/\epsilon)_{1\leq i\leq k}$.\;The final map $\kappa_L$ sends $(\psi_1,\ldots,\psi_k)$ to $(\psi_i-\psi_{i+1})_{1\leq i\leq k-1}$.\;

In Section \ref{cgsformulaSection}, we prove that $F_{\Dpik,\cF}^0$ is pro-representable and is a subfunctor of $F_{\Dpik,\cF}$.\;We then use a parabolic Colmez--Greenberg--Stevens formula to compute $\dim_EF_{\Dpik,\cF}^0(E[\epsilon]/\epsilon^2)$.



\subsection{Construction of a pairing}\label{conspairing}Suppose that $\Dpik$ admits a non-critical special $\omepik$-filtration $\cF$ with parameter $(\underline{x}_{\beta},z_{\beta}^\bh)\in \sbanpik\times\rigchl $ (see Definition \ref{dfnnoncriticalspecial}).\;Throughout this section, fix $1\leq i\leq k-1$.\;Set
\[\delta^{\bh,i+1}_{\beta,i}:=\unr(q_L^{-1})z^{\bh}_{\beta,i}(z^{\bh}_{\beta,i+1})^{-1}=\unr(q_L^{-1})z^{\bh_{ir}-\bh_{(i+1)r}}=\ccyc z^{\bh_{ir}-\bh_{(i+1)r}-1}.\]
We construct a pairing that is nondegenerate on the middle term:
\begin{equation}\label{normalperfectpairing}
	\ext^1_{(\varphi,\Gamma)}(\Dpik_{i+1},\Dpik_{i})  \;\;\times\;\;  \hH^1_{(\varphi,\Gamma)}(\cR_{E,L}) \xrightarrow{\cup} \ext^2_{(\varphi,\Gamma)}(\cR_{E,L}(\delta^{\bh,i+1}_{\beta,i}))\cong E.
\end{equation}
Note that $\hH^1_{(\varphi,\Gamma)}(\cR_{E,L}) \cong\homo(L^\times,E)$.\;By \cite{liu2007cohomology},\;the cup product induces a pairing
\begin{equation}\label{usualcupproduct}
\xymatrix@C=5ex{
	\ext^1_{(\varphi,\Gamma)}(\Dpik_{i+1},\Dpik_{i})    \;\;\times\;\;  \ext^1_{(\varphi,\Gamma)}(\Dpik_{i},\Dpik_{i})  \xrightarrow{\cup}  \ext^2_{(\varphi,\Gamma)}(\Dpik_{i+1},\Dpik_{i}) .}
	\end{equation}
We show below that $\dim_E\ext^1_{(\varphi,\Gamma)}(\Dpik_{i+1},\Dpik_{i})=\dim_E\ext^1_{(\varphi,\Gamma)}(\Dpik_{i},\Dpik_{i})=1+d_Lr^2$ and $\ext^2_{(\varphi,\Gamma)}(\Dpik_{i+1},\Dpik_{i})\cong E$.\;More precisely, we obtain (\ref{normalperfectpairing}) by restricting the second factor of the cup-product pairing to the subspace $\hH^1_{(\varphi,\Gamma)}(\cR_{E,L})$ of $\ext^1_{(\varphi,\Gamma)}(\Delta_{x,i},\Delta_{x,i})$.\;We show that this pairing is nondegenerate on the middle term.\;The parabolic Fontaine--Mazur simple $\sL$-invariants are then defined in terms of this pairing.
	
By (\ref{Dpikinjection}), the injections $\mathbf{I}_{j}:\Dpik_{j}\hookrightarrow \Delta_{x,j}$ for $j\in\{i,i+1\}$ induce the following natural morphisms:
\begin{equation}\label{preliminamorphism1}
\begin{aligned}
	\ext^j_{(\varphi,\Gamma)}(\Dpik_{i+1},\Dpik_{i}) & \xrightarrow{f_j^{i,2}}  & \ext^j_{(\varphi,\Gamma)}(\Dpik_{i+1},\Delta_{x,i}) & \xleftarrow{g_j^{i,1}} & \ext^j_{(\varphi,\Gamma)}(\Delta_{x,i+1},\Delta_{x,i}),\\
	\ext^j_{(\varphi,\Gamma)}(\Delta_{x,i+1},\Delta_{x,i}) & \xleftarrow{g_j^{i,2}} & \ext^j_{(\varphi,\Gamma)}(\Delta_{x,i+1},\Dpik_{i})& \xrightarrow{f_j^{i,1}} & \ext^j_{(\varphi,\Gamma)}(\Dpik_{i+1},\Dpik_{i}).
\end{aligned}
\end{equation}
for $j=0,1,2$.
\begin{lem}\label{pullbackforiso}For $l\in \{1,2\}$ and $j\in\{0,1,2\}$,\;the morphisms $f_j^{i,l}$ and $g_j^{i,l}$ in (\ref{preliminamorphism1}) induce isomorphisms.\;
\end{lem}
\begin{proof}Let $D_0\hookrightarrow D_1$ be an injection of $(\varphi,\Gamma)$-modules over $\cR_{E,L}$ with the same rank.\;Using the identification between the cohomology of $(\varphi,\Gamma)$-modules and the Galois cohomology of $E$-$B$-pairs (see, for example, \cite[Section 3]{nakamura2009classification}), the injection $D_0\hookrightarrow D_1$ gives a long exact sequence
\begin{equation}\label{isoIIIlong}
	\begin{aligned}
		0\longrightarrow   \hH^0_{(\varphi,\Gamma)}(D_0)&\longrightarrow
		\hH^0_{(\varphi,\Gamma)}(D_1)  \longrightarrow \hH^0_{(\varphi,\Gamma)}(D_1/D_0)\\
		\longrightarrow   &\;\hH^1_{(\varphi,\Gamma)}(D_0)\longrightarrow
		\hH^1_{(\varphi,\Gamma)}(D_1)  \longrightarrow \hH^1_{(\varphi,\Gamma)}(D_1/D_0)\\
		&\;\longrightarrow   \hH^2_{(\varphi,\Gamma)}(D_0)\longrightarrow
		\hH^2_{(\varphi,\Gamma)}(D_1)  \longrightarrow \hH^2_{(\varphi,\Gamma)}(D_1/D_0).
	\end{aligned}
\end{equation}
By \cite[Theorem 4.7]{liu2007cohomology}, we have $\hH^2_{(\varphi,\Gamma)}(D_1/D_0)=0$ and
\begin{equation}\label{equHIH2}
	\dim_E\hH^0_{(\varphi,\Gamma)}(D_1/D_0)=\dim_E\hH^1_{(\varphi,\Gamma)}(D_1/D_0).
\end{equation}
For a $(\varphi,\Gamma)$-module $D'$ over $\cR_{E,L}$, let $\wdr(D')$ denote the associated $\bdr\otimes_{\bQ_p}E$-representation of $\gal_L$.\;For $\tau\in \Sigma_L$, let $\bdr_{,\tau,E}:=\bdr\otimes_{L,\tau} E$.\;By \cite[Lemma 5.1.1]{breuil2020probleme}, we have
\begin{equation}\label{H0wdr}
	\hH^0_{(\varphi,\Gamma)}(D_1/D_0)\cong \hH^0(\gal_L,\wdr(D_1)/\wdr(D_0)),
\end{equation}
Applying the above discussion to the following pairs $(D_0,D_1)$:
\begin{equation*}
	\begin{aligned}
		&\;(\Dpik_{i+1}^\vee\tee\Dpik_{i},\Dpik_{i+1}^\vee\tee\Delta_{x,i}),\;
		\text{resp.,}\;(\Delta_{x,i+1}^\vee\tee\Delta_{x,i},\Dpik_{i+1}^\vee\tee\Delta_{x,i}),\\
		\text{resp.,}&\;(\Delta_{x,i+1}^\vee\tee\Dpik_{i},\Delta_{x,i+1}^\vee\tee\Delta_{x,i}),\;
		\text{resp.,}\;(\Delta_{x,i+1}^\vee\tee\Dpik_{i},\Dpik_{i+1}^\vee\tee\Dpik_{i}).\;
	\end{aligned}
\end{equation*}
Note that 
\begin{equation*}
	\begin{aligned}
		&\wdr(\Dpik_{i+1}^\vee\tee\Delta_{x,i})/\wdr(\Dpik_{i+1}^\vee\tee\Dpik_{i})\\
		&\hspace{80pt}\cong \oplus_{\tau\in \Sigma_L}\oplus_{s=1}^r\oplus_{l=1}^rt^{\hpi_{ir,\tau}-\hpi_{ir+s,\tau}}\bdr_{,\tau,E}/t^{\hpi_{(i-1)r+l,\tau}-\hpi_{ir+s,\tau}}\bdr_{,\tau,E},\\
		\text{resp.,\;}&\wdr(\Dpik_{i+1}^\vee\tee\Delta_{x,i})/\wdr(\Delta_{x,i+1}^\vee\tee\Delta_{x,i})\\
		&\hspace{80pt}\cong \oplus_{\tau\in \Sigma_L}\oplus_{s=1}^r\oplus_{l=1}^rt^{\hpi_{ir,\tau}-\hpi_{ir+s,\tau}}\bdr_{,\tau,E}/t^{\hpi_{ir,\tau}-\hpi_{(i+1)r,\tau}}\bdr_{,\tau,E},
		\\
		\text{resp.,\;}&		
		\wdr(\Delta_{x,i+1}^\vee\tee\Delta_{x,i})/\wdr(\Delta_{x,i+1}^\vee\tee\Dpik_{i})\\
		&\hspace{80pt}\cong \oplus_{\tau\in \Sigma_L}\oplus_{s=1}^r\oplus_{l=1}^rt^{\hpi_{ir,\tau}-\hpi_{(i+1)r,\tau}}\bdr_{,\tau,E}/t^{\hpi_{(i-1)r+l,\tau}-\hpi_{(i+1)r,\tau}}\bdr_{,\tau,E},\\
		\text{resp.,\;}&\wdr(\Dpik_{i+1}^\vee\tee\Dpik_{i})/\wdr(\Delta_{x,i+1}^\vee\tee\Dpik_{i})\\
		&\hspace{80pt}\cong \oplus_{\tau\in \Sigma_L}\oplus_{s=1}^r\oplus_{l=1}^rt^{\hpi_{(i-1)r+l,\tau}-\hpi_{ir+s,\tau}}\bdr_{,\tau,E}/t^{\hpi_{(i-1)r+l,\tau}-\hpi_{(i+1)r,\tau}}\bdr_{,\tau,E}.
	\end{aligned}
\end{equation*}
For each of these pairs $(D_0,D_1)$, we have $\hH^0(\gal_L,\wdr(D_1)/\wdr(D_0))=0$.\;Hence $\hH^0_{(\varphi,\Gamma)}(D_1/D_0)=0$, and (\ref{equHIH2}) gives $\hH^1_{(\varphi,\Gamma)}(D_1/D_0)=0$.\;The lemma now follows from the long exact sequence (\ref{isoIIIlong}).
\end{proof}

For $j=0,1,2$,\;we have
\begin{equation}\label{Linvariantdecomp}
	\begin{aligned}
		\ext^j_{(\varphi,\Gamma)}(\Delta_{x,i+1},\Delta_{x,i})=&\;\hH^j_{(\varphi,\Gamma)}(\cR_{E,L}(\delta^{\bh,i+1}_{\beta,i}))\oplus \hH^j_{(\varphi,\Gamma)}(\mathrm{Ad}_x(\delta^{\bh,i+1}_{\beta,i}))\\
		\text{resp.,\;}\ext^j_{(\varphi,\Gamma)}(\Delta_{x,i},\Delta_{x,i})=&\; \hH^j_{(\varphi,\Gamma)}(\cR_{E,L})\oplus \hH^j_{(\varphi,\Gamma)}(\mathrm{Ad}_x).
	\end{aligned}
\end{equation}
For $j=1,2$, these decompositions induce the following injections:
\begin{equation}\label{i3p3q3}
	\begin{aligned}
		&\widetilde{\mathbf{I}}_{i}^j:\hH^j_{(\varphi,\Gamma)}(\cR_{E,L}(\delta^{\bh,i+1}_{\beta,i}))\rightarrow \ext^j_{(\varphi,\Gamma)}(\Delta_{x,i+1},\Delta_{x,i}),\\
		\text{resp.,\;}& \widetilde{\mathbf{I}}'_{i}:\hH^1_{(\varphi,\Gamma)}(\cR_{E,L})\rightarrow \ext^1_{(\varphi,\Gamma)}(\Delta_{x,i},\Delta_{x,i}).\;
	\end{aligned}
\end{equation}

	Recall the theory of $E$-$B$-pairs from \cite{nakamura2009classification}.\;By \cite[Theorem 1.36]{nakamura2009classification}, the category of $E$-$B$-pairs is equivalent to the category of $(\varphi,\Gamma)$-modules over $\cR_{E,L}$.\;For a continuous character $\chi:L^\times\rightarrow E^\times$, let $B_E(\chi)$ be the associated rank-one $E$-$B$-pair \cite[Theorem 1.45]{nakamura2009classification}.\;Let $\bW=(\bW_e,\bW_{\dr}^+)$ (resp.,\;$\bW^0=(\bW_e^0,\bW_{\dr}^{0,+})$) be the $E$-$B$-pair associated with $\EndO(\Delta_x)$ (resp.,\;$\mathrm{Ad}_x$), and put $\bW(\chi):=\bW\otimes_EB_E(\chi)$ and $\bW^0(\chi):=\bW^0\otimes_EB_E(\chi)$.\;
	
	\begin{lem}\label{isoI2i}The morphism $\widetilde{\mathbf{I}}_{i}^2$ is an isomorphism.
	\end{lem}
	\begin{proof}By Tate duality, it suffices to prove that
	\[(\widetilde{\mathbf{I}}_{i}^2)^\vee:\ext^0_{(\varphi,\Gamma)}(\EndO(\Delta_x)\tee \cR_{E,L}(\ccyc(\delta^{\bh,i+1}_{\beta,i})^{-1}))
		\rightarrow  \ext^0_{(\varphi,\Gamma)}(\cR_{E,L}(\ccyc(\delta^{\bh,i+1}_{\beta,i})^{-1}))\] 
		is an isomorphism; equivalently, we need to show that
		$\hH^0_{(\varphi,\Gamma)}(\mathrm{Ad}_x\tee \cR_{E,L}(\ccyc(\delta^{\bh,i+1}_{\beta,i})^{-1}))=0$.\;A generator of $\hH^0_{(\varphi,\Gamma)}(\cR_{E,L}(\ccyc(\delta^{\bh,i+1}_{\beta,i})^{-1}))\cong E$ is represented by an injection $\cR_{E,L}\hookrightarrow\cR_{E,L}(\ccyc(\delta^{\bh,i+1}_{\beta,i})^{-1})$. Moreover,
		\begin{eqnarray*}
			\bW^0(\ccyc(\delta^{\bh,i+1}_{\beta,i})^{-1})^+_{\dr}/\bW^{0,+}_{\dr}
			\cong \oplus_{\tau\in \Sigma_L}\big(t^{\bh_{(i+1)r,\tau}-\bh_{ir,\tau}}\bdr_{,\tau,E}/\bdr_{,\tau,E}\big)^{\oplus (r^2-1)},
		\end{eqnarray*}
		we deduce that the corresponding quotient has no $\gal_L$-invariants.\;Since $\Delta_x$ is irreducible,\;a d\'{e}vissage argument gives
		$\hH^0_{(\varphi,\Gamma)}(\mathrm{Ad}_x\tee\cR_{E,L}(\ccyc(\delta^{\bh,i+1}_{\beta,i})^{-1}))\cong\hH^0_{(\varphi,\Gamma)}(\mathrm{Ad}_x)=0$.\;
	\end{proof}

\begin{lem}\label{phigammacohononcrtialspecial}
	\begin{itemize}
		\item[(a)] For any $j\neq i-1$,\;we have $\hH^2_{(\varphi,\Gamma)}(\Dpik_{i},\Dpik_{j})\cong \hH^0_{(\varphi,\Gamma)}(\Dpik_{j},\Dpik_{i}(\ccyc))=0.$
		\item[(b)] For any $i\neq j$,\;we have $\hH^0_{(\varphi,\Gamma)}(\Dpik_{i},\Dpik_{j})=0$ and $\hH^2_{(\varphi,\Gamma)}(\Dpik_{i},\Dpik_{i-1})\cong E$.\;
	\end{itemize}
\end{lem}
\begin{proof}If $j\neq i-1$, the pairs $(\underline{x}_{\beta,i},z^{\bh}_{\beta,i})$ and $(\underline{x}_{\beta,j},z^{\bh}_{\beta,j})$ satisfy the genericity condition in \cite[(4.13)]{Ding2021}; hence, (a) follows from \cite[Lemma 4.1.12]{Ding2021}.\;For $i\neq j$, the proof of the vanishing $\hH^0_{(\varphi,\Gamma)}(\Dpik_i,\Dpik_j)=0$ in that lemma applies without change.\;The final assertion in (b) follows from Lemma \ref{isoI2i}.
\end{proof}

The injection $\mathbf{I}_{i}:\Dpik_{i}\hookrightarrow \Delta_{x,i}$ also induces morphisms
\begin{equation}\label{preliminamorphism2}
\begin{aligned}
	\ext^1_{(\varphi,\Gamma)}(\Dpik_i,\Dpik_i)  \xrightarrow{\mathbf{I}_{i}'}  \ext^1_{(\varphi,\Gamma)}(\Dpik_i,\Delta_{x,i})  \xleftarrow{\mathbf{I}_{i}''}  \ext^1_{(\varphi,\Gamma)}(\Delta_{x,i},\Delta_{x,i}).
\end{aligned}
\end{equation}

\begin{lem}\label{deformationker}The morphism $\mathbf{I}_{i}''$ is an isomorphism,\;and $\dim_E\ker(\mathbf{I}_{i}')=d_L\frac{r(r-1)}{2}$.\;
\end{lem}
\begin{proof}The injectivity argument in \cite[Proposition 4.1.4, (4.10)]{Ding2021}, together with (\ref{isoIIIlong}) and (\ref{equHIH2}), shows that $\mathbf{I}_i''$ is an isomorphism.\;The dimension of $\ker(\mathbf{I}_i')$ is computed in the final paragraph of the proof of \cite[Proposition 4.1.4]{Ding2021}.
\end{proof}

For $1\leq i\leq k-1$, the morphisms in (\ref{preliminamorphism1}) and (\ref{preliminamorphism2}) fit into the following commutative diagram:
\begin{equation}\label{bigcommutediag}
\xymatrix@C=5ex{
	\ext^1_{(\varphi,\Gamma)}(\Dpik_{i+1},\Dpik_{i}) \ar@{=}[d]                        & \times & \ext^1_{(\varphi,\Gamma)}(\Dpik_{i},\Dpik_{i})\hspace{9pt} \ar[d]^{\mathbf{I}_{i}'}         \ar[r]^{\cup}_{c_1}
	& \ext^2_{(\varphi,\Gamma)}(\Dpik_{i+1},\Dpik_{i}) \ar[d]_{\backsimeq}^{f_1^{i,2}} \\
	\ext^1_{(\varphi,\Gamma)}(\Dpik_{i+1},\Dpik_{i}) \ar[d]^{f_1^{i,2}}_{\simeq}             & \times & \ext^1_{(\varphi,\Gamma)}(\Dpik_{i},\Delta_{x,i})\hspace{5pt}                 \ar[r]^{\cup}_{c_2}
	& \ext^2_{(\varphi,\Gamma)}(\Dpik_{i+1},\Delta_{x,i}) \ar@{=}[d]           \\
	\ext^1_{(\varphi,\Gamma)}(\Dpik_{i+1},\Delta_{x,i})                             & \times & \ext^1_{(\varphi,\Gamma)}(\Delta_{x,i},\Delta_{x,i})\ar[u]_{\mathbf{I}_{i}''}^{\simeq}  \ar[r]^{\cup}_{c_3}
	&\ext^2_{(\varphi,\Gamma)}(\Dpik_{i+1},\Delta_{x,i})                       \\
	\ext^1_{(\varphi,\Gamma)}(\Delta_{x,i+1},\Delta_{x,i}) \ar[u]_{g_1^{i,1}}^{\simeq}  & \times & \ext^1_{(\varphi,\Gamma)}(\Delta_{x,i},\Delta_{x,i}) \ar@{=}[u]            \ar[r]^{\cup}_{c_4}
	&\ext^2_{(\varphi,\Gamma)}(\Delta_{x,i+1},\Delta_{x,i}) \ar[u]^{\backsimeq}_{g_1^{i,1}}.}
	\end{equation}
	
	\begin{lem}\label{lembigcommutediag}
The cup products $c_2$, $c_3$, and $c_4$ in (\ref{bigcommutediag}) are perfect pairings.
\end{lem}
\begin{proof}It suffices to prove that $c_4$ is perfect.\;The natural morphism $j:\bW(\delta^{\bh,i+1}_{\beta,i})\rightarrow\bW(\ccyc)$, induced by $B_E(\delta^{\bh,i+1}_{\beta,i})\hookrightarrow B_E(\ccyc)$, gives an exact sequence of $\gal_L$-complexes
\begin{equation*}
	\begin{aligned}
		0\longrightarrow [\bW(\delta^{\bh,i+1}_{\beta,i})_e\oplus \bW(\delta^{\bh,i+1}_{\beta,i})^+_{\dr}\rightarrow \bW(\delta^{\bh,i+1}_{\beta,i})_{\dr}]
		&\longrightarrow  [\bW(\ccyc)_e\oplus \bW(\ccyc)^+_{\dr}\rightarrow \bW(\ccyc)_{\dr}]\\
		&\longrightarrow [\oplus_{\tau\in\Sigma_L}t\bdr_{,\tau,E}/t^{\hpi_{\tau,ir}-\hpi_{\tau,(i+1)r}}\bdr_{,\tau,E}\rightarrow0]\longrightarrow 0.\\
	\end{aligned}
\end{equation*}
Since $\hH^i(\gal_L,\bigoplus_{\tau\in\Sigma_L}t\bdr_{,\tau,E}/t^{\hpi_{\tau,ir}-\hpi_{\tau,(i+1)r}}\bdr_{,\tau,E})=0$ for every $i\geq0$, the morphism $j$ induces isomorphisms $\hH^i(\gal_L,\bW(\delta^{\bh,i+1}_{\beta,i}))\xrightarrow{\sim}\hH^i(\gal_L,\bW(\ccyc))$.\;The following diagram therefore commutes:
\[\xymatrix@C=3.5ex{
	\ext^1_{(\varphi,\Gamma)}(\Delta_{x,i+1},\Delta_{x,i}) \ar[d]_{\simeq}  & \times & \ext^1_{(\varphi,\Gamma)}(\Delta_{x,i},\Delta_{x,i}) \ar@{=}[d] \ar[r]^{\cup}_{c_4} & \ext^2_{(\varphi,\Gamma)}(\Delta_{x,i+1},\Delta_{x,i}) \ar[d]_{\simeq}\\
	\ext^1_{(\varphi,\Gamma)}(\Delta_x,\Delta_x(\ccyc))   & \times & \ext^1_{(\varphi,\Gamma)}(\Delta_x(\ccyc),\Delta_x(\ccyc))\ar[r]^{\cup} & \ext^2_{(\varphi,\Gamma)}(\Delta_x,\Delta_x(\ccyc)) ,}\]
where $\Delta_x(\ccyc):=\Delta_x\tee\cR_{E,L}(\ccyc)$ and $\ext^2_{(\varphi,\Gamma)}(\Delta_{x,i+1},\Delta_{x,i})\xrightarrow{\sim}\ext^2_{(\varphi,\Gamma)}(\Delta_x,\Delta_x(\ccyc))\cong E$.\;The bottom cup product is perfect by Tate duality.
\end{proof}

We can now define the desired pairing.

\begin{pro}\label{pairingdef}For $1\leq i\leq k-1$, there exists a pairing that is nondegenerate on the middle term:
\begin{equation}
	\xymatrix{
		\langle -,-\rangle_{i}: \ext^1_{(\varphi,\Gamma)}(\Dpik_{i+1},\Dpik_{i})  \times  \hH^1_{(\varphi,\Gamma)}(\cR_{E,L}) \ar[r]^{\hspace{30pt}\cup}
		& \ext^2_{(\varphi,\Gamma)}(\cR_{E,L}(\delta^{\bh,i+1}_{\beta,i}))\cong E.}
\end{equation}
\end{pro}
\begin{proof}Combining the diagrams (\ref{bigcommutediag}) and (\ref{i3p3q3}) gives the following commutative diagram:
\begin{equation}\label{bigcommutediagc5}
	\xymatrix@C=2.5ex{
		\ext^1_{(\varphi,\Gamma)}(\Dpik_{i+1},\Dpik_{i}) \ar@{=}[d]                        & \times & \ext^1_{(\varphi,\Gamma)}(\Dpik_{i},\Dpik_{i})\hspace{9pt} \ar[d]^{ {\mathbf{I}}'_{i}}         \ar[r]^{\cup}_{c_1}
		& \ext^2_{(\varphi,\Gamma)}(\Dpik_{i+1},\Dpik_{i}) \ar[d]_{\backsimeq}^{f_1^{i,2}} \\
		\ext^1_{(\varphi,\Gamma)}(\Dpik_{i+1},\Dpik_{i})           & \times & \ext^1_{(\varphi,\Gamma)}(\Dpik_{i},\Delta_{x,i})\hspace{5pt}                 \ar[r]^{\cup}_{c_2}
		& \ext^2_{(\varphi,\Gamma)}(\Dpik_{i+1},\Delta_{x,i})          \\
		\hH^1_{(\varphi,\Gamma)}(\cR_{E,L}(\delta^{\bh,i+1}_{\beta,i})) \ar@{_(->}[u]_{(f_1^{i,2})^{-1}\circ g_1^{i,1}\circ \widetilde{\mathbf{I}}_{i}^1}  & \times & \hH^1_{(\varphi,\Gamma)}(\cR_{E,L}) \ar@{_(->}[u]_{ {\mathbf{I}}''_{i}\circ  \widetilde{\mathbf{I}}'_{i}}         \ar[r]^{\cup\quad}_{c_5\quad}
		&\hH^2_{(\varphi,\Gamma)}(\cR_{E,L}(\delta^{\bh,i+1}_{\beta,i})) \ar[u]^{\backsimeq}_{g_1^{i,1}\circ \widetilde{\mathbf{I}}_{i}^2}.
	}
\end{equation}
Since $\hH^2_{(\varphi,\Gamma)}(\cR_{E,L}(\delta^{\bh,i+1}_{\beta,i}))\cong E$, the desired pairing is obtained by restricting the cup product to the $(2,1)$ and $(3,2)$ terms of this diagram.\;The claimed nondegeneracy follows from the diagram.
\end{proof}

Recall $\homo(L^\times,E)$ (resp.,\;$\homo_{\sm}(L^{\times},E)$,\;resp.,\;$\homo_\sigma(L^\times,E)$) denotes the $E$-vector space of $E$-valued additive continuous (locally constant,\;resp.,\;locally $\sigma$-analytic) characters of $L^\times$,\;here  $\sigma\in\Sigma_L$.\;Let $\varepsilon_{L}: L^{\times} \rightarrow \co_{L}^{\times}$ be the character that restricts to the identity on $\co_{L}^{\times}$ and sends $\varpi_L$ to $1$.\;Let $\log:L^{\times}\rightarrow L$ be the $p$-adic logarithm restricted to $\cO_L^{\times}$.\;For $\sigma\in \Sigma_{L}$, set $\psi_{\sigma}:=\sigma\circ \log \circ \varepsilon_{L}: L^{\times} \rightarrow E$, and let $\psi_{\ur}: L^{\times} \rightarrow E$ be the unramified character sending $p$ to $1$ (and hence $\varpi$ to $e^{-1}$).\;By \cite[Lemma 1.14]{2015Ding},\;the characters $\{\psi_{\sigma}\}_{\sigma\in \Sigma_{L}}$ together with $\psi_{\ur}$ form a basis of $\homo(L^{\times},E)$,\;and $\homo_{\sm}(L^{\times},E)$ (resp.,\;$\homo_{\sigma}(L^\times,E)$ ) is the one-dimensional (resp.,\;two-dimensional) subspace generated by $\psi_{\ur}$ (resp,\;$\psi_{\ur}$ and $\psi_{\sigma}$).\;We have isomorphisms of $(d_L+1)$-dimensional  $E$-vector spaces:
\[	\hH^1_{(\varphi,\Gamma)}(\cR_{E,L})\cong  \homo(\gal_L,E)\cong\homo(L^\times,E).\]

\subsection{Parabolic Fontaine--Mazur simple {$\sL$}-invariants}\label{dfnFMparalINVSECTION}

For $1\leq i\leq k-1$, the subquotient $\Dpik_i^{i+1}$ determines an extension class $[\Dpik_i^{i+1}]\in\ext^1_{(\varphi,\Gamma)}(\Dpik_{i+1},\Dpik_i)$.
\begin{dfn}[Parabolic Fontaine--Mazur simple $\sL$-invariants] Suppose that $\Dpik$ admits a
non-critical special $\omepik$-filtration with parameter $(\underline{x}_{\beta},z_{\beta}^\bh)\in \sbanpik\times\rigchl$ (or, respectively, parameter $(\underline{x}_{\beta,\bh},z_{\beta}^\bh)\in \sbanpik\times\rigch$).\;For $1\leq i\leq k-1$, put (see Proposition \ref{pairingdef})
\begin{equation}\label{dfnl-invariant}
	\begin{aligned}
		\sL(\Dpik)_{ir}:&=\{\psi\in \hH^1_{(\varphi,\Gamma)}(\cR_{E,L})\;|\;\langle [\Dpik^{i+1}_{i}],\psi\rangle_{i}=0\}\subseteq \hH^1_{(\varphi,\Gamma)}(\cR_{E,L})\subseteq \ext^1_{(\varphi,\Gamma)}(\Delta_{x,i},\Delta_{x,i}).\;
	\end{aligned}
\end{equation}
We call 
\[\sL(\Dpik)=\prod_{i=1}^{k-1}\sL(\Dpik)_{ir}\]
the parabolic Fontaine--Mazur simple $\sL$-invariants of $\Dpik$.
\end{dfn}

Let $[({\Delta_{\Dpik}})^{i+1}_{i}]=(f_1^{i,2})^{-1}\circ g_1^{i,1}([\Dpik^{i+1}_{i}])\in \ext^1_{(\varphi,\Gamma)}(\Delta_{x,i+1},\Delta_{x,i})$.\;Let $[\padelc^{i+1}_{i}]$ (resp.,\;$[\padel^{i+1}_{i}]$) be the projection of $[({\Delta_{\Dpik}})^{i+1}_{i}]$ to the first (resp.,\;second) factor in (\ref{Linvariantdecomp}),\;so that
\[[(\Delta_{\Dpik})^{i+1}_{i}]:=[\padelc^{i+1}_{i}]\oplus [\padel^{i+1}_{i}].\] 

\begin{cor} One has $\langle\psi,[\Dpik_i^{i+1}]\rangle_i=\langle(\Delta_\Dpik)_i^{i+1},\widetilde{\mathbf{I}}'_i(\psi)\rangle_{c_4}=\langle\psi,[\padelc_i^{i+1}]\rangle_{c_5}$ (see (\ref{bigcommutediag}) and (\ref{bigcommutediagc5})).\;If $[\padelc_i^{i+1}]\neq0$, then $\dim_E\sL(\Dpik)_{ir}=d_L$.
\end{cor}

\begin{lem}\label{criofnoncrystalline}The following are equivalent:
\begin{itemize}
	\item[(a)] For every $1\leq i\leq k-1$, $\Dpik_i^{i+1}$ is potentially semistable and non-crystalline;
	\item[(b)] For every $1\leq i\leq k-1$, $\padelc_i^{i+1}$ is semistable and non-crystalline;
	\item[(c)] For every $1\leq i\leq k-1$, $\homo_{\sm}(L^\times,E)\nsubseteq\sL(\Dpik)_{ir}$;
	\item[(d)] $N^{k-1}\neq0$ on $\df_\Dpik$, i.e.\;$\df_\Dpik$ is absolutely indecomposable.
\end{itemize}
\end{lem}
\begin{proof}For $1\leq i\leq k-1$, the cup product gives a perfect pairing; see \cite[Lemma 1.13]{2015Ding}:
\[\hH^1_{(\varphi,\Gamma)}(\cR_{E,L}(\delta^{\bh,i+1}_{\beta,i}))\times\hH^1_{(\varphi,\Gamma)}(\cR_{E,L})\xrightarrow{\cup}\hH^2_{(\varphi,\Gamma)}(\cR_{E,L}(\delta^{\bh,i+1}_{\beta,i}))\cong E.\]
The orthogonal complement $\hH^1_{(\varphi,\Gamma),e}(\cR_{E,L}(\delta^{\bh,i+1}_{\beta,i}))^\bot$ is naturally isomorphic to $\homo_{\sm}(L^\times,E)$; see \cite[Proposition 1.9 and Lemma 1.15]{2015Ding}.\;Thus $\homo_{\sm}(L^\times,E)\subseteq\sL(\Dpik)_{ir}$ if and only if $\padelc_i^{i+1}$ is crystalline, proving the equivalence of (b) and (c).\;Next, since $\Dpik_i^{i+1}[1/t]=(\Delta_\Dpik)_i^{i+1}[1/t]$, (a) is equivalent to $(\Delta_\Dpik)_i^{i+1}=\padelc_i^{i+1}\oplus\padel_i^{i+1}$ being potentially semistable and non-crystalline.\;It therefore suffices to show that $\padel_i^{i+1}$ is always potentially crystalline.\;Note that $\mathrm{Ad}_x(\delta^{\bh,i+1}_{\beta,i})$ becomes crystalline over $L'$.\;Recall that $\bW^0$ and $\bW^0(\delta^{\bh,i+1}_{\beta,i})$ are the $E$-$B$-pairs associated with $\mathrm{Ad}_x$ and $\mathrm{Ad}_x(\delta^{\bh,i+1}_{\beta,i})$, respectively.\;By \cite[Proposition 2.7]{nakamura2009classification},
\begin{equation*}
	\begin{aligned}
		&\dim_E\hH^1_f(\gal_{L'},\bW^0(\delta^{\bh,i+1}_{\beta,i}))\\
		=\;&\dim_E\hH^0(\gal_{L'},\bW^0(\delta^{\bh,i+1}_{\beta,i}))+\dim_ED_{\dr}(\bW^0(\delta^{\bh,i+1}_{\beta,i}))/\fil_H^0D_{\dr}(\bW^0(\delta^{\bh,i+1}_{\beta,i}))\\
		=\;&\dim_E\hH^0(\gal_{L'},\bW^0(\delta^{\bh,i+1}_{\beta,i}))+\dim_ED_{\dr}(\bW^0(\delta^{\bh,i+1}_{\beta,i}))\\
		=\;&\dim_E\hH^1(\gal_{L'},\bW^0(\delta^{\bh,i+1}_{\beta,i})).\;
	\end{aligned}
\end{equation*}
Thus $\hH^1_f(\gal_{L'},\bW^0(\delta^{\bh,i+1}_{\beta,i}))=\hH^1(\gal_{L'},\bW^0(\delta^{\bh,i+1}_{\beta,i}))$.\;It follows from \cite[Remark 2.5]{nakamura2009classification} that $\padel_i^{i+1}|_{L'}$ is crystalline.\;Finally, the equivalence of (b) and (d) is \cite[Lemma 3.2]{2019DINGSimple}.
\end{proof}

For $\sigma\in \Sigma_{L}$,\;put $\sL(\Dpik)_{ir,\sigma}:=\sL(\Dpik)_{ir}\cap\homo_\sigma(L^\times,E)$.
\begin{rmk}\label{taulinvariantrmk}For each $1\leq i\leq k-1$, assume that $\Dpik_i^{i+1}$ is potentially semistable and non-crystalline.\;Lemma \ref{criofnoncrystalline} then gives $\homo_{\sm}(L^\times,E)\nsubseteq\sL(\Dpik)_{ir}$.\;Consequently, $\dim_E\sL(\Dpik)_{ir,\sigma}=1$ for every $\sigma\in\Sigma_L$, and $\sL(\Dpik)_{ir}=\bigoplus_{\sigma\in\Sigma_L}\sL(\Dpik)_{ir,\sigma}$.\;Thus there exists $\sL_{ir,\sigma}\in E$ such that $\sL(\Dpik)_{ir,\sigma}$ is generated by $\psi_{\sigma}-\sL_{ir,\sigma}\psi_{\ur}$.
\end{rmk}

\subsection{Parabolic Colmez--Greenberg--Stevens formula}\label{cgsformulaSection}

In \cite[Section 3.3]{2019DINGSimple}, Ding established a Colmez--Greenberg--Stevens formula for the Fontaine--Mazur simple $\sL$-invariants of a rank-$n$ non-critical special trianguline $(\varphi,\Gamma)$-module over $\cR_{E,L}$.\;Here we establish a parabolic analogue for a rank-$n$ module $\Dpik$ equipped with a non-critical special $\omepik$-filtration $\cF$ of parameter $(\underline{x}_\beta,z_\beta^\bh)\in\sbanpik\times\rigchl$.\;This formula expresses the parabolic simple Fontaine--Mazur $\sL$-invariants in terms of first-order paraboline deformations of type $\omepik$.

Fix a deformation $\widetilde{\Dpik}_1^{k-1}$ of $\Dpik_1^{k-1}$ over $\cR_{E[\epsilon]/\epsilon^2,L}$ that admits an $\omepik$-filtration $\widetilde{\cF}_1^{k-1}$ with parameter
\[((\underline{x}_{\beta,i})_{1\leq i\leq k-1},(\bm{\delta}_{E[\epsilon]/\epsilon^2,i})_{1\leq i\leq k-1}),\;\]
i.e.\;$\widetilde\Dpik_i\hookrightarrow\Delta_x(\unr(q_L^{i-k}))\otimes_{\cR_{E,L}}\cR_{E[\epsilon]/\epsilon^2,L}(\bm{\delta}_{E[\epsilon]/\epsilon^2,i})$ for $1\leq i\leq k-1$.\;Let $\widetilde{\Dpik}_{k-1}:=\gr_{k-1}^{\widetilde{\cF}_1^{k-1}}\widetilde{\Dpik}_1^{k-1}$.\;The proof of the relevant proposition in \cite{Ding2021}, together with (\ref{bigcommutediag}), shows that $\mathbf{I}'_{k-1}([\widetilde{\Dpik}_{k-1}])\in\ext^1_{(\varphi,\Gamma)}(\Dpik_{k-1},\Delta_{x,k-1})$ lies in the image of
\[{\mathbf{I}}''_{k-1}\circ\widetilde{\mathbf{I}}'_{k-1}:\hH^1_{(\varphi,\Gamma)}(\cR_{E,L})\longrightarrow\ext^1_{(\varphi,\Gamma)}(\Dpik_{k-1},\Delta_{x,k-1}).\]
We therefore regard $\mathbf{I}'_{k-1}([\widetilde{\Dpik}_{k-1}])$ as an element of $\hH^1_{(\varphi,\Gamma)}(\cR_{E,L})$.

\begin{thm}\label{defor1thdefor}Assume the setup above.\;Let
$\bm{\delta}_{E[\epsilon]/\epsilon^2,k}$ be a deformation of $z^{\bh}_{\beta,k}$ over $E[\epsilon]/\epsilon^2$, and let $\widetilde\Dpik_{k}$ be a deformation of $\Dpik_{k}$ over $\cR_{E[\epsilon]/\epsilon^2,L}$ such that \[\widetilde\Dpik_{k}\hookrightarrow {\Delta_x}\otimes_{\,\;\cR_{E,L}}\cR_{E[\epsilon]/\epsilon^2,L}(\bm{\delta}_{E[\epsilon]/\epsilon^2,k}),\quad \text{i.e.,\;}\widetilde\Dpik_{k}\in F^0_{\Dpik_{k}}(E[\epsilon]/\epsilon^2).\]
Then there exists a deformation $\widetilde\Dpik$ of $\Dpik$ over $\cR_{E[\epsilon]/\epsilon^2,L}$ that admits an $\omepik$-filtration $\widetilde{\cF}$ with parameter $((\underline{x}_{\beta,i})_{1\leq i\leq k},(\bm{\delta}_{E[\epsilon]/\epsilon^2,i})_{1\leq i\leq k})$ and $\gr^{\widetilde{\cF}}_{i}\widetilde{\Dpik}=\widetilde\Dpik_i$ for $1\leq i\leq k$ if and only if
\[\mathbf{I}'_{k-1}([\widetilde{\Dpik}_{k-1}])\otimes_{\cR_{E[\epsilon]/\epsilon^2,L}}\cR_{E[\epsilon]/\epsilon^2,L}(\bm{\delta}_{E[\epsilon]/\epsilon^2,k}^{-1}z^{\bh}_{\beta,k})\in \sL(\Dpik)_{(k-1)r}.\]
\end{thm}
\begin{rmk}The parabolic Colmez--Greenberg--Stevens formula shows that the parabolic simple Fontaine--Mazur $\sL$-invariants characterize obstructions to paraboline deformations of $\Dpik$ of type $\omepik$ over $E[\epsilon]/\epsilon^2$.
\end{rmk}

Theorem \ref{defor1thdefor} has the following consequence.

\begin{cor}\label{corlinvaraint}The map $\kappa$ in (\ref{kappakappaL}) factors through a surjective map
$\kappa:F_{\Dpik,\cF}^0(E[\epsilon]/\epsilon^2)\twoheadlongrightarrow \sL(\Dpik).$
\end{cor}

We end this section by computing the $E$-dimension of the tangent space $F_{\Dpik,\cF}^0(E[\epsilon]/\epsilon^2)$. Unlike the generic case treated in \cite[Proposition 4.1.16]{Ding2021}, the corollary below shows that the morphism $\Upsilon^0_{\Dpik}$ in (\ref{FmorphismGRD}) is not formally smooth. Thus, \cite[Proposition 4.1.17]{Ding2021} does not apply here; instead, we use Lemma \ref{fibredim}. We first recall some definitions; see (\ref{kappakappaL}) for $\Upsilon^0_{\Dpik}(E[\epsilon]/\epsilon^2)$, $\omega_{\Dpik}$, and $\kappa_L$.
\begin{cor}\label{corlinvaraint1}The image of the map \[\Upsilon^0_{\Dpik}(E[\epsilon]/\epsilon^2):F_{\Dpik,\cF}^0(E[\epsilon]/\epsilon^2)\longrightarrow \prod_{i=1}^kF_{\Dpik_i}^0(E[\epsilon]/\epsilon^2)\] is $(\kappa_L\circ\omega_{\Dpik})^{-1}(\sL(\Dpik))$.\;In particular,\;$\dim_E\im\Upsilon^0_{\Dpik}(E[\epsilon]/\epsilon^2)=1+d_L\big(k+\frac{n(r-1)}{2}\big)$ and  $\dim_E\mathrm{coker}(\Upsilon^0_{\Dpik}(E[\epsilon]/\epsilon^2))=k-1.$
\end{cor}
\begin{proof}By Corollary \ref{corlinvaraint}, the image of $\Upsilon^0_{\Dpik}(E[\epsilon]/\epsilon^2)$ is contained in $(\kappa_L\circ\omega_{\Dpik})^{-1}(\sL(\Dpik))$.\;Since $\kappa_L$ is a surjection with kernel $\homo(L^\times,E)$, it remains to prove the reverse inclusion, namely \[\im\Upsilon^0_{\Dpik}(E[\epsilon]/\epsilon^2)=(\omega_{\Dpik})^{-1}(\kappa_L^{-1}(\sL(\Dpik))).\;\]Indeed, let $(\psi_i)_{1\leq i\leq k}\in \kappa_L^{-1}(\sL(\Dpik))$ be arbitrary, and let $(\widetilde\Dpik_i)_{1\leq i\leq k}\in(\omega_{\Dpik})^{-1}((\psi_i)_{1\leq i\leq k}))\subseteq \prod_{i=1}^kF_{\Dpik_i}^0(E[\epsilon]/\epsilon^2)$ be arbitrary. Applying Theorem \ref{defor1thdefor} successively, we obtain a deformation $\widetilde{\Dpik}$ of $\Dpik$ over $\cR_{E[\epsilon]/\epsilon^2,L}$ that admits an $\omepik$-filtration $\widetilde{\cF}$ with parameter $((x_{\pi,i})_{1\leq i\leq k},(\bm{\delta}_{E[\epsilon]/\epsilon^2,i})_{1\leq i\leq k})$ and $\gr^{\widetilde{\cF}}_{i}\widetilde{\Dpik}=\widetilde\Dpik_i$ for $1\leq i\leq k$.\;By the proof of \cite[Proposition 4.1.4]{Ding2021}, we obtain
\[\dim_E(\omega_{\Dpik})^{-1}((\psi_i)_{1\leq i\leq k}))=k\cdot d_L\frac{r(r-1)}{2}= d_L\frac{n(r-1)}{2}.\]
Thus $\dim_E\im\Upsilon^0_{\Dpik}(E[\epsilon]/\epsilon^2)= d_L\frac{n(r-1)}{2}+\dim_E\homo(L^\times,E)+\dim_E\sL(\Dpik)=1+d_L(k+\frac{n(r-1)}{2})$.
\end{proof}

Define
\[\EndO_{\cF}(\Dpik,\Dpik):=\{f\in \EndO(\Dpik,\Dpik)\;|\;f(\fil_i^{\cF}\Dpik)\subseteq \fil_i^{\cF}\Dpik\},\]
This is a saturated $(\varphi,\Gamma)$-submodule of $\EndO(D,D)$.\;Each $f\in \EndO_{\cF}(\Dpik,\Dpik)$ satisfies $f(\Dpik_i)\subset \Dpik_i$ for all $1\leq i\leq k$, and therefore induces an element $(f_i)_{1\leq i\leq k}\in \prod_{i=1}^k\EndO(\Dpik_i)$.\;The following proposition computes the cohomology of $\EndO_{\cF}(\Dpik,\Dpik)$ and describes properties of the functor $F_{\Dpik,\cF}$.

\begin{pro}\label{dimlemmaDpik}
(1) We have $\hH^0_{(\varphi,\Gamma)}(\EndO(\Dpik,\Dpik))\cong\hH^0_{(\varphi,\Gamma)}(\EndO_{\cF}(\Dpik,\Dpik))\cong E$.\;\\
(2) We have  $\hH^2_{(\varphi,\Gamma)}(\EndO_{\cF}(\Dpik,\Dpik))=0$.\;\\
(3) The functor  $F_{\Dpik,\cF}$ is pro-representable and is formally smooth over $E$ of dimension  $1+d_L\frac{n(n+r)}{2}$.\;
\end{pro}
\begin{proof}For (1), it suffices to show that $\hH^0_{(\varphi,\Gamma)}(\EndO(\Dpik,\Dpik))\cong E$.\;We have an exact sequence of $(\varphi,\Gamma)$-modules over $\cR_{E,L}$:
\[0\rightarrow \homo_{\cR_{E,L}}(\Dpik_{k},\Dpik)\rightarrow \EndO(\Dpik,\Dpik)\rightarrow \homo_{\cR_{E,L}}(\bD^{k-1}_1,\Dpik) \rightarrow0.\]
Taking cohomology gives
\[0\rightarrow \homo_{(\varphi,\Gamma)}(\Dpik_{k},\Dpik)\rightarrow  \hH^0_{(\varphi,\Gamma)}(\EndO(\Dpik,\Dpik)) \rightarrow  \homo_{(\varphi,\Gamma)}(\bD^{k-1}_1,\Dpik).\]
Since $\Dpik^k_{k-1}$ is nonsplit, we have $\homo_{(\varphi,\Gamma)}(\Dpik_{k},\Dpik^k_{k-1})=0$.\;Lemma \ref{phigammacohononcrtialspecial} and a d\'{e}vissage argument show that the natural map $\homo_{(\varphi,\Gamma)}(\Dpik_{k},\Dpik^k_{k-1})\xrightarrow{\sim}\homo_{(\varphi,\Gamma)}(\Dpik_{k},\Dpik)$ is an isomorphism; hence $\homo_{(\varphi,\Gamma)}(\Dpik_{k},\Dpik)=0$.\;The same lemma and argument show that $\homo_{(\varphi,\Gamma)}(\bD^{k-1}_1,\Dpik_{k})=0$.\;Consequently, \[\homo_{(\varphi,\Gamma)}(\bD^{k-1}_1,\Dpik)\cong \hH^0_{(\varphi,\Gamma)}(\EndO(\bD^{k-1}_1,\bD^{k-1}_1)).\] The induction hypothesis now gives (1).\;By \cite[Proposition 3.4]{chenevier2011infinite}, the functor $F_{\Dpik,\cF}$ is pro-representable. For (2), consider the following natural exact sequence of $(\varphi,\Gamma)$-modules over $\cR_{E,L}$:
\[0\rightarrow \homo_{\cR_{E,L}}(\Dpik_k,\Dpik)\rightarrow \EndO_{\cF}(\Dpik,\Dpik)\rightarrow \EndO_{\cF^{k-1}_1}(\Dpik^{k-1}_1,\Dpik^{k-1}_1) \rightarrow0,\;\]
Applying the same argument to $\Dpik_1^s$ for $1\leq s\leq k-1$ shows that $\EndO_{\cF}(\Dpik,\Dpik)$ is a successive extension of the modules $\homo_{\cR_{E,L}}(\Dpik_s,\Dpik_1^s)$ for $1\leq s\leq k$.\;It therefore has rank $\frac{k(k+1)}{2}r^2$.\;The vanishing condition in Definition \ref{dfnnoncriticalspecial} gives $\hH^2_{(\varphi,\Gamma)}(\EndO_{\cF}(\Dpik,\Dpik))=0$.\;Finally, (1) and \cite[Proposition 3.4]{chenevier2011infinite} imply pro-representability, while (2) and \cite[Proposition 4.1.15(2)]{Ding2021} imply formal smoothness of dimension $1+\frac{k(k+1)}{2}d_Lr^2$.
\end{proof}
\begin{rmk}\label{splitcondeplain}We further discuss the vanishing condition $	\ext^2_{(\varphi,\Gamma)}(\Dpik_{s},\Dpik^{s}_1)=0$.\;For $2\leq s\leq k$, Tate duality, Lemma \ref{phigammacohononcrtialspecial}, and d\'{e}vissage give
	\begin{equation*}
		\begin{aligned}
			\ext^2_{(\varphi,\Gamma)}(\Dpik_{s},\Dpik^{s}_1)&\cong \hH^0_{(\varphi,\Gamma)}(\Dpik^{s}_1,\Dpik_{s}(\ccyc))\cong \hH^0_{(\varphi,\Gamma)}(\Dpik^{s}_{s-1},\Dpik_{s}(\ccyc))\cong \hH^2_{(\varphi,\Gamma)}(\Dpik_{s},\Dpik^{s}_{s-1}).
		\end{aligned}
	\end{equation*}
	Let $M_1:=\EndO(\Dpik_1)\otimes_{\cR_{E,L}}\cR_{E,L}(\ccyc)$ and, for $2\leq s\leq k$, let $M_s:=(\Dpik_{s-1}^{s})^\vee\otimes_{\cR_{E,L}}\Dpik_s\otimes_{\cR_{E,L}}\cR_{E,L}(\ccyc)$.\;Then $M_s$ is an extension of $M_s'$ by $\EndO(\Dpik_s)\otimes\cR_{E,L}(\ccyc)$,\;here $M_s':=\homo(\Dpik_{s-1},\Dpik_s)\otimes\cR_{E,L}(\ccyc)$.\;Suppose that $\hH^0_{(\varphi,\Gamma)}(M_s)\neq0$,\;then there is a nonzero injection $j:\cR_{E,L}\rightarrow M_s$.\;Since
	$\hH^0_{(\varphi,\Gamma)}(\EndO(\Dpik_s)\otimes_{\cR_{E,L}}\cR_{E,L}(\ccyc)[{1}/{t}])=\hH^0_{(\varphi,\Gamma)}(\EndO({\Delta_x})\otimes_{\cR_{E,L}}\cR_{E,L}(\ccyc)[{1}/{t}])=0$,\;we see that $j$ gives a nonzero injection $j_1:\cR_{E,L}\rightarrow M_s'$ and $\mathrm{Im}(j_1)[{1}/{t}]\cap M_s$ is a saturated  $(\varphi,\Gamma)$-submodule of $M_s'$.\;The injections $\bI_j$ and $\bI'_{j}$ for $j\in\{s-1,s\}$ induces injections: 
	\begin{equation*}
		\begin{aligned}
			j':M_s'\hookrightarrow \EndO_x(z^{\hpi_{sr}-\hpi_{(s-1)r+1}+1}),\quad
			j'':\EndO_x(z^{\hpi_{(s-1)r+1}-\hpi_{(s-1)r}+1})\hookrightarrow M_s'.\;
		\end{aligned}
	\end{equation*}
	which induce isomorphisms (similar to Lemma \ref{pullbackforiso})
	\begin{equation*}
		\begin{aligned}
			&\hH^0_{(\varphi,\Gamma)}(M_s')\xrightarrow{\sim}\hH^0_{(\varphi,\Gamma)}(\EndO_x(z^{\hpi_{sr}-\hpi_{(s-2)r+1}+1})),\\
			&\hH^0_{(\varphi,\Gamma)}(\EndO_x(z^{\hpi_{(s-1)r+1}-\hpi_{(s-1)r}+1}))\xrightarrow{\sim}\hH^0_{(\varphi,\Gamma)}(M_s').\;
		\end{aligned}
	\end{equation*}
	Since 
	$\hH^0_{(\varphi,\Gamma)}(\mathrm{Ad}_x(z^{\hpi_{sr}-\hpi_{(s-2)r+1}+1}))=0$,\;by comparing Hodge-Tate weights,\;we see that 
	$\mathrm{Im}(j'\circ j_1)[1/t]\cap \EndO_x(z^{\hpi_{sr}-\hpi_{(s-2)r+1}+1})=\cR_{E,L}(z^{\hpi_{sr}-\hpi_{(s-2)r+1}+1})$
	and thus \[\cR_{E,L}(z^{\hpi_{(s-1)r+1}-\hpi_{(s-1)r}+1})\hookrightarrow\mathrm{Im}(j)[1/t]\cap M_s\hookrightarrow \cR_{E,L}(z^{\hpi_{sr}-\hpi_{(s-2)r+1}+1}).\]This implies that $\homo_{(\varphi,\Gamma)}(M_s',M_s)\neq0$,\;which is impossible if $r=1$ and $\Dpik_i^{i+1}$ is non-split for every $1\leq i\leq n$.\;When $r>1$,\;the vanishing of $\ext^2_{(\varphi,\Gamma)}(\Dpik_{s},\Dpik^{s}_1)$ also indicates certain non-split property.\;
\end{rmk}

In the remainder of this section, we compute the dimension of the tangent space $F_{\Dpik,\cF}^0(E[\epsilon]/\epsilon^2)$.\;

\begin{lem}\label{fibredim}Let $V,V'$, and $V''$ be finite-dimensional $E$-vector spaces, and let $u:V\rightarrow V''$ and $u':V'\rightarrow V''$ be linear maps.\;Then
\[\dim_E(V\times_{V''}V')=\dim_EV-\dim_E\im u+\dim_E(u')^{-1}(\im u\cap\im u'),\]
where $V\times_{V''}V'=\{(v,v')\in V\times V':u(v)=u'(v')\in V''\}$.\;Moreover,\;\[\dim_E V\times_{V''}V'=\dim_EV+\dim_EV'-\dim_EV''\] if and only if $\dim_E\mathrm{coker}(u)=\dim_E\mathrm{coker}(V\times_{V''}V'\rightarrow V')$.\;In particular,\;if one of $u,u'$ is surjective,\;we have $\dim_E V\times_{V''}V'=\dim_EV+\dim_EV'-\dim_EV''$.\;
\end{lem}

Lemma \ref{fibredim}, Corollary \ref{corlinvaraint1}, and Lemma \ref{dimlemmaDpik} give the following result.
\begin{cor}\label{corlinvaraint2}The $E$-dimension of the cokernel of  $\Upsilon:F_{\Dpik,\cF}(E[\epsilon]/\epsilon^2)\longrightarrow \prod_{i=1}^kF_{\Dpik_i}(E[\epsilon]/\epsilon^2)$ is $k-1$.\;
\end{cor}
\begin{proof}Induction on $k$ gives the following commutative diagram of $(\varphi,\Gamma)$-modules over $\cR_{E,L}$:
\[\xymatrix{0 \ar[r] & \homo_{\cR_{E,L}}(\Dpik_k,\Dpik)\ar[r] \ar[d] & \EndO_{\cF}(\Dpik)\ar[r] \ar[d] & \EndO_{\cF^{k-1}_1}(\Dpik^{k-1}_1) \ar[r]\ar[d]  &0\\
	0 \ar[r] & \EndO(\Dpik_k) \ar[r] & \prod_{i=1}^k\EndO(\Dpik_i)  \ar[r] & \prod_{i=1}^{k-1}\EndO(\Dpik_i) \ar[r] & 0  .}\]
Taking cohomology,\;we deduce from Proposition \ref{dimlemmaDpik} (and the proof of Proposition \ref{dimlemmaDpik}  (2)) the following long exact sequences
\[\xymatrix{ 
	0 \ar[r] & \ext^0_{(\varphi,\Gamma)}(\Dpik_k,\Dpik)=0 \ar[r] \ar[d] &  	\ext^0_{(\varphi,\Gamma)}(\EndO_{\cF}(\Dpik))\cong E \ar[d]   \\
	0 \ar[r] & \ext^0_{(\varphi,\Gamma)}(\Dpik_k,\Dpik_k)\cong E \ar[r] & \prod\limits_{i=1}^{k}\hH^0_{(\varphi,\Gamma)}( \EndO(\Dpik_i))\cong E^{k} }\]
\begin{equation}\label{corgamma}
	\begin{aligned}
		&\hspace{30pt}\xymatrix{ 
			\ar[r]^{\sim\hspace{60pt}} &\ext^0_{(\varphi,\Gamma)}(\EndO_{\cF}(\Dpik^{k-1}_1))\cong E \ar[r] \ar[d] &  \ext^1_{(\varphi,\Gamma)}(\Dpik_k,\Dpik) \ar[d]^{\Upsilon''}  \ar[r] &  F_{\Dpik,\cF}(E[\epsilon]/\epsilon^2) \ar[d]^{\Upsilon=\Upsilon_k}  \\
			\ar[r] &\prod\limits_{i=1}^{k-1}\hH^0_{(\varphi,\Gamma)}( \EndO(\Dpik_i))\cong E^{k-1} \ar[r] & \ext^1_{(\varphi,\Gamma)}(\Dpik_k,\Dpik_k) \ar[r] & \prod\limits_{i=1}^{k}F_{\Dpik_i}(E[\epsilon]/\epsilon^2)  }
	\end{aligned}
\end{equation}	
\[\hspace{130pt}\xymatrix{
	\ar[r]& F_{\Dpik^k_1,\cF^{k-1}_1}(E[\epsilon]/\epsilon^2) \ar[r] \ar[d]^{\Upsilon_{k-1}} &  \ext^2_{(\varphi,\Gamma)}(\Dpik_k,\Dpik)=0 \ar[d]  \\
	\ar[r]&\prod\limits_{i=1}^{k-1}F_{\Dpik_i}(E[\epsilon]/\epsilon^2) \ar[r] & \ext^2_{(\varphi,\Gamma)}(\Dpik_k,\Dpik_k)=0 .}\]

Comparing dimensions shows that the three middle terms involving the $\ext^1$-groups in the two rows of (\ref{corgamma}) form two short exact sequences.\;An easy d\'{e}vissage argument shows that $\Upsilon''$ is injective and $\dim_E\mathrm{coker}\Upsilon''=1$.\;The snake lemma then gives
\[0\rightarrow \ker\Upsilon_k \rightarrow \ker\Upsilon_{k-1}\rightarrow \mathrm{coker}\Upsilon''\rightarrow \mathrm{coker}\Upsilon_k \rightarrow \mathrm{coker}\Upsilon_{k-1}\rightarrow 0.\]
Note that $\Upsilon_1$ is an isomorphism, so $\ker\Upsilon_1=\mathrm{coker}\Upsilon_1=0$.\;Induction on $k$ now shows that $\ker\Upsilon_k\cong \ker\Upsilon_{k-1}\cong\cdots\cong \ker\Upsilon_1=0$ and $\dim_E\mathrm{coker}\Upsilon_k=1+\dim_E\mathrm{coker}\Upsilon_{k-1}=\cdots=k-1$.\;This completes the proof.\;
\end{proof}

\begin{pro}\label{dimtanFD0}	The functor $F_{\Dpik,\cF}^0$ is  pro-representable and $\dim_EF_{\Dpik,\cF}^0(E[\epsilon]/\epsilon^2)=1+d_L\Big(k+\frac{n(n-1)}{2}\Big)$.\;
\end{pro}
\begin{proof}By \cite[Proposition 4.1.3]{Ding2021},\;the functor $\prod_{i=1}^k F_{\Dpik_i}^0$ is relatively representable over $\prod_{i=1}^k F_{\Dpik_i}$. By Proposition \ref{dimlemmaDpik} (3),\;the functor $F_{\Dpik,\cF}^0$ is pro-representable.\;By \cite[Proposition 4.1.4]{Ding2021},\;the functor $\prod_{i=1}^k F_{\Dpik_i}^0$ is formally smooth of dimension $k\big(1+d_L\big(1+\frac{r(r-1)}{2}\big)\big)$.\;It follows from Lemma \ref{fibredim}, Corollary \ref{corlinvaraint1}, Proposition \ref{dimlemmaDpik}, \cite[Proposition 4.1.4]{Ding2021}, and Corollary \ref{corlinvaraint2} that
\begin{equation}
	\begin{aligned}
		\dim_EF_{\Dpik,\cF}^0(E[\epsilon]/\epsilon^2)
		=&\;\dim_E\sF_{\Dpik,\cF}(E[\epsilon]/\epsilon^2)-\sum_{i=1}^k\dim_EF_{\Dpik_i}(E[\epsilon]/\epsilon^2)+\sum_{i=1}^k\dim_EF^0_{\Dpik_i}(E[\epsilon]/\epsilon^2)\\
		=\;&1+d_L\frac{n(n+r)}{2}-k(1+d_Lr^2)+k\Big(1+d_L\Big(1+\frac{r(r-1)}{2}\Big)\Big),\\
		=\;&1+d_L\Big(k+\frac{n(n-1)}{2}\Big).\\
\end{aligned}\end{equation}
The result follows.\;
\end{proof}

\section{Parabolic Breuil's simple {$\sL$}-invariants}\label{BreuilLINVSUMMMANY}

Let $\ul{\lambda}:=(\lambda_{1,\sigma},\ldots,\lambda_{n,\sigma})_{\sigma\in\Sigma_L}$ be a weight of $\ft_{L}$. For $I\subseteq\Delta=\{1,\ldots,n-1\}$, we say that $\ul{\lambda}$ is $I$-dominant with respect to $\res_{L/\bQ_p}\bB$ if $\lambda_{i,\sigma}\geq\lambda_{i+1,\sigma}$ for every $i\in I$ and $\sigma\in\Sigma_L$.\;Let $X_I^+$  denote the set of integral weights that are $I$-dominant.\;For $\ul{\lambda}\in X_I^+$, let $L(\ul{\lambda})_I$ be the unique irreducible algebraic representation of $\res_{L/\bQ_p}\bL_I$ of highest weight $\ul{\lambda}$ with respect to $\res_{L/\bQ_p}(\bL_I\cap\bB)$.\;Put $\overline{L}(-\ul{\lambda})_I:=L(\ul{\lambda})_I^\vee$; this is an irreducible algebraic representation of $\res_{L/\bQ_p}\bL_I$ of highest weight $-\ul{\lambda}$ with respect to $\res_{L/\bQ_p}(\bL_I\cap\ob)$. Write $\chi_{\ul{\lambda}}:=L(\ul{\lambda})_{\emptyset}$. If $\ul{\lambda}\in X_{\Delta}^+$, set $L(\ul{\lambda}):=L(\ul{\lambda})_{\Delta}$. By a $\bQ_p$-algebraic representation of $G$ over $E$, we mean the restriction to $G\subset(\res_{L/\bQ_p}\GLN_n)(\bQ_p)$ of an algebraic representation of $\res_{L/\bQ_p}\GLN_n$.\;For an integral weight $\ul{\lambda}$, let $M(\ul{\lambda}):=\text{U}(\fg_{L})\otimes_{\text{U}(\fb_{L})}\ul{\lambda}$ (resp. $\overline{M}(\ul{\lambda}):=\text{U}(\fg_{L})\otimes_{\text{U}(\overline{\fb}_{L})}\ul{\lambda}$) be the corresponding Verma module with respect to $\fb_{L}$ (resp. $\overline{\fb}_{L}$). Let $L(\ul{\lambda})$ (resp. $\overline{L}(\ul{\lambda})$) be the unique simple quotient of $M(\ul{\lambda})$ (resp. $\overline{M}(\ul{\lambda})$). If $\ul{\lambda}\in X_{\Delta}^+$, then $L(\ul{\lambda})$ is finite-dimensional and isomorphic to the algebraic representation introduced above, so the notation is unambiguous.\;Moreover,\;$\overline{L}(-\ul{\lambda})\cong L(\ul{\lambda})^\vee$.\;For $I\subseteq\Delta$ and $\ul{\lambda}\in X_I^+$, define the generalized parabolic Verma modules
\begin{equation}
\begin{aligned}
	M_I(\ul{\lambda}):=\text{U}(\fg_{L})\otimes_{\text{U}(\fp_{I,\Sigma_L})} L(\ul{\lambda})_I,\;\quad\text{resp.,\;}\overline{M}_I(-\ul{\lambda}):=\text{U}(\fg_{L})\otimes_{\text{U}(\overline{\fp}_{I,\Sigma_L})} \overline{L}(-\ul{\lambda})_I)
\end{aligned}
\end{equation}
with respect to $\fp_{I,\Sigma_L}$ and $\overline{\fp}_{I,\Sigma_L}$, respectively; see \cite[Chapter 9]{humphreysBGG}.

For $I\subseteq\Delta$, recall the Orlik--Strauch functor $\cF^G_{\op_I}(-,-)$ (see \cite[Main Theorem]{orlik2015jordan} and \cite[Section 2]{breuil2016socle}). It associates to an object $M$ of the Bernstein--Gelfand--Gelfand category $\cO_{\alge}^{\overline{\fp}_{I,L}}$ and a finite-length smooth admissible representation ${\pi}_{\beta,I}$ the locally $\bQ_p$-analytic representation $\cF^G_{\op_I}(M,{\pi}_{\beta,I})$.

\subsection{Locally $\bQ$-analytic generalized Steinberg representations}\label{reviewlocallyanaST}

Following \cite{av1980induced2}, let
\[\bm\beta:=\{{\pi}\nu^{k-1+\frac{1-n}{2}},\;\cdots,\;{\pi}\nu,{\pi}\nu^{\frac{1-n}{2}}\}\]
be the normalized Zelevinsky segment of length $k$ associated with the fixed absolutely irreducible smooth cuspidal representation ${\pi}\in\Omega_r$.\;Let $\nu$ be the character of $\GLN_r(L)$ defined by $g\mapsto |\det(g)|_L$.\;The segment $\bm\beta$ determines an irreducible smooth cuspidal representation of $\bL_{\beta}(L)$ over $E$:
\begin{equation}\label{pilrr}
	\pi_{\beta}:=\big(\boxtimes_{i=1}^{k}{\pi}\nu^{k-i+\frac{1-n}{2}}\big)\otimes_E \delta_{\op_{\beta}}^{1/2}=\boxtimes_{i=1}^{k}{\pi}\nu^{-\frac{r}{2}(k-2i+1)+k-i+\frac{1-n}{2}},
\end{equation}
where $\delta_{\op_{\beta}}$ is the modulus character of $\op_{\beta}$.\;Let $I\subseteq\Delta_\beta$.\;By \cite[Proposition 2.10]{av1980induced2},\;the smooth parabolic induction $\big(\mathrm{Ind}_{\op_\beta(L)\cap\bL_{\beta,I}(L)}^{\bL_{\beta,I}(L)}\pi_\beta\big)^\infty$ admits a unique irreducible subrepresentation, denoted by ${\pi}_{\beta,I}$.\;

Let $\underline{\lambda}\in X_{\Delta}^+$, and let $J\subset I\subseteq \Delta_{\beta}$.\;Put
\begin{equation*}
	\begin{aligned}
		&i_{I,J}(\bm\beta,\underline{\lambda})
		:=\left(\mathrm{Ind}_{\op_{\beta,{J}}(L)\cap \bL_{\beta,I}(L)}^{\bL_{\beta,I}(L)}\pi_{\beta,J}\right)^{\infty}\otimes_EL(\underline{\lambda}).,\\
		&\bI_{I,J}(\bm\beta,\underline{\lambda})
		:=\left(\mathrm{Ind}_{\op_{\beta,{J}}(L)\cap \bL_{\beta,I}(L)}^{\bL_{\beta,I}(L)}\pi_{\beta,J}\otimes_EL(\underline{\lambda})_{S^{\beta}(J)}\right)^{\bQ_p-\ana}.\;
	\end{aligned}
\end{equation*}
Note that $i_{I,J}(\bm\beta,\underline{\lambda})$ is the subspace of locally $\bQ_p$-algebraic vectors in $\bI_{I,J}(\bm\beta,\underline{\lambda})$.\;As shown in \cite[Section 4.2, Proposition 3.1]{2022ext1hyq}, if $J\subsetneq K$, there are injections of $G$-representations $\bI_{I,K}(\bm\beta,\underline{\lambda})\hookrightarrow\bI_{I,J}(\bm\beta,\underline{\lambda})$ and $i_{I,K}(\bm\beta,\underline{\lambda})\hookrightarrow i_{I,J}(\bm\beta,\underline{\lambda})$.\;

The locally $\bQ_p$-algebraic (resp.\;$\bQ_p$-analytic) generalized Steinberg representations of $\bL_{\beta,I}(L)$ associated with the Zelevinsky segment $\bm\beta$ and the weight $\ul{\lambda}$ are defined by
\begin{equation*}
	\begin{aligned}
		&v_{J,I}^{\lalg}(\bm\beta,\underline{\lambda})=i_{I,J}(\bm\beta,\underline{\lambda})\big/\sum_{J \subsetneq K\subseteq I}i_{I,K}(\bm\beta,\underline{\lambda}),\\
	\text{resp.,\;}	&	v_{J,I}^{\ana}(\bm\beta,\underline{\lambda})=\bI_{I,J}(\bm\beta,\underline{\lambda})\big/\sum_{J \subsetneq K\subseteq I}\bI_{I,K}(\bm\beta,\underline{\lambda}).\;
	\end{aligned}
\end{equation*}
We call $\st_I^{\lalg}(\bm\beta,\ul{\lambda}):=v_{\emptyset,I}^{\lalg}(\bm\beta,\underline{\lambda})$ (resp.,\;$\st_I^{\ana}(\bm\beta,\ul{\lambda}):=v_{\emptyset,I}^{\ana}(\bm\beta,\underline{\lambda})$) the locally $\bQ_p$-algebraic (resp.,\;$\bQ_p$-analytic) Steinberg representation of $\bL_{\beta,I}(L)$.\;We define $\st_I^{\infty}(\bm\beta):=v_{\emptyset,I}^{\lalg}(\bm\beta,\underline{0})$ to be the smooth Steinberg representation of $\bL_{\beta,I}(L)$.\;When $J\neq \emptyset$ and $I=\Delta_{\beta}$, we omit $I$ from the notation.\;If $J=\emptyset$ and $I=\Delta_{\beta}$, write $\st_n^{\lalg}(\bm\beta,\ul{\lambda}):=\st_{\Delta_{\beta}}^{\lalg}(\bm\beta,\ul{\lambda})$ (resp.,\;$\st_n^{\ana}(\bm\beta,\ul{\lambda}):=\st_{\Delta_{\beta}}^{\ana}(\bm\beta,\ul{\lambda})$) for the locally $\bQ_p$-algebraic (resp.,\;$\bQ_p$-analytic) Steinberg representation of $G$.\;

Following \cite[Section 5.3]{2022ext1hyq},\;we recall some subrepresentations of $\st_n^{\ana}(\bm\beta,\ul{\lambda})$.\;Fix $1\leq i\leq k-1$.\;Let $\widetilde{\Sigma}_{ir}(\bm\beta,\ul{\lambda})$ be the image of the composition
\begin{equation}\label{tautauI}
	\tau_{ir}: \Big(\mathrm{Ind}_{\op_{\Delta\backslash\{ir\}}(L)}^{G}\st_{\Delta_{\beta}\backslash\{ir\}}^\lalg(\bm\beta,\underline{\lambda})\Big)^{\bQ_p-\mathrm{an}} \hooklongrightarrow \Big(\mathrm{Ind}_{\op_{\Delta\backslash\{ir\}}(L)}^{G}\st_{\Delta_{\beta}\backslash\{ir\}}^{\ana}(\bm\beta,\underline{\lambda})\Big)^{\bQ_p-\mathrm{an}} \twoheadlongrightarrow \mathrm{St}_n^{\ana}(\bm\beta,\underline{\lambda}).\;
\end{equation}
The statement and proof of \cite[Lemma 5.12]{2022ext1hyq} show that
\begin{equation}\label{scolesumwidetildeidentity}
	\begin{aligned}
		\soc_G \big(\ind_{\op_{\Delta\backslash\{ir\}}(L)}^G \st^{\lalg}_{\Delta_{\beta}\backslash\{ir\}}(\bm\beta,\ul{\lambda})\big)^{\bQ_p-\ana}   
		\cong v_{ir}^{\lalg}(\bm\beta,\ul{\lambda}).
	\end{aligned}
\end{equation}
and $\widetilde{\Sigma}_{ir}(\bm\beta,\ul{\lambda})=\big(\ind_{\op_{\Delta\backslash\{ir\}}(L)}^G \st^{\lalg}_{\Delta_{\beta}\backslash\{ir\}}(\bm\beta,\ul{\lambda})\big)^{\bQ_p-\ana}  /v_{ir}^{\lalg}(\bm\beta,\ul{\lambda})$.

Moreover,\;for each $\sigma\in \Sigma_L$,\;write $\ul{\lambda}_{\sigma}:=(\lambda_{1,\sigma}, \cdots, \lambda_{n,\sigma})$ and $\ul{\lambda}^{\sigma}:=(\lambda_{1,\sigma'}, \cdots, \lambda_{n,\sigma'})_{\sigma'\in \Sigma_L\backslash\{\sigma\}}$. Set
\begin{equation*}
	\begin{aligned}
			& \widetilde{\Sigma}_{ir,\sigma}(\bm\beta,\underline{\lambda}):=\Big(\big(\ind_{\op_{\Delta_{\beta}\backslash\{ir\}}(L)}^G \st^{\lalg}_{\Delta_{\beta}\backslash\{ir\}}(\bm\beta,\ul{\lambda}_{\sigma}) )^{\sigma-\ana}\otimes_E L(\ul{\lambda}^{\sigma})\Big)/v_{ir}^{\lalg}(\bm\beta,\underline{\lambda}),\\
		&	C_{ir,\sigma}(\bm\beta):=\cF_{\op_{\Delta\backslash\{ir\}}}^G\big(\overline{L}(-\ul{\lambda}^{\sigma})\otimes_E \overline{L}(-s_{ir,\sigma}\cdot \ul{\lambda}_{\sigma}), \st_{\Delta_{\beta}\backslash\{ir\}}^{\infty}(\bm\beta)\big),
	\end{aligned}
\end{equation*}
By \cite[Propositions 4.6 \& (5.34)]{2022ext1hyq}, $C_{ir,\sigma}(\bm\beta)=\soc_G\big(\widetilde{\Sigma}_{ir,\sigma}(\bm\beta,\underline{\lambda})/\mathrm{St}_n^{\lalg}(\bm\beta,\underline{\lambda})\big)$,\;and it occurs with multiplicity one as an irreducible constituent of $\st_n^{\ana}(\bm\beta,\underline{\lambda})$.\;Then $\widetilde{\Sigma}_{ir}(\bm\beta,\underline{\lambda})$ admits a subrepresentation $\Sigma_{ir,\sigma}(\bm\beta,\underline{\lambda})$ (resp.,\;$\Sigma_{ir}(\bm\beta,\underline{\lambda})$) that is an extension of $C_{ir,\sigma}(\bm\beta)$ (resp.,\;$\bigoplus_{\sigma\in \Sigma_L}C_{ir,\sigma}(\bm\beta)$) by $\st_n^{\lalg}(\bm\beta,\underline{\lambda})$.\;Set
\begin{equation}\label{subofsigma}
	\begin{aligned}
		& \Sigma(\bm\beta,\underline{\lambda}):=\bigoplus^{1\leq i\leq k-1}_{\st_n^{\lalg}(\bm\beta,\underline{\lambda})} \Sigma_{ir}(\bm\beta,\underline{\lambda})
		\hooklongrightarrow\widetilde{\Sigma}(\bm\beta,\underline{\lambda}):=\bigoplus^{1\leq i\leq k-1}_{\st_n^{\lalg}(\bm\beta,\underline{\lambda})}\widetilde{\Sigma}_{ir}(\bm\beta,\underline{\lambda})
	\end{aligned}
\end{equation}
are subrepresentations  of $\st_n^{\ana}(\bm\beta,\underline{\lambda})$.\;

\begin{pro}\label{scolesumwidetilde}
	$\soc_G(\widetilde{\Sigma}_{i}(\bm\beta,\ul{\lambda}))\cong \mathrm{St}_n^{\lalg}(\bm\beta,\underline{\lambda})$.
\end{pro}
\begin{proof}We adapt the proof of \cite[Proposition 2.23]{2019DINGSimple} to our setting.\;Suppose that an irreducible constituent $W$ of $\widetilde{\Sigma}_{i}(\bm\beta,\ul{\lambda})/\st_n^{\infty}(\bm\beta,\ul{\lambda})$ embeds into $\soc_G(\widetilde{\Sigma}_{i}(\bm\beta,\ul{\lambda}))$. The pullback $V$ of $\big(\ind_{\op_{\Delta\backslash\{ir\}}(L)}^G \st^{\lalg}_{\Delta_{\beta}\backslash\{ir\}}(\bm\beta,\ul{\lambda})\big)^{\bQ_p-\ana}$ along this embedding is an extension of $W$ by $v^{\infty}_I(\bm\beta,\ul{\lambda})$.\;Thus, $V$ fits into the following commutative diagram:
	\[\xymatrix{0 \ar[r] & v_{ir}^{\infty}(\bm\beta,\underline{\lambda}) \ar@{=}[d]  \ar[r] & \hspace{20pt}V_{\hspace{0.2pt}}\hspace{20pt} \ar@{^(->}[d] \ar[r] & \hspace{16pt}W_{\hspace{0.2pt}}\hspace{16pt} \ar@{^(->}[d] \ar[r] & 0  \\
		0 \ar[r] & v_{ir}^{\infty}(\bm\beta,\underline{\lambda})  \ar[r] & \Big(\mathrm{Ind}_{\op_{\Delta\backslash\{ir\}}(L)}^{G}\st_{\Delta_{\beta}\backslash\{ir\}}^\infty(\bm\beta,\underline{\lambda})\Big)^{\bQ_p-\mathrm{an}} \ar[r] & \widetilde{\Sigma}_{i}(\bm\beta,\underline{\lambda})\ar[r] & 0.}\]
		By (\ref{scolesumwidetildeidentity}),\;the first row of this diagram is nonsplit.\;We claim that $V$ is very strongly admissible.\;A straightforward adaptation of the proof of \cite[Lemma 2.24]{2019DINGSimple} then shows that $V$ splits.\;We briefly indicate the changes below.\;Only part $(e)$ of the proof of \cite[Lemma 2.4]{2019DINGSimple} requires further modification.\;In our case, we need to show that
	\begin{equation}\label{equ: lgln-dJE}
		\homo_{{\bL^{\lrr,+}_{\Delta_{\beta}\backslash\{ir\}}}}\Big(\st_{\Delta_{\beta}\backslash\{ir\}}^{\infty}(\bm\beta) \otimes_E \delta_{\op_{\Delta\backslash\{ir\}(L)}}\otimes_E L(\mu)_{\Delta_{\beta}\backslash\{ir\}},\hH^1_{\ana}\big({\bN}_0(L), v^{\infty}_I(\bm\beta,\ul{\lambda})\big)\Big)=0, 
	\end{equation}
		where ${\bN}_0(L)$ is a compact open subgroup of $\bN_{\Delta_{\beta}\backslash\{ir\}}$, and ${\bL^{\lrr,+}_{\Delta_{\beta}\backslash\{ir\}}}:=\{z\in \bL_{\Delta\backslash\{ir\}}: z{\bN}_0(L)z^{-1}\subset {\bN}_0(L)\}$.\;If this vector space is nonzero, the argument following \cite[(55)]{2019DINGSimple}, together with the adjunction property (\ref{smoothadj2}), implies $\homo_{G}\big(\big(\ind_{\op_{\Delta\backslash\{ir\}}(L)}^G \st_{\Delta_{\beta}\backslash\{ir\}}^{\infty}(\bm\beta)\big)^{\infty}, v_I^{\infty}(\bm\beta)\big) \neq 0$.\;This gives $\homo_G\big(i_{\emptyset}(\bm\beta), v_I^{\infty}(\bm\beta)\big)\neq 0$, which contradicts \cite[Lemma 7.3]{2022ext1hyq}.\;The proposition follows.\;
	
		To show that $V$ is very strongly admissible,\;it suffices to find a continuous $G$-equivariant $E$-linear injection $V \rightarrow V_1$, where $V_1$ is an admissible continuous representation of $G$ on an $E$-Banach space (see \cite[Proposition-Definition 6.2.3]{emerton2017locally} for the definition of an admissible continuous representation).\;Indeed,\;by \cite[Proposition 2.1.2]{emerton2007jacquet} and the inclusion $V\hookrightarrow (\ind_{\op_{\Delta\backslash\{ir\}}(L)}^G \st^{\lalg}_{\Delta_{\beta}\backslash\{ir\}}(\bm\beta,\ul{\lambda}))^{\bQ_p-\ana}$, it suffices to show that $\st^{\lalg}_{\Delta_{\beta}\backslash\{ir\}}(\bm\beta,\ul{\lambda})$ is very strongly admissible.\;Tensoring an admissible Banach representation with a finite-dimensional $\bQ_p$-algebraic representation of $\bL_{\Delta\backslash\{ir\}}(L)$ again gives an admissible Banach representation, so it remains to show that $\st_{\Delta_{\beta}\backslash\{ir\}}^{\infty}(\bm\beta)$ is very strongly admissible.\;This is proved in Lemma \ref{verystrongad}.
\end{proof}

In our setting, the very strong admissibility of the smooth Steinberg representation does not follow directly from the strategy used in the proof of \cite[Proposition 2.23]{2019DINGSimple}.\;We prove it using global methods, specifically Emerton's completed cohomology.\;

\begin{lem}\label{verystrongad}The smooth Steinberg representation $\st_{I}^{\infty}(\bm\beta)$ of $\bL_{\beta,I}(L)$ is very strongly admissible.\;
\end{lem}
\begin{proof}We divide into three steps.\;\\
	\textbf{Step 1.} Write $\bL_{\beta,I}(L)=\GLN_{k_1r}(L)\times \cdots\times\GLN_{k_lr}(L)$ for an ordered partition $\underline{k}_{I}:=(k_1,\cdots,k_{l})$ of the integer $k$.\;Thus, we have an isomorphism $\st_{I}^{\infty}(\bm\beta,\underline{\lambda})\cong V_1\otimes_E\cdots \otimes_E V_l$
	of smooth representations of $\bL_{\beta,I}(L)$, where each $V_i$ is a locally $\bQ_p$-algebraic Steinberg representation of $\GLN_{k_ir}(L)$ for $1\leq i\leq l$.\;By \cite[Lemma A.3]{breuil2015ordinary}, it suffices to prove that each $V_i$ is very strongly admissible.\;Therefore, we may assume that $I=\Delta_{\beta}$.\;Fix an isomorphism $\iota:\overline{\bQ}_p\xrightarrow{\sim}\BC$.\;As explained after \cite[Definition 3]{sorensen2013proof},\;after enlarging $E$ if necessary, there exists a smooth character $\eta:G\rightarrow E^\times$ such that $\iota(\st_n^{\infty}(\bm\beta)\otimes_E\eta)$ is a unitary discrete-series representation of $G$ over $\iota(E)\subset \BC$.\;Set $\cW:=\st_n^{\infty}(\bm\beta)\otimes_E\eta$ and let $\omega_{\cW}$ be its central character.\;\\
	\textbf{Step 2.} There exists a totally real number field $F^+$ and a finite place $w$ of $F^+$ such that $F^+_w=L$.\;Let $S_\infty$ be the set of infinite places of $F$.\;Let $\BG$ be a definite unitary group over $F^+$ that is compact at the places in $S_\infty$ and agrees with $\GLN_n$ at $w$.\;Then $\BG^{\mathrm{ad}}$ is a semisimple anisotropic $\BQ$-group.\;By \cite[Corollary 3, Lemma 4]{sorensen2013proof},\;there is a classical algebraic automorphic representation $\Pi$ of $\BG^{\mathrm{ad}}$ such that the local factor $\Pi_w\cong \iota(\cW\otimes_E \omega_{\cW}^{-1})$.\;By enlarging $E$ if necessary, we may realize $\iota^{-1}(\Pi)$ over $E$.\\
	\textbf{Step 3.} Let $K^w\subset \BG^{\mathrm{ad}}(\BA_{F^+}^{\infty,w})$ be a tame level such that $(\Pi^w)^{K^w}\neq 0$.\;For $K_w\subseteq \mathrm{PGL}_{n}(F^+_w)$,\;put $Y_{K_w K^w}:=\BG^{\mathrm{ad}}(F^+)\backslash\BG^{\mathrm{ad}}(\BA_{F^+}^{\infty})/A_\infty^\circ K_{\infty}^\circ K_w K^w$.\;Emerton's completed cohomology (see \cite{Em1}) is defined by
	$$\widetilde{H}^i(K^w,\cO_E):=\varprojlim_{s}\varinjlim_{K_w}H^i( Y_{K_w K^w}, \cO_E/\varpi_E^s).$$
		Then $\widetilde{H}^i(K^w,E):=\widetilde{H}^i(K^w,\cO_E) \otimes_{\co_E}E$ is an admissible Banach representation of $\mathrm{PGL}_{n}(F^+_w)$ over $E$.\;Let $\widehat{\Pi}$ be the classical $p$-adic automorphic representation of $\BG^{\mathrm{ad}}(\BA_{F^+}^{\infty})$ over $E$ attached to the automorphic form $\iota^{-1}(\Pi)$ constructed in Step 2 (see \cite[Definition 3.1.5]{Em1}). By \cite[Corollary 2.2.25, Proposition 3.2.2]{Em1}, for some algebraic representation $W$ of $\mathrm{PGL}_{n}(F^+_w)$, there is an injection
	\[(\cW\otimes_E \omega_{\cW}^{-1})\otimes W\otimes (\Pi^w)^{K^w}\hookrightarrow \widetilde{H}^0(K^w,E),\]
		Choosing a nonzero vector $v\in (\Pi^w)^{K^w}$ gives an injection $\cW\otimes_E \omega_{\cW}^{-1}\hookrightarrow \widetilde{H}^0(K^w,E)\otimes W^{\vee}$ of representations of $\mathrm{PGL}_{n}(F_w)$ over $E$.\;Twisting by the character $\eta^{-1}\omega_{\cW}$ completes the proof.\;
\end{proof}

\begin{pro}\label{socleofreps1}For any $V\in \{\widetilde{\Sigma}_{ir,\sigma}(\bm\beta,\ul{\lambda}),\Sigma_{ir}(\bm\beta,\ul{\lambda}),\Sigma(\bm\beta,\ul{\lambda})\}$, we have $\soc_{G}V=\st_n^{\lalg}(\bm\beta,\ul{\lambda}).$
\end{pro}
\begin{proof}For each $V\in \{\widetilde{\Sigma}_{ir,\sigma}(\bm\beta,\ul{\lambda}),\Sigma_{ir}(\bm\beta,\ul{\lambda})\}$, the result follows from Proposition \ref{scolesumwidetilde}.\;By \cite[Corollary 2.7]{breuil2019ext1}, the irreducible constituents $\{C_{ir,\sigma}(\bm{\beta})\}_{ir\in \Delta_{\beta}, \sigma\in \Sigma_L}$ are all distinct; hence, $\soc_{G}\Sigma_{ir}(\bm\beta,\ul{\lambda})\cong \st_n^{\lalg}(\bm\beta,\ul{\lambda})$.
\end{proof}

\subsection{Extensions between locally analytic representations}\label{reviewEXTlocallyanaST}

The following is the main theorem of \cite{2022ext1hyq} (see \cite[Theorem 5.9]{2022ext1hyq}):

\begin{pro}\label{thmintroLINV}
For $1\leq i\leq k-1$, there is an isomorphism of $(d_L+1)$-dimensional $E$-vector spaces
\begin{equation}\label{thmintro}
	\homo(L^\times,E)\xrightarrow{\sim }\ext^1_{G}\big(v_{ir}^{\lalg}(\bm\beta,\underline{\lambda}), \st_n^{\ana}(\bm\beta,\underline{\lambda})\big).
\end{equation}
\end{pro}
For $\psi\in\homo(L^{\times},\;E)$,\;we construct its image under (\ref{thmintro}) explicitly.\;The isomorphism is
\begin{equation}
\begin{aligned}
	\iota_\upsilon: \homo(L^{\times},E) &\xrightarrow{\sim}  \homo(\bZ_{\beta}(L),E)/\homo(\bZ_{\beta,\{ir\}}(L),E),\\
	\psi& \mapsto\; [(a_1I_{r},\cdots, a_{k}I_{r})\mapsto \psi(a_i/a_{i+1})].\;
\end{aligned}
\end{equation}
For $\psi\in\homo(L^\times,E)$, choose a lift $\Psi\in\homo(\bZ_\beta(L),E)$.\;This determines an extension $\iota(\Psi)$ of $\pi_\beta$ by itself:
\begin{equation*}
\iota(\Psi)(a)=\pi_{\beta}(a)\otimes_E\begin{pmatrix}
	1 & \Psi\circ\det(a) \\ 0 & 1
\end{pmatrix}, \\ \forall ~a\in \bL_{\beta}(L).
\end{equation*}
The locally $\bQ_p$-analytic parabolic induction $\big(\ind_{\op_{\beta}(L)}^G \iota(\Psi)\otimes_E\lamoverx_{\Delta^{\beta}}\big)^{\bQ_p-\ana}$ fits into the following exact sequence of locally $\bQ_p$-analytic representations:
\begin{equation}\label{explicitexactseq}
0 \longrightarrow \bI_{\emptyset}(\bm\beta,\underline{\lambda}) \longrightarrow \big(\ind_{\op_{\beta}(L)}^G \iota(\Psi)\otimes_EL(\underline{\lambda})_{\Delta^{\beta}}\big)^{\bQ_p-\ana} \xrightarrow{\pr} \bI_{\emptyset}(\bm\beta,\underline{\lambda}) \longrightarrow 0.
\end{equation}
Pushing out \eqref{explicitexactseq} along the surjection $\bI_{\emptyset}(\bm\beta,\underline{\lambda})\twoheadrightarrow\st_n^{\ana}(\bm\beta,\underline{\lambda})$ and then pulling back along the natural injections $i_{ir}(\bm\beta,\underline{\lambda}) \hooklongrightarrow \bI_{ir}(\bm\beta,\underline{\lambda}) \hooklongrightarrow \bI_{\emptyset}(\bm\beta,\underline{\lambda})$ yields a locally $\bQ_p$-analytic representation:
\[\sE_{ir}^\emptyset(\bm\beta,\ul{\lambda},\Psi)^0=[\st_n^{\ana}(\bm\beta,\underline{\lambda})-i_{ir}(\bm\beta,\underline{\lambda})].\;\]
Its class $[\sE_{ir}^\emptyset(\bm\beta,\ul{\lambda},\psi)^0]\in\ext^1_G\big(i_{ir}(\bm\beta,\underline{\lambda}),\st_n^{\ana}(\bm\beta,\underline{\lambda})\big)$ is independent of the choice of $\Psi$.

By \cite[Theorem 5.9]{2022ext1hyq},\;the natural map $\ext^1_G\big(v_{ir}^{\lalg}(\bm\beta,\ul{\lambda}),\st_n^{\ana}(\bm\beta,\underline{\lambda})\big)\hookrightarrow\ext^1_G\big(i_{ir}(\bm\beta,\underline{\lambda}),\st_n^{\ana}(\bm\beta,\underline{\lambda})\big)$ induced by $i_{ir}(\bm\beta,\underline{\lambda})\twoheadrightarrow v_{ir}^{\lalg}(\bm\beta,\underline{\lambda})$ is an isomorphism.\;Hence, the pullback of $\sE_{ir}^\emptyset(\bm\beta,\ul{\lambda},\psi)^0$ along $w_{ir}(\bm\beta,\underline{\lambda})\hookrightarrow i_{ir}(\bm\beta,\underline{\lambda})$ splits. Quotienting by $w_{ir}(\bm\beta,\underline{\lambda})$ yields a locally $\bQ_p$-analytic representation:
\begin{equation}\label{dfnforsmoothanaStein}
	\widetilde{\Sigma}_{ir}(\bm\beta,\ul{\lambda}, \psi)=[\st_n^{\ana}(\bm\beta,\underline{\lambda})-v_{ir}^{\lalg}(\bm\beta,\underline{\lambda})],
\end{equation}
which represents the extension class associated with $\psi$ under \eqref{thmintro}.


Moreover,\;the natural embeddings induce isomorphisms (see \cite[Proposition 5.14,\;Proposition 5.17]{2022ext1hyq}):
\begin{equation}\label{Intro:equ: BrLi}
\ext^1_{G}\big(v_{ir}^{\lalg}(\bm\beta,\underline{\lambda}),	\Sigma_{ir}(\bm\beta,\underline{\lambda})\big)\xrightarrow{\sim }\ext^1_{G}\big(v_{ir}^{\lalg}(\bm\beta,\underline{\lambda}), \widetilde{\Sigma}_{ir}(\bm\beta,\underline{\lambda})\big)\xrightarrow{\sim }\ext^1_{G}\big(v_{ir}^{\lalg}(\bm\beta,\underline{\lambda}), \st_n^{\ana}(\bm\beta,\underline{\lambda})\big).
\end{equation}
For any $\psi \in \homo(L^{\times}, E)$,\;let $[\Sigma_{ir}(\bm\beta,\ul{\lambda}, \psi)]$ (resp.,\;$[\widetilde{\Sigma}_{ir}(\bm\beta,\ul{\lambda}, \psi)]$) be the image of $\psi$ under the first (resp.,\;third) isomorphism in (\ref{Intro:equ: BrLi}).\;Moreover,\;$[\widetilde{\Sigma}_{ir}(\bm\beta,\ul{\lambda}, \psi)]$ is obtained from $[\Sigma_{ir}(\bm\beta,\ul{\lambda}, \psi)]$ by pushing out along $\Sigma_{ir}(\bm\beta,\underline{\lambda})\hookrightarrow \st_n^{\ana}(\bm\beta,\underline{\lambda})$.\;

For $1\leq i\leq k-1$,\;let $V_i$ be an $E$-vector subspace of $\homo(L^\times,E)$ of dimension $d_L$.\;Let $\{\psi_{i,1}, \cdots, \psi_{i,d_L}\}$ be a basis of $V_i$.\;We put
\begin{eqnarray*}
	\Sigma_{ir}(\bm\beta,\ul{\lambda}, V_i):=\bigoplus_{\st_n^{\ana}(\bm\beta, \ul{\lambda})}^{j=1,\cdots, d_L} \Sigma_{ir}(\bm\beta,\ul{\lambda}, \psi_{i,j})\hookrightarrow		\widetilde{\Sigma}_{ir}(\bm\beta,\ul{\lambda}, V_i):=\bigoplus_{\Sigma_{ir}(\bm\beta,\ul{\lambda})}^{j=1,\cdots, d_L} 	\widetilde{\Sigma}_{ir}(\bm\beta,\ul{\lambda}, \psi_{i,j}),
\end{eqnarray*}
Both representations are independent of the choice of basis for $V_i$ and determine $V_i$.\;Set $V:=\prod_{1\leq i\leq k-1}V_i$ and
\begin{eqnarray*}
	{\Sigma}(\bm\beta,\ul{\lambda}, V):=\bigoplus_{\st_n^{\ana}(\bm\beta,\ul{\lambda})}^{1\leq i\leq k-1} \Sigma_{ir}(\bm\beta,\ul{\lambda}, V_i)\hookrightarrow
	\widetilde{\Sigma}(\bm\beta,\ul{\lambda}, V):=\bigoplus_{\st_n^{\lalg}(\bm\beta,\ul{\lambda})}^{1\leq i\leq k-1} 	\widetilde{\Sigma}_{ir}(\bm\beta,\ul{\lambda}, V_i).
\end{eqnarray*}
For $\alpha\in E^\times$,\;when the representations above are tensored with $\unr(\alpha)\circ\det$, we prepend $\alpha$ to the notation $(\bm\beta,\underline{\lambda},\cdots)$.\;For example,\;we write $v_{ir}^{\lalg}(\alpha,\bm\beta,\ul{\lambda})$, $w_{ir}(\alpha,\bm\beta,\ul{\lambda})$, $\st_n^{\ana}(\alpha,\bm\beta,\ul{\lambda})$, $\Sigma_{ir}(\alpha,\bm\beta,\ul{\lambda})$, and $\widetilde{\Sigma}_{ir}(\alpha,\bm\beta,\ul{\lambda}, \psi)$, and so on.

\subsection{Subcandidate in the locally analytic {$\bQ_p$}-adic local Langlands program}

We retain the notation of Section \ref{BreuilLINVSUMMMANY}.\;Fix a $(\varphi,\Gamma)$-module $\Dpik$ equipped with a non-critical special $\omepik$-filtration $\cF$ of parameter $(\underline{x}_\beta,z_\beta^\bh)\in\sbanpik\times\rigchl$ (Definition \ref{dfnnoncriticalspecial}).\;Put ${\bm\lambda}=(\hpi_{i,\tau}+i-1)_{\tau\in\Sigma_L,1\leq i\leq n}\in X^+_{\Delta}$.\;Let $\sL(\Dpik)$  be the parabolic Fontaine--Mazur simple $\sL$-invariants of $\Dpik$.\;

\begin{pro}We have the following statements:
	\begin{itemize}
		\item[(1)] $\widetilde{\Sigma}(\bm\beta,{\bm\lambda}, \sL(\Dpik))$ is isomorphic to an extension of $\bigoplus\limits_{1\leq i\leq k-1}v_{ir}^{\lalg}(\bm\beta,{\bm\lambda})^{d_L}$ by $\st_n^{\ana}(\bm\beta,{\bm\lambda})$.\;
		\item[(2)] $\Sigma_{ir}(\bm\beta,{\bm\lambda},\sL(\Dpik))$ has the following form:
		\begin{equation}
			\bigoplus_{\st_n^{\lalg}(\bm\beta,{\bm\lambda})}^{\sigma\in \Sigma_{L}}\big[\st_n^{\lalg}(\bm\beta,{\bm\lambda})-C_{ir,\sigma}(\bm\beta)-v_{ir}^{\lalg}(\bm\beta,{\bm\lambda})\big].
		\end{equation}
		\item[(3)] $\Sigma(\bm\beta,{\bm\lambda},\sL(\Dpik))$ has the following form:
		\begin{equation}
			\bigoplus_{\st_n^{\lalg}(\bm\beta,{\bm\lambda})}^{1\leq i\leq k-1,\sigma\in \Sigma_{L}}\big[\st_n^{\lalg}(\bm\beta,{\bm\lambda})-C_{ir,\sigma}(\bm\beta)-v_{ir}^{\lalg}(\bm\beta,{\bm\lambda})\big].
		\end{equation}
		\item[(4)] $\soc_G\Sigma_{ir}(\bm\beta,{\bm\lambda},\sL(\Dpik)_{ir})\cong \soc_G \Sigma(\bm\beta,{\bm\lambda},\sL(\Dpik))\cong \st_n^{\lalg}(\bm\beta,{\bm\lambda})$.\;The locally $\bQ_p$-algebraic subrepresentation of $\widetilde{\Sigma}(\bm\beta,{\bm\lambda},\sL(\Dpik))$ (resp. of $\widetilde{\Sigma}_{ir}(\bm\beta,{\bm\lambda},\sL(\Dpik)_{ir})$ for $1\leq i\leq k-1$) is isomorphic to $\st_n^{\lalg}(\bm\beta,{\bm\lambda})$.
		\item[(5)] $\Sigma(\alpha_x,\bm\beta,{\bm\lambda},\sL(\Dpik))$ and  $\widetilde{\Sigma}(\alpha_x,\bm\beta,{\bm\lambda},\sL(\Dpik))$ determine exactly the data $\{\wdre_{\Dpik},\bh,\sL(\Dpik)\}$,\;and vice versa.
	\end{itemize}
\end{pro}
\begin{proof}Statement (1) is clear.\;For (2),\;note that Remark \ref{taulinvariantrmk} asserts that $\dim_E\sL(\Dpik)_{ir,\tau}=1$ for all $\tau\in \Sigma_L$ and that $\sL(\Dpik)_{ir}=\bigoplus_{\tau\in \Sigma_L}\sL(\Dpik)_{ir,\tau}$.\;Thus,\;there exists $\sL_{i,\tau}\in E$ such that $\sL(\Dpik)_{ir,\tau}$ is generated by $\psi_{\tau}-\sL_{ir,\tau}\psi_{\ur}$.\;The vectors $\{\psi_{\tau}-\sL_{ir,\tau}\psi_{\ur}\}_{\tau\in \Sigma_L}$ then form a basis of $\sL(\Dpik)_{ir}$.\;The same argument as in the discussion following \cite[Lem.\;3.3]{2019DINGSimple} gives the decomposition
	\[\Sigma_{ir}(\bm\beta,{\bm\lambda}, \psi_{i,\tau})={\Sigma}_{ir,\tau}(\bm\beta,{\bm\lambda}, \psi_{\tau}-\sL_{ir,\tau}\psi_{\ur})\oplus_{\st_n^{\lalg}(\bm\beta,{\bm\lambda})}\left(
	\bigoplus_{\st_n^{\lalg}(\bm\beta,{\bm\lambda})}^{\sigma\in\Sigma_L,\sigma\neq\tau} \Sigma_{ir,\sigma}(\bm\beta,{\bm\lambda})\right).\]
	This proves (2).\;The same argument proves $(3)$.\;By \cite[Remark 5.11]{2022ext1hyq}, the locally algebraic subrepresentation of $\widetilde{\Sigma}_{ir}(\bm\beta,{\bm\lambda},\sL(\Dpik)_{ir})$ is strictly larger than $\st_n^\infty(\bm\beta,{\bm\lambda})$ if and only if $\homo_{\sm}(L^\times,E)\subseteq\sL(\Dpik)_{ir}$.\;Statement (4) follows, and (5) follows from Proposition \ref{criofnoncrystalline} and Remark \ref{taulinvariantrmk}.
\end{proof}

\section{Local-global compatibility}

In this section, we prove new $p$-adic local--global compatibility results for potentially semistable non-crystalline $p$-adic Galois representations. To treat non-trianguline representations, we work with the patched Bernstein eigenvariety constructed by Breuil and Ding \cite{Ding2021}; classical eigenvarieties generally detect only finite-slope $p$-adic automorphic forms and trianguline representations.

We briefly outline the section.\;In Section \ref{BENVARPARAVAR}, we review the global patching construction under the Taylor--Wiles hypotheses of \cite{PATCHING2016}, and then construct the patched Bernstein eigenvariety, by adapting \cite[Section 3.3]{Ding2021}, together with the purely local Bernstein paraboline variety of \cite[Section 4.2]{Ding2021}. Let $x$ be a point of the patched Bernstein eigenvariety whose associated Galois representation $\rho_x$ admits a non-critical special $\omepik$-filtration. In Section \ref{smoothpoint}, we prove that $x$ is a smooth point and show that the tangent map of the weight map recovers the parabolic Fontaine--Mazur simple $\sL$-invariants $\sL(\rho_L)$; see Proposition \ref{coreigentangentmap2}.\;The constructions of Section \ref{BreuilLINVSUMMMANY} then attach to $\rho_L$ a locally $\bQ_p$-analytic representation $\Sigma(\alpha,\bm\beta,{\bm\lambda},\sL(\rho_L))$. Our main local--global compatibility theorem states that this representation embeds into the relevant Hecke-isotypic subspace of a Banach space of $p$-adic automorphic forms on a definite unitary group; see Theorem \ref{thm: lgln-main}.\;

In this section,\;we also fix a potentially semistable and non-crystalline $p$-adic Galois representation $\rho_L:\gal_L\rightarrow\GLN_n(E)$, and let $\Dpik$ be its associated rank-$n$ $(\varphi,\Gamma)$-module over $\cR_{E,L}$.\;Let $\bh:=(\hpi_{1,\tau}>\hpi_{2,\tau}>\cdots>\hpi_{n,\tau})_{\tau\in\Sigma_L}$ be its Hodge--Tate weights. Recall that $\hpi_i=(\hpi_{i,\tau})_{\tau\in\Sigma_L}$ and ${\bm\lambda}=(\hpi_{i,\tau}+i-1)_{\tau\in\Sigma_L,1\leq i\leq n}\in X^+_{\Delta}$. Suppose that $\Dpik$ admits a non-critical special $\omepik$-filtration $\cF$ with parameter
\[(\underline{x}_{\beta,\bh},z_{\beta}^\bh)\in \sbanpik\times\rigch ,\quad(\underline{x}_{\beta},z_{\beta}^\bh)\in \sbanpik\times\rigchl, \]
as in Definition \ref{dfnnoncriticalspecial}.\;After inverting $t$, the filtration induces an increasing filtration $\cM_\bullet=\{\fil_i^\cF\Dpik[1/t]\}$ of $\cM_\Dpik$ by $(\varphi,\Gamma)$-submodules over $\cR_{E,L}[1/t]$.\;


\subsection{Patched Bernstein eigenvarieties and Bernstein parabolic varieties}\label{BENVARPARAVAR}

We formulate local--global compatibility in the setting of patched $p$-adic automorphic forms used in \cite[Section 4.1.1]{2019DINGSimple}. We first recall the patched Bernstein eigenvariety and the Bernstein paraboline variety of Breuil--Ding \cite[Sections 3.3 and 4.2]{Ding2021}. Our patching setup differs slightly because we vary weights and levels at only one $p$-adic place, but the arguments in loc. cit. adapt directly. We use the patched setting and notation of \cite[Section 4.1.1]{2019DINGSimple} and \cite[Section 2]{PATCHING2016}.

Let $F^+$ be a totally real field, and let $F$ be a totally imaginary quadratic extension of $F^+$ such that all places of $F^+$ dividing $p$ split in $F$.\;Fix a definite unitary algebraic group $\bG$ over $F^+$ that is compact at the infinite places and isomorphic to $\GLN_n$ at the $p$-adic places.\;We assume the following hypotheses.
\begin{hypothesis}[Taylor--Wiles hypothesis]\label{TWhypo}
	\begin{itemize}
		\item[(1)] $p>2$;
		\item[(2)] the field $F$ is unramified over $F^+$, contains no nontrivial $p$-th root of unity, and $\BG$ is quasi-split at every finite place of $F^+$;
		\item[(3)] $U_{v}$ is hyperspecial when the finite place $v$ of $F^+$ is inert in $F$;
		\item[(4)] $\overline{\rho}$ is absolutely irreducible and $\overline{\rho}(\gal_{F(\sqrt[p]{1})})$ is adequate.\;
	\end{itemize}
\end{hypothesis}
Let $\fp$ be a $p$-adic place of  $F^+$ and put $L:=F^+_\fp$. Let $\overline r=\overline{\rho}|_{\gal_{F^+_\fp}}:\gal_L\rightarrow\GLN_n(k_E)$, and let $R_{\overline r}^\square$ be the maximal reduced, $p$-torsion-free quotient of its universal framed $\co_E$-lifting ring. The patching construction provides a ring (see  \cite[Section 4.1.1]{2019DINGSimple} and \cite[Section 2]{PATCHING2016} for the local deformation ring $R_\infty^\fp$ at places $\neq \fp$)
\[R_\infty=R_\infty^\fp\llbracket x_1,\ldots,x_g\rrbracket\widehat\otimes R_{\overline r}^\square\]
and an $R_\infty$-admissible unitary representation $\Pi_\infty$ of $G$ over $E$, called the patched Banach representation. Let $\Pi_\infty^{R_\infty-\ana}$ denote its locally $R_\infty$-analytic vectors \cite[Section 3.1]{breuil2017interpretation}. Put $\FX_\infty:=(\Spf\hspace{2pt}R_\infty)^\rig$, $\FX^\fp:=(\Spf\hspace{2pt}  R_\infty^\fp)^\rig$, and $\mathfrak X_{\overline r}^\Box:=(\Spf\hspace{2pt} R_{\overline r}^\square)^\rig$. If $\BU$ denotes the open unit ball in $\BA^1$, then
\[\FX_\infty\cong\FX^\fp\times\BU^g\times\mathfrak X_{\overline r}^\Box.\]

Recall that $\bZ_\beta(L)\cong Z_{\beta,\varpi_L}\times\bZ_\beta(\cO_L)$,\;where
$Z_{\beta,\varpi_L}$ is the image of $\oplus_{i=1}^k\BZ\hookrightarrow \bZ_{\beta}(L),\;(m_i)\mapsto (\varpi_L^{m_i})$.\;We recall the construction, from \cite[Sections 3.1.1,\;3.1.2,\;3.3]{Ding2021}, of a $R_{\infty}\times\FZ_{\omepik}\times Z_{\beta,\varpi_L}\times \bZ_{\beta}(\cO_L)$-module $B_{\omepik,{\bm\lambda}}(\Pi_\infty^{R_\infty-\ana})$.\;By the discussion in \cite[Section 3.1.1]{Ding2021},\;let $(J,\sigma_r^0)$ be a maximal simple type of $\Omega_r$ with $J\subseteq \GLN_{r}(\cO_L)$ (Bushnell--Kutzko theory of types).\;Put $\sigma_r:=\big(\ind^{\GLN_{r}(\cO_L)}_{J}\sigma_r^0\big)^{\infty}$,\;which is an absolutely irreducible smooth representation  of $\GLN_r(\cO_L)$.\;Let $\sigma_{\beta}:=\boxtimes_{i=1}^k\sigma_r$.\;We have \[\FZ_{\omepik}\cong \EndO_{\bL_{\beta}(L)}\big(\textrm{c-ind}_{\bL_{\beta}(\cO_L)}^{\bL_{\beta}(L)}\sigma\big),\;\]where ``$\textrm{c-ind}$'' denotes compact induction.\;We define
\begin{equation*}
\begin{aligned}
	B_{\sigma,{\bm\lambda}}(\Pi_\infty^{R_\infty-\ana}):=&\;\homo_{\bL_{\beta}(\cO_L)}\big(\sigma,J_{\bP_{\beta}}(\Pi_\infty^{R_\infty-\ana})_{{\bm\lambda}}\widehat{\otimes}_E\cC^{\bQ_p-\ana}(\bZ_{\beta}(\cO_L),E)\big)\\
	\cong &\; \homo_{\bL_{\beta}(L)}\big(\textrm{c-ind}_{\bL_\beta(\cO_L)}^{\bL_{\beta}(L)}\sigma,J_{\bP_{\beta}}(\Pi_\infty^{R_\infty-\ana})_{{\bm\lambda}}\widehat{\otimes}_E\cC^{\bQ_p-\ana}(\bZ_{\beta}(\cO_L),E)\big),\;
\end{aligned}
\end{equation*}
where $J_{\bP_{\beta}}(\Pi_\infty^{R_\infty-\ana})_{{\bm\lambda}}:=\homo_{\fd_{\beta}}(L({\bm\lambda}),J_{\bP_{\beta}}(\Pi_\infty^{R_\infty-\ana}))$ (recall that $\fd_{\beta}$ is the Lie algebra of the derived subgroup $\bD_{\beta}$ of $\bL_{\beta}$ and that $J_{\bP_{\beta}}$ is the Emerton--Jacquet functor \cite{emerton2006jacquet}).\;Following \cite[Section 3.1.2]{Ding2021},\;we recall the group actions on $B_{\sigma,{\bm\lambda}}(\Pi_\infty^{R_\infty-\ana})$.\;We write $\cZ_1$ (resp.,\;$\cZ_0$)
	for the action of $\bZ_{\beta}(L)$ on $B_{\sigma,{\bm\lambda}}(\Pi_\infty^{R_\infty-\ana})$ induced by the action of $\bZ_{\beta}(L)$ (resp.,\;$\bZ_{\beta}(\cO_L)$) on $J_{\bP_{\beta}}(\Pi_\infty^{R_\infty-\ana})_{{\bm\lambda}}$ (resp.,\;$\cC^{\bQ_p-\ana}(\bZ_{\beta}(\cO_L),E)$; see \cite[(3.4)]{Ding2021}.\;

Next, $\FZ_{\omepik}$ acts on the module $B_{\sigma,{\bm\lambda}}(\Pi_\infty^{R_\infty-\ana})$ via the factor $\textrm{c-ind}_{\bL_{\beta}(\cO_L)}^{\bL_{\beta}(L)}\sigma$; this action commutes with that of $\cZ_1\times \cZ_0$.\;We write $\Delta_0$ for the action of $\bZ_{\beta}(L)$ on $B_{\sigma,{\bm\lambda}}(\Pi_\infty^{R_\infty-\ana})$ induced by its diagonal action on $J_{\bP_{\beta}}(\Pi_\infty^{R_\infty-\ana})_{{\bm\lambda}}\widehat{\otimes}_E\cC^{\bQ_p-\ana}(\bZ_{\beta}(\cO_L),E)$ (see the argument before \cite[(3.4)]{Ding2021}).\;The argument preceding \cite[Lemma 3.1.2]{Ding2021} shows that the action of $\Delta_0$ is determined by $\Delta_0|_{Z_{\beta,\varpi_L}}$;\;moreover,\;$\Delta_0|_{\bZ_{\beta}(\cO_L)}$ acts via the central character of $\sigma$.\;The $R_{\infty}\times\FZ_{\omepik}\times Z_{\beta,\varpi_L}\times \bZ_{\beta}(\cO_L)$-module structure of $B_{\omepik,{\bm\lambda}}(\Pi_\infty^{R_\infty-\ana})^\vee$ is given by the $R_{\infty}\times\FZ_{\omepik}\times\iota_1(Z_{\beta,\varpi_L})\times \cZ_0$-action,\;with $R_{\infty}$ acting on $B_{\sigma,{\bm\lambda}}(\Pi_\infty^{R_\infty-\ana})$ through the factor $\Pi_\infty^{R_\infty-\ana}$.\;By \cite[Lemma 3.17]{Ding2021},\;the $\FZ_{\omepik}\times Z_{\beta,\varpi_L}\times \bZ_{\beta}(\cO_L)$-module $B_{\sigma,{\bm\lambda}}(\Pi_\infty^{R_\infty-\ana})$ is independent of the choice of $\sigma$.\;We therefore write $B_{\omepik,{\bm\lambda}}(\Pi_\infty^{R_\infty-\ana}):=B_{\sigma,{\bm\lambda}}(\Pi_\infty^{R_\infty-\ana})$.\;

We also recall how to recover $\cZ_1$ from the $\cZ_0$- and $\FZ_{\omepik}$-actions. Let $\psi_\sigma$ be the central character of $\sigma$, viewed as a character of $\bZ_\beta(\cO_L)$. By \cite[(3.5)]{Ding2021},
\[\cZ_1(z_0)=\psi_{\sigma}(z_0)\cZ_0(\mathrm{det}_{\bL_{\beta}(L)}(z_0))\]
for any $z_0\in \bZ_{\beta}(\cO_L)$.\;On the other hand, $\cZ_1|_{Z_{\beta,\varpi_L}}=\Delta_0|_{Z_{\beta,\varpi_L}}$.\;Moreover, the $\bZ_{\beta}(L)$-action on $B_{\omepik,{\bm\lambda}}(\Pi_\infty^{R_\infty-\ana})$ induced by the map \[\bZ_{\beta}(L)\rightarrow  \EndO_{\bL_{\beta}(L)}\big(\textrm{c-ind}_{\bL_{\beta}(\cO_L)}^{\bL_{\beta}(L)}\sigma\big)\cong\FZ_{\omepik}\] coincides with the $\Delta_0$-action.\;Thus,\;the $\cZ_1$-action is determined by the $\FZ_{\omepik}$-action.\;

Adapting the proof of \cite[Lemma 3.1.3]{Ding2021} and the argument preceding \cite[(3.28)]{Ding2021},\;we see that $B_{\omepik,{\bm\lambda}}(\Pi_\infty^{R_\infty-\ana})^\vee$ is a coadmissible module over $\cO(\FX_{\infty}\times\zuni\times\rigch)$,\;corresponding to a coherent sheaf $\cM_{\omepik,{\bm\lambda}}^{\infty,0}$ over $\FX_{\infty}\times \zuni\times \rigch$.\;Since the action of $Z_{\beta,\varpi_L}$ factors through $\sbanpik$,\;the sheaf $\cM_{\omepik,{\bm\lambda}}^{\infty,0}$ gives rise to a coherent sheaf $\cM_{\omepik,{\bm\lambda}}^{\infty}$ over $\FX_{\infty}\times\sbanpik\times \rigch$ such that
\[\Gamma\Big(\FX_{\infty}\times\sbanpik\times \rigch,\cM_{\omepik,{\bm\lambda}}^{\infty}\Big)\cong B_{\omepik,{\bm\lambda}}(\Pi_\infty^{R_\infty-\ana})^\vee.\]
Let $\mathcal{E}_{\omepik,\fp,{\bm\lambda}}^\infty(\overline{\rho})\hookrightarrow\FX_{\infty}\times\sbanpik\times \rigch$ be the Zariski-closed support of $\cM_{\omepik,{\bm\lambda}}^{\infty}$.\;We call this space the patched Bernstein eigenvariety.\;Adapting the proofs of Propositions 3.3.2, 3.3.4, and 3.3.6, Corollary 3.3.3, and Theorem 3.3.5 of \cite{Ding2021} gives the following.
\begin{pro}\label{basicpropertyeigen}\hspace{20pt}
\begin{itemize}
		\item[(1)] The point $y=(\fm_y, \pi_{y},\chi_y)\in \FX_{\infty}\times\sbanpik\times \rigch$ lies in $\mathcal{E}_{\omepik,\fp,{\bm\lambda}}^\infty(\overline{\rho})$ if and only if
	\[\homo_{\bL_{\beta}(L)}\Big(\pi_{y}\otimes_E((\chi_y)_{\varpi_{L}}\circ\mathrm{ det}_{\bL_{\beta}(L)})\otimes_E L({\bm\lambda}),J_{\bP_{\beta}}(\Pi_\infty^{R_\infty-\ana}[\fm_y] )\Big)\neq0\]
	\item[(2)] The rigid space $\mathcal{E}_{\omepik,\fp,{\bm\lambda}}^\infty(\overline{\rho})$ is reduced and equidimensional of dimension
	\[g+kd_L+n^2(|S_p|+1)+[F^+:\bQ]\frac{n(n-1)}{2}.\]
	\item[(3)] The coherent sheaf $\cM_{\omepik,{\bm\lambda}}^{\infty}$ is Cohen-Macaulay over $\mathcal{E}_{\omepik,\fp,{\bm\lambda}}^\infty(\overline{\rho})$.\;
	\item[(4)] The set of very classical non-critical generic points (see \cite[Definition 3.2.7 and (3.22)]{Ding2021}) is Zariski-dense in $\mathcal{E}_{\omepik,\fp,{\bm\lambda}}^\infty(\overline{\rho})$ and is an accumulation set. More precisely, these points accumulate at every point $y'=(\fm_{y'},\pi_{y'},\chi_{y'})$ for which $\chi_{y'}$ is locally algebraic.
\end{itemize}
\end{pro}

Let $\defvar$ be the Bernstein paraboline variety associated with $(\omepik,\bh)$ \cite[Section 4.2]{Ding2021}. It is closely related to the patched Bernstein eigenvariety. It contains a Zariski-open dense subspace $U^\Box_{\omepik,\mathbf h}(\overline r)$ of generic points and is the Zariski closure of this subspace in $\mathfrak X_{\overline r}^\Box\times\sbanpik\times\rigch$.\;By \cite[Theorem 4.2.5 and Corollary 4.2.6]{Ding2021}, we have the following results.
\begin{pro}\label{propertyparavar}\hspace{20pt}
\begin{itemize}
	\item[(1)] The rigid space $\defvar$ is equidimensional of dimension $n^2+\left(\frac{n(n-1)}{2}+k\right)d_L$.\;
	\item[(2)] If $x'=(\rho_{x'},(x'_i)_{1\leq i\leq k},\undelram)\in\defvar$, then $D_{\rig}(\rho_{x'})$ admits an $\omepik$-filtration $\cF=\{\fil_i^\cF D_{\rig}(\rho_{x'})\}$ such that, for $1\leq i\leq k$,
	\[\gr_i^\cF D_{\rig}(\rho_{x'})\otimes_{\cR_{k(x'),L}}\cR_{k(x'),L}((\delta_i^0)^{-1}_{\varpi_L})\big[1/t\big]=\Delta_{x'_i}\big[1/t\big].\]
\end{itemize}
\end{pro}
Consider the composition
\begin{equation}\label{composition}
\begin{aligned}
	\bersteineigenvarpik\hooklongrightarrow \FX^{\fp}&\times \mathfrak{X}_{\overline{r}}^\Box \times \sbanpik\times \rigch\xrightarrow{\iota_{\omepik}} \FX^{\fp}\times \mathfrak{X}_{\overline{r}}^\Box \times \sbanpik\times \rigch.
\end{aligned}
\end{equation}
where $\iota_{\omepik}:\sbanpik\xrightarrow{\sim}\sbanpik$ is characterized by \[\pi_{\iota_{\omepik}(x)_i}=\pi_{x_i}\otimes_E\unr(q_L^{(i-1)r+\frac{r-1}{2}})\circ\det\] for $x=(x_i)_{1\leq i\leq k}\in\sbanpik$. Adapting the proof of \cite[Theorem 3.3.9]{Ding2021} gives the following proposition.
\begin{pro}\label{linkingpatched}The composition in (\ref{composition}) factors through $\FX^{\fp}\times \defvarrho$,\;i.e.,\;
\begin{equation}\label{mapbersteineigenvarpikdefvarrho}
	\Lambda:\bersteineigenvarpik\hooklongrightarrow \FX^{\fp}\times \iota_{\omepik}^{-1}(\defvar).
\end{equation}	
It induces an isomorphism between $\bersteineigenvarpik$ and a union of irreducible components of $\FX^{\fp}\times \iota_{\omepik}^{-1}(\defvar)$, equipped with the reduced closed rigid subspace structure.\;
\end{pro}

\subsection{Non-critical special point}\label{appearpoint}

We freely use the notation of Section \ref{noncrispecfildfn}.\;Recall that $\pi_x\cong{\pi}\otimes\unr(\alpha_x)\circ\det$. As in \eqref{mapbersteineigenvarpikdefvarrho} and Section \ref{noncrispecfildfn}, put
\begin{equation}
\begin{aligned}
	&(\breve{\underline{x}}_{\beta},\underline{\mathbf{1}})=(\iota_{\omepik}^{-1}(\underline{x}_{\beta}),\mathbf{\underline{1}})\in \sbanpik\times\rigch,\\
	&\breve{\underline{x}}_{\beta}=(\breve{x}_{i})_{1\leq i\leq k},\quad\pi_{\breve{x}_{i}}\cong {\pi}\te\unr\big(\alpha_x q_L^{i-k+\frac{r-1}{2}+r(i-1)}\big)\circ\det,
\end{aligned}
\end{equation}
where $\underline{\mathbf{1}}$ is the trivial character.\;

Suppose there exists $y^\fp\in\FX^\fp$ such that $y=(y^\fp,\rho_L,\breve{\underline{x}}_\beta,\underline{\mathbf1})\in\bersteineigenvarpik$.\;Set $y_L:=(\rho_L,\underline{x}_\beta,\underline{\mathbf1})\in\defvar$. We call $y_L$ and $y$ the non-critical special points of $\defvarrho$ and $\bersteineigenvarpik$, respectively.

We first establish the accumulation property at $y_L$ and $y$.\;Let $\cE(y)$ be the union of the irreducible components of $\bersteineigenvarpik$ containing $y$.\;By Theorem \ref{linkingpatched}, $\cE(y)=\cup_j\big(X_j^\fp\times\BU^g\times\iota_{\omepik}^{-1}(X_{j,\fp})\big)$,\;where $X_{j,\fp}$ is an irreducible component of $\defvarrho$ containing $y_L$ and $X_j^\fp$ is an irreducible component of $\FX^\fp$ containing $y^\fp$.\;By \cite[Theorem 3.3.8]{kisin2008potentially} and \cite[Lemma 2.5]{PATCHING2016},\;$\FX^\fp$ is smooth at $y^\fp$, so there is a unique such component, denoted $X^\fp$.\;Therefore,
\begin{equation}\label{OZfactor}
\cE(y)=X^{\fp} \times \BU^g \times \iota_{\omepik}^{-1}(\xfpx)\hookrightarrow X^{\fp} \times \BU^g \times \iota_{\omepik}^{-1}(\defvarrho)
\end{equation}
with $\xfpx=\cup_j X_{j,\fp}$.\;

The following definition generalizes the accumulation property in \cite[Definition 2.11]{breuil2017smoothness}.\;
\begin{dfn}\label{accum}Let $X$ be a union of irreducible components of an open subspace of $\defvar$. We say that $X$ satisfies the accumulation property at $y$ if, for every $C>0$, the set of potentially crystalline points $y'=(\rho',\underline{y}',\delta^0_{y'})\in\defvar$ satisfying the following conditions accumulates at $y$:
	\begin{itemize}
		\item[(a)] $y'$ is generic and non-critical;
		\item[(b)] $\wt_\tau(\delta_{y',i}^0)-\wt_\tau(\delta_{y',i+1}^0)>C$ for $1\leq i\leq k-1$ and $\tau\in\Sigma_L$.
	\end{itemize}
\end{dfn}

\begin{lem}\label{accumulation} $\xfpx$ satisfies the accumulation property at $y_L$.
\end{lem}
\begin{proof}We follow the approach of \cite[Proposition 3.10]{breuil2017smoothness}, drawing on the proof of \cite[Theorem 3.9]{breuil2017interpretation} and the statement of \cite[Theorem 3.3.5]{Ding2021}.\;Recall that \cite[Theorem 3.3.5]{Ding2021} states that the set of very classical, non-critical, generic points (and hence potentially crystalline points) is Zariski-dense in $\bersteineigenvarpik$ and accumulates at every point $y'=(\fm_{y'}, \pi_{y'},\chi_{y'})\in \bersteineigenvarpik$ for which $\chi_{y'}$ is locally algebraic.\;We adapt the proof of this statement to our setting.\;More precisely, let $y=(y^{\fp}, \rho_L, \breve{\underline{x}}_{\beta},\underline{\mathbf{1}})\in \bersteineigenvarpik$ be the non-critical point, and let $X$ be an irreducible component of $\bersteineigenvarpik$ containing $y$.
By \cite[Proposition 3.3.2 (2)]{Ding2021}, one can choose an affinoid, connected neighborhood $U$ of $y$ such that $U\rightarrow \omega_X(U)$ is finite and surjective and $\omega_X(U)$ is an affinoid open subset.\;For any sufficiently large $C$, let $W_C$ be the subset of
$\rigch$ consisting of dominant algebraic characters $\delta^0$ such that $\wt_\tau(\delta_i^0)-\wt_\tau(\delta_{i+1}^0)>C$ for $1\leq i\leq k-1$ and $\tau\in\Sigma_L$. Since $U$ is affinoid, \cite[Propositions 3.2.9 and 3.2.14]{Ding2021} implies that, for sufficiently large $C$, every point $y'=(y'^\fp,\rho',\underline y',\delta^0_{y'})$ with $\delta^0_{y'}\in W_C$ is very classical, non-critical, and generic. Arguing as in the second paragraph of the proof of \cite[Theorem 3.9]{breuil2017interpretation}, we obtain an affinoid neighborhood $U$ of $y$ in $X$ and a subset $W\subseteq\rigch$ such that $\omega_X^{-1}((\Spf S_\infty)^\rig\times W)\cap U$ consists of such points.\;Here, $\omega_X$ is the restriction of the morphism defined before \cite[Proposition 3.3.2]{Ding2021}. The factorization \eqref{OZfactor} then proves the lemma.
\end{proof}

Let $r^{\univer}$ be the universal framed Galois deformation of $\overline{r}$ over $\mathfrak{X}_{\overline{r}}^\Box$, let $\delta_i^{0,\univer}$ be the universal character of $\cO_L^\times$ over $\widehat{\cO_L^\times}$, and let $\Delta^{\univer}_{\Omega_r}$ be the universal $p$-adic differential equation over $\sbanpi$ (see \cite[Section 2.2]{Ding2021}).\;For $X\subset \xfpx\hookrightarrow\mathfrak X_{\overline r}^\Box\times\sbanpik\times\rigch$, let $r^{\univer}_{X}$, $\delta_{X,i}^{0,\univer}$, and $\Delta^{\univer}_{\Omega_r,X}$ be the pullbacks to $X$ of $r^{\univer}$, $\delta_{i}^{0,\univer}$, and $\Delta^{\univer}_{\pi}$.\;Let $D_{\rig}\big(r^{\univer}_X\big)$ be the $(\varphi,\Gamma)$-module over $\cR_{X,L}$ associated with $r^{\univer}_{X}$.\;For $1\leq i\leq k$, set
\[\Delta_{X,i}=\Delta^{\univer}_{\Omega_r,X}\otimes_{\cR_{X,L}}\cR_{X,L}\big(\big(\delta_{X,i}^{0,\univer}\big)_{\varpi_L}z^{{\bh}_{(i-1)r+1}}\big).\]
By the argument following (\ref{Dpikinjection}),\;there is a natural injection $\Delta_{X,y_L,i}\hookrightarrow \gr^{\cF}_i\Dpik\cong \gr_i^{\cF}D_{\rig}(r^{\univer}_{X,y_L})$ for $1\leq i\leq k$, where $\Delta_{X,y_L,i}$ is the specialization of $\Delta_{X,i}$ at $y_L$.\;We show that the $\omepik$-filtration at $y_L$ extends to an open affinoid neighborhood of $y_L$.\;

\begin{thm}\label{localfiltration} There exists an open affinoid neighborhood $X\subset\xfpx$ of $y_L$ such that the $(\varphi,\Gamma)$-module $D_{\rig}(r^{\univer}_X)[1/t]$ over $\cR_{X,L}[1/t]$ admits a filtration $\cM_\bullet$ satisfying $\gr_i^{\cM_\bullet}\cong\Delta_{X,i}[1/t]$.\;In particular, its specialization at $y_L$ is an $\omepik$-filtration of $D_{\rig}(\rho_L)[1/t]$ with parameter $(\underline{x}_\beta,z_\beta^\bh)$.
\end{thm}
\begin{proof}
By the accumulation property at $y_L$,\;there is a Zariski-dense set $S$ of potentially crystalline, non-critical, generic points. For each $z\in S$, we have (see the proof of \cite[Proposition 4.2.7]{Ding2021})
\[\dim_{k(z)}\ext^i_{(\varphi,\Gamma)}(\Delta_{X,z,1},D_{\rig}(r^{\univer}_{X,z}))=\left\{
\begin{array}{ll}
	1, & \hbox{$i=0$;} \\
	1+nr\cdot d_L, & \hbox{$i=1$;} \\
	0, & \hbox{$i=2$.}
\end{array}
\right.\]
Using arguments similar to those in the proofs of Lemmas \ref{phigammacohononcrtialspecial} and \ref{pullbackforiso} (recall that $\gr_1^{\cF}D_{\rig}(r^{\univer}_{X,y_L})/\Delta_{X,y_L,1}$ is a torsion $(\varphi,\Gamma)$-module),\;we deduce that, for $s>1$, and $j\in\{0,2\}$,\;
\begin{itemize}
	\item[(1)] $\ext^j_{(\varphi,\Gamma)}(\Delta_{X,y_L,1},\gr_s^{\cF}D_{\rig}(r^{\univer}_{X,y_L}))=0$,\;
	\item[(2)] $\ext^j_{(\varphi,\Gamma)}(\Delta_{X,y_L,1},\gr_1^{\cF}D_{\rig}(r^{\univer}_{X,y_L})/\Delta_{X,y_L,1})=0$.\;
\end{itemize}
An elementary computation together with a d\'{e}vissage argument gives
\[\dim_{k(y_L)}\ext^i_{(\varphi,\Gamma)}(\Delta_{X,y_L,1},D_{\rig}(r^{\univer}_{X,y_L}))=\left\{
\begin{array}{ll}
	1, & \hbox{$i=0$;} \\
	1+nr\cdot d_L, & \hbox{$i=1$;} \\
	0, & \hbox{$i=2$.}
\end{array}
\right.\]
By \cite[Corollary 4.7]{bergdall_2017},\;$\ext^0_{(\varphi,\Gamma)}(\Delta_{X,1},D_{\rig}(r^{\univer}_X))$ is locally free of rank one at $y_L$.\;After shrinking $X$,\;we may assume that $\ext^0_{(\varphi,\Gamma)}(\Delta_{X,1},D_{\rig}(r^{\univer}_X))$ is free of rank one.\;Choosing a generator gives a map
\[\iota_1:\Delta_{X,1}\longrightarrow D_{\rig}(r^{\univer}_X).\]
By \cite[Lemma 5.4 (a)]{bergdall_2017},\;$\iota_1$ is injective.\;

Now let $Q_{X,1}:=\cokerr\iota_1$ be a generalized $(\varphi,\Gamma)$-module over $\cR_{X,L}$. After shrinking $X$, Lemma 5.4(b) of \cite{bergdall_2017} allows us to assume that $Q_{X,1}$ is nearly flat in the sense of \cite[Definition 4.1]{bergdall_2017}. Unlike in the trianguline case, $\iota_1$ is never saturated, so $Q_{X,1}$ has $t$-torsion at every point. For $z\in S$ and at $y_L$, the corresponding specializations $Q_{X,z,1}$ and $Q_{X,y_L,1}$ are extensions of
\begin{equation}
	\begin{aligned}
		&\fil_2^kD_{\rig}(r^{\univer}_{X,z}):=D_{\rig}(r^{\univer}_{X,z})/\gr_1^{\cF}D_{\rig}(r^{\univer}_{X,z})\;\\
		(\text{resp.},\;&\fil_2^kD_{\rig}(r^{\univer}_{X,y_L}):=D_{\rig}(r^{\univer}_{X,y_L})/\gr_1^{\cF}D_{\rig}(r^{\univer}_{X,y_L}))
	\end{aligned}
\end{equation}
by the torsion $(\varphi,\Gamma)$-modules $\gr_1^{\cF}D_{\rig}(r^{\univer}_{X,z})/\Delta_{X,z,1}$ and $\gr_1^{\cF}D_{\rig}(r^{\univer}_{X,y_L})/\Delta_{X,y_L,1}$, respectively.\;Moreover,\;since $z\in S$ is generic, we have  \[\ext^1_{(\varphi,\Gamma)}(\fil_2^kD_{\rig}(r^{\univer}_{X,z}),\gr_1^{\cF}D_{\rig}(r^{\univer}_{X,z})/\Delta_{X,z,1})=0,\] and hence $Q_{X,z,1}\cong \fil_2^kD_{\rig}(r^{\univer}_{X,z})\oplus \gr_1^{\cF}D_{\rig}(r^{\univer}_{X,z})/\Delta_{X,z,1}$.\;

Because $Q_{X,y_L}[\frac{1}{t}]=\fil_2^kD_{\rig}(r^{\univer}_{X,y_L})[\frac{1}{t}]$, the same $E$-$B$-pair arguments as in the proofs of Lemmas \ref{phigammacohononcrtialspecial} and \ref{pullbackforiso} yield, for some sufficiently large integer $N$, an isomorphism
\[\ext^1_{(\varphi,\Gamma)}(t^N\Delta_{X,y_L,2},Q_{X,y_L,1})\xrightarrow{\sim}\ext^1_{(\varphi,\Gamma)}(t^N\Delta_{X,y_L,2},\fil_2^kD_{\rig}(r^{\univer}_{X,y_L})).\]
In this case,\;for $s>2$ and $j\in\{0,2\}$, note that
\begin{itemize}
	\item[(1)] $\ext^j_{(\varphi,\Gamma)}(t^N\Delta_{X,y_L,2},\gr_s^{\cF}D_{\rig}(r^{\univer}_{X,y_L}))=0$,\;
	\item[(2)]  $\ext^j_{(\varphi,\Gamma)}(t^N\Delta_{X,y_L,2},\gr_2^{\cF}D_{\rig}(r^{\univer}_{X,y_L})/\Delta_{X,y_L,2})=0$.\;
	\end{itemize}
It follows that
\begin{equation}
	\begin{aligned}
		&\dim_E\ext^i_{(\varphi,\Gamma)}(t^N\Delta_{X,y_L,2},\fil_2^kD_{\rig}(r^{\univer}_{X,y_L}))=\left\{
		\begin{array}{ll}
			1, & \hbox{$i=0$;} \\
			1+(k-1)r^2\cdot d_L, & \hbox{$i=1$;} \\
			0, & \hbox{$i=2$.}
		\end{array}
		\right.
	\end{aligned}
\end{equation}
Fix this integer $N$.\;For each $z\in S$, we claim that $\wt(\delta_{1,z}^0)-\wt(\delta_{2,z}^0)>N$ implies
\[\ext^1_{(\varphi,\Gamma)}(t^N\Delta_{X,z,2},\gr_1^{\cF}D_{\rig}(r^{\univer}_{X,z})/\Delta_{X,z,1})=0.\]
Therefore,\;we obtain
\begin{equation}
	\begin{aligned}
		&\dim_E\ext^i_{(\varphi,\Gamma)}(t^N\Delta_{X,z,2},Q_{X,z,1})=\dim_E\ext^i_{(\varphi,\Gamma)}(t^N\Delta_{X,z,2},\fil_2^kD_{\rig}(r^{\univer}_{X,z}))\\
		=\;&\left\{
		\begin{array}{ll}
			1, & \hbox{$i=0$;} \\
			1+(k-1)r^2\cdot d_L, & \hbox{$i=1$;} \\
			0, & \hbox{$i=2$.}
		\end{array}
		\right.
	\end{aligned}
\end{equation}
By the accumulation property at $y_L$,\;the subset of potentially crystalline, non-critical, generic points $z\in S$ satisfying $\wt(\delta_{1,z}^0)-\wt(\delta_{2,z}^0)>N$ is Zariski-dense in $X$.\;By \cite[Corollary 5.4(b)]{bergdall_2017},\;$\ext^0_{(\varphi,\Gamma)}(t^N\Delta_{X,2},Q_{X,1})$ is locally free of rank one at $y_L$.\;After shrinking $X$,\;we may assume that $\ext^0_{(\varphi,\Gamma)}(t^N\Delta_{X,2},Q_{X,1})$ is free of rank one.\;Choosing a generator gives a map
\[\iota_2:t^N\Delta_{X,2}\longrightarrow Q_{X,1}.\]
By \cite[Corollary 5.4(a)]{bergdall_2017},\;$\iota_2$ is injective.\;Now consider $Q_{X,2}:=\cokerr \iota_2$.\;Repeating the argument for $Q_{X,1}:=\cokerr \iota_1$ inductively completes the proof.\;Thus,\;the $(\varphi,\Gamma)$-module $D_{\rig}(r^{\univer}_X)[\frac{1}{t}]$ over $\cR_{X,L}$ admits a filtration $\cM_{\bullet}$ such that $\gr_i^{\cM_{\bullet}}D_{\rig}(r^{\univer}_X)[\frac{1}{t}]\cong \Delta_{X,i}[\frac{1}{t}]$.\;Specializing the filtration $\cM_\bullet$ at the non-critical special point $y_L$ gives a $\omepik$-filtration of $D_{\rig}(\rho_L)[\frac{1}{t}]$ such that
\begin{equation}
	\begin{aligned}
		\gr_i^{\cM_\bullet}D_{\rig}(\rho_L)[1/t]\xrightarrow{\sim}\Delta_{x_i}\tee \cR_{E,L}(z^{\bh_{(i-1)r+1}})[1/t]\cong\Delta_{x_i}\tee \cR_{E,L}(z^{{\bh}_{ir}})[1/t].
	\end{aligned}
\end{equation}
This proves the last assertion.\;
\end{proof}

\subsection{Smoothness of non-critical special point}\label{smoothpoint}

Let $X\subset\xfpx$ be the open affinoid neighborhood of $y_L$ furnished by Theorem \ref{localfiltration}.\;We first prove that $X$ is smooth at $y_L$,\;then we show that $\bersteineigenvarpik$ is smooth  at $y$.\;It suffices to  study the tangent spaces $T_{X,y_L}$ and $T_{\bersteineigenvarpik,y}$.\;Consider the morphisms
\begin{equation}\label{variousmorphisms}
\begin{aligned}
	&\omega:=\omega|_{X}:	X\subset \defvarrho \longrightarrow \sbanpik\times\rigch,\\
	&\zeta:\;X\subset \defvarrho \longrightarrow {\mathfrak{X}_{\overline{r}}^\Box}.\\
\end{aligned}
\end{equation}
Put $\delta_{\beta}^\bh:=(\delta^{\bh}_{\beta,i}:=\unr(q_L^{i-k})z^{\hpi_{ir}})_{1\leq i\leq k}$.\;The induced tangent map of $\omega$ at $y_L$ is
\begin{equation}\label{tangentwxl2}
	\begin{aligned}
		d\omega_{y_L}:	T_{X,y_L} \longrightarrow T_{\sbanpik\times \rigch,\omega(y_L)}&=T_{\sbanpik,\breve{\underline{x}}_{\beta}}\oplus T_{\rigch,\underline{\mathbf{1}}}\\&\xrightarrow[\cong ]{-\cdot (z^{\bh_{ir}})_{1\leq i\leq k}} T_{\rigchl,\delta_{\beta}^\bh}\cong\homo(\bZ_{\beta}(L),E).
	\end{aligned}
\end{equation}		
Indeed, identify $\bv\in T_{X,y_L}$ with an $E[\epsilon]/\epsilon^2$-valued point of $X$.\;Its image in $\rigch$ defines a continuous character $\delta_{\bv}^0=\boxtimes_{i=1}^k\delta_{\bv,i}^0:\bZ_\beta(\cO_L)\rightarrow(E[\epsilon]/\epsilon^2)^\times$ with $\delta_{\bv}^0\equiv\underline{\mathbf{1}}\pmod\epsilon$.\;Under the identification $T_{\rigch,\underline{\mathbf1}}\cong\prod_{i=1}^k\homo(\cO_L^\times,E)$, this character corresponds to $((\delta_{\bv,i}^0-1)/\epsilon)_{1\leq i\leq k}$. On the other hand, recall that
$\mathfrak{Z}_{\omepik}\cong \otimes_{i=1}^kE[z,z^{-1}]^{\eta_{\Omega_r}}$.\;Therefore,\;the $E[\epsilon]/\epsilon^2$-valued points of $\sbanpik$ are identified with the space of unramified characters of $\bZ_{\beta}(L)$ valued in $(E[\epsilon]/\epsilon^2)^\times$.\;

By \cite[Lemma 4.13]{breuil2017smoothness},\;there is an exact sequence of $E$-vector spaces
\begin{equation}\label{dimcaluexatsequence}
0\rightarrow K(\rho_L)\rightarrow T_{\mathfrak{X}^{\Box}_{\overline{r}},\rho_L}\xrightarrow{f_{\rho_L}} \ext^1_{\gal_L}(\rho_L,\rho_L)\cong \ext^1_{(\varphi,\Gamma)}(\Dpik,\Dpik)=F_{\Dpik}(E[\epsilon]/\epsilon^2)\rightarrow 0,\;
\end{equation}
where $K(\rho_L)$ is a vector subspace of $T_{\mathfrak X_{\overline r}^\Box,\rho_L}$ of dimension $n^2-\dim_E\EndO_{\gal_L}(\rho_L)$.

\begin{lem}\label{isotwofunctor} The tangent map of $\zeta$ at $y_L$ induces an injection $T_{X,y_L}\hookrightarrow T_{\mathfrak X_{\overline r}^\Box,\rho_L}$,\;and $f_{\rho_L}(T_{X,y_L})\subseteq F_{\Dpik,\cF}^0(E[\epsilon]/\epsilon^2)$.
\end{lem}
\begin{proof}The closed embedding $X\hookrightarrow\defvarrho\hookrightarrow\mathfrak X_{\overline r}^\Box\times\sbanpik\times\rigch$ gives an injection into the direct sum of the corresponding tangent spaces. The map $d\zeta_{y_L}$ is its projection onto the Galois-deformation factor. If $\bv\in T_{X,y_L}$ maps to zero, then the underlying deformation of $\Dpik$ is trivial.\;Uniqueness of the non-critical $\omepik$-filtration implies that its parameter is also trivial; hence the second tangent-space component vanishes as well. Thus $d\zeta_{y_L}$ is injective.\;For the second assertion, use the filtration over $X$ constructed in Theorem \ref{localfiltration}. Let $A\in\Art_E$ and let $\Spec(A)\rightarrow X$ send its unique closed point to $y_L$. Pullback along $\Spec(A)\rightarrow X\rightarrow\mathfrak X_{\overline r}^\Box$ gives a deformation $\rho_A$ such that
	\[D_{\rig}(\rho_A)[{1}/{t}]\cong A\otimes_{\Gamma(X,\cO_X)}D_{\rig}(r^{\univer}_X)[{1}/{t}],\;\]
	and $A\otimes_{\Gamma(X,\cO_X)}\cM_\bullet $ gives a filtration $\cM_{A,\bullet}$ on $D_{\rig}(\rho_A)[\frac{1}{t}]$.\;Let $\Delta_{A,i}$ be the  pull-back of $\Delta_{X,i}$ along $\Spec(A)\rightarrow X$.\;Therefore,\;the filtration $\cM_{A,\bullet}$ on $D_{\rig}(\rho_A)[\frac{1}{t}]$ satisfies
	\begin{equation}\label{Adefor}
		\begin{aligned}
			\gr_i^{\cM_{A,\bullet}}D_{\rig}(\rho_{A})[{1}/{t}]\xrightarrow{\sim}\Delta_{A,i}[{1}/{t}]
			\xrightarrow{\sim}\Delta_{x}\otimes_{\cR_{E,L}}\cR_{A,L}\big(\unr(a_{A,i} q_L^{i-k})\cdot({\delta}_{A,i})_{\varpi_L}\cdot z^{{\bh}_{(i-1)r+1}}\big)[{1}/{t}]
		\end{aligned}
	\end{equation}
	for $1\leq i\leq k$,\;where $\delta_A$ is the pullback of the universal character of  $\cO_L^\times$ along  $\Spec(A)\rightarrow X$,\;and $a_{A,i}$ is an element of $A$ such that $a_{A,i}\equiv 1\mod \fm_A$.\;Since the
	point $y_L$ is non-critical,\;the uniqueness of $\omepik$-filtration on $\cM_\Dpik$ implies that $\cM_{A,\bullet}\otimes_AE\cong \cM_{\bullet}$,\;i.e.,\;the filtration $\cM_{\bullet}$ on $D_{\rig}(\rho_L)[\frac{1}{t}]$ satisfies
	\[\gr_i^{\cM_\bullet}D_{\rig}(\rho_L)[{1}/{t}]\xrightarrow{\sim}\Delta_{x}\otimes_{\cR_{E,L}}\cR_{E,L}(\delta^{\bh}_{\beta,i})[{1}/{t}],\]
	for $1\leq i\leq k$.\;The filtration $\cM_{A,\bullet}$ induces an increasing filtration $\cF'_A=(\cM_{A,i}\cap\Dpik_A)_{1\leq i\leq k}$ of $\Dpik_A$ by $(\varphi,\Gamma)$-submodules over $\cR_{A,L}$.\;Therefore,
	\[\gr_{i}^{\cF_{A}}\Dpik_A\big[{1}/{t}\big]=\gr_{i}\cM_A\big[{1}/{t}\big]=\Delta_x\otimes_{\cR_{E,L}}\cR_{A,L}(\delta_{A,i})\big[{1}/{t}\big]\]
	for $1\leq i\leq k$. These modules need not be projective over $A$, so $\cF'_A$ may be unsaturated. We first show that $\cF'_E=\cF_E$. An adaptation of \cite[Lemma 2.4.2]{AST2009324R10} shows that $\cF':=\cF'_E$ is saturated. Let $\{\hpi'_{1,\tau},\ldots,\hpi'_{ir,\tau}\}_{\tau\in\Sigma_L}$ be the Hodge--Tate weights of $\fil_i^{\cF'}\Dpik$. The vanishing statements in Lemma \ref{phigammacohononcrtialspecial} induce injections between the $i$-th graded pieces in both directions:
	\begin{equation}
		\begin{aligned}
			&\Delta_x\otimes_{\cR_{E,L}}\cR_{E,L}(\unr(q_L^{i-k})z^{\hpi'_{(i-1)r+1}})\hookrightarrow \gr_i^{\cF'}\Dpik\hookrightarrow \gr_i^{\cF}\Dpik \hookrightarrow \Delta_x\otimes_{\cR_{E,L}}\cR_{E,L}(\unr(q_L^{i-k})z^{\hpi_{ir}}),\\
			&\Delta_x\otimes_{\cR_{E,L}}\cR_{E,L}(\unr(q_L^{i-k})z^{\hpi_{(i-1)r+1}})\hookrightarrow \gr_i^{\cF}\Dpik\hookrightarrow \gr_i^{\cF'}\Dpik \hookrightarrow \Delta_x\otimes_{\cR_{E,L}}\cR_{E,L}(\unr(q_L^{i-k})z^{\hpi'_{ir}}).
		\end{aligned}
	\end{equation}
		Thus $\hpi'_{(i-1)r+1,\tau}\geq\hpi_{ir,\tau}$ and $\hpi_{(i-1)r+1,\tau}\geq\hpi'_{ir,\tau}$ for every $\tau$.\;Non-criticality implies $\hpi'_{j,\tau}=\hpi_{j,\tau}$ for $1\leq j\leq ir$. Uniqueness of the $\omepik$-filtration therefore gives $\cF'_E=\cF_E$, and the same argument shows that $\cF'_A$ reduces to $\cF$ modulo $\fm_A$. Arguing as in \cite[Lemma 2.2.3]{AST2009324R10} completes the proof. In particular, $f_{\rho_L}(T_{X,y_L})\subseteq F_{\Dpik,\cF}^0(E[\epsilon]/\epsilon^2)$.
\end{proof}
Composing $d\omega_{y_L}$ with $\kappa_L$ gives a map
\begin{equation}\label{tangentLinveta}
	\begin{aligned}
		\eta:\;T_{X,y_L}\xrightarrow{d\omega_{y_L}}T_{\rigchl,\delta_{\beta}^\bh}&\cong\homo(\bZ_{\beta}(L),E)\xrightarrow{\kappa_L} \prod_{i=1}^{k-1}\homo(L^\times,E).
	\end{aligned}
\end{equation}

\begin{pro}\label{smoothnessdefvar}$X$ is smooth at $y_L$, and \eqref{tangentLinveta} factors through a surjective map
\begin{equation}\label{defortangentmap1eta}
	\eta:T_{X,y_L}\twoheadlongrightarrow \sL(\rho_L),
\end{equation}
where $\sL(\rho_L):=\sL(\Dpik)$ is the parabolic Fontaine-Mazur simple $\sL$-invariants associated to $\rho_L$.\;
\end{pro}
\begin{proof}We obtain from Lemma \ref{isotwofunctor} a short exact sequence
\begin{equation}\label{dimcaluexatsequenceproof}
	\begin{aligned}
		0\rightarrow K(\rho_L)\cap T_{X,y_L}\rightarrow T_{X,y_L}\xrightarrow{f_{\rho_L}} F_{\Dpik,\cF}^0(E[\epsilon]/\epsilon^2).\;
	\end{aligned}
\end{equation}
Then Proposition \ref{dimtanFD0} gives an upper bound:
\[\dim_ET_{X,y_L}\leq n^2+d_L\Big(\frac{n(n-1)}{2}+k\Big)=\dim_E{\defvarrho}.\]
Since $X$ is equidimensional of this dimension at $y_L$, equality holds and $X$ is smooth at $y_L$. Finally, comparison of the definitions shows that the composition
\begin{equation}
	\begin{aligned}
		T_{X,y_L}\xrightarrow{f_{\rho_L}} F_{\Dpik,\cF}^0(E[\epsilon]/\epsilon^2)\xrightarrow[(\ref{omegadeltapinon})]{\omega_{\Dpik}(E[\epsilon]/\epsilon^2)}\widehat{(\cZ_{{\beta},L})}_{z_{\beta}^\bh}(E[\epsilon]/\epsilon^2)=T_{\rigchl,\delta_{\beta}^\bh}
	\end{aligned}
\end{equation}
coincides with $d\omega_{y_L}$ (see (\ref{tangentwxl2})).\;We deduce from Corollary \ref{corlinvaraint} the surjection (\ref{defortangentmap1eta}).\;	
\end{proof}

We now prove the main proposition of this section.\;Consider the map 
\begin{equation}
\omega^{\infty}:	\bersteineigenvarpik \longrightarrow \sbanpik\times \rigch.\;
\end{equation}
and the tangent map of $\omega$ at point $y=(y^{\fp}, \rho_L, \breve{\underline{x}}_{\beta},\underline{\mathbf{1}})$:
\begin{equation}\label{tangentpatchvar1eta}
d\omega^{\infty}_y:	T_{\bersteineigenvarpik,y} \longrightarrow T_{\sbanpik\times \rigch,\omega(y)}\cong  T_{\rigchl,\delta_{\beta}^\bh},
\end{equation}
where the second isomorphism  follows by the same argument of (\ref{tangentwxl2}).\;
\begin{pro}\label{coreigentangentmap2}$\bersteineigenvarpik$ is smooth at $y$.\;There is  a surjective map
\begin{equation}\label{eigentangentmap2}
\eta^{\infty}:=\kappa_L\circ d\omega^{\infty}_y:	T_{\bersteineigenvarpik,y}\twoheadlongrightarrow \sL(\rho_L).
\end{equation}
\end{pro}
\begin{proof}Recall that $\cE(y)$ is the union of  irreducible components of $\bersteineigenvarpik$ containing $y$.\;By  (\ref{OZfactor}),\;we see that  $\kappa_L\circ d\omega_y$ is equal to the following composition
\begin{equation}\label{tangentpatchvar1}
	\begin{aligned}
		T_{\bersteineigenvarpik,y}\cong T_{\cE(y),y}
		\cong &\;T_{\FX^{\fp} \times \BU^g \times \iota_{\omepik}^{-1}(X),(y^{\fp},y_L)}\twoheadlongrightarrow T_{X,y_L}\xrightarrow[(\ref{tangentLinveta})]{\eta}\prod_{1\leq i\leq k-1}\homo(L^\times,E).\;
	\end{aligned}
\end{equation}
Now the result follows.\;
\end{proof}


\subsection{Local-global compatibility and the main theorem}\label{lgcompthemainSECTION}

Suppose that $\rho_L$ appears in the patched Bernstein eigenvariety $\bersteineigenvarpik$; that is, suppose there exist a point $y^{\fp}\in (\Spf\;R_{\infty}^{\fp})^{\rig}$ and a pair $(\pi',\chi)\in \sbanpik\times\rigch$ such that $x:=(y^{\fp}, \rho_L, \pi',\chi)\in \bersteineigenvarpik$.\;Let $\fm_y$ be the maximal ideal of $R_{\infty}[1/p]$ corresponding to the point $y:=(y^{\fp}, \rho_L)$ of $(\Spf R_{\infty})^{\rig}$.\;Recall the continuous $R_{\infty}$-admissible patched unitary representation $\Pi_{\infty}$ of $G$ over $E$ from Section \ref{BENVARPARAVAR}.\;Then $\Pi_{\infty}[\fm_y]$ (resp.,\;$\Pi^{R_{\infty}-\ana}_{\infty}[\fm_y]$) is an admissible unitary Banach (resp., locally $\bQ_p$-analytic) representation of $G$, which one might expect to be the appropriate representation (up to multiplicities) corresponding to $\rho_L$ in the $p$-adic local Langlands program.\;Suppose the following conditions hold:
\begin{itemize}
\item[(a)] $\rho_L$ is potentially semi-stable.\;The  $(\varphi,\Gamma)$-module $\Dpik=D_{\rig}(\rho_L)$ over $\cR_{E,L}$ admits a non-critical special $\omepik$-filtration  with parameter $(\underline{x}_{\beta,\bh},z_{\beta}^\bh)\in \sbanpik\times\rigch $ (see Definition \ref{dfnnoncriticalspecial});
\item[(b)] The monodromy operator $N$ on $D_{\pst}(\rho_L)$ satisfies $N^{k-1}\neq 0$.
\end{itemize}
By the argument preceding \cite[Corollary 3.1.11]{Ding2021}, the finite self-twist group $\mu_{\omepik}$ acts on $\sbanpik\times\rigch$. For ${\psi}=(\psi_i)_{1\leq i\leq k}\in\mu_{\omepik}$,
\[{\psi}\cdot((x_i)_{1\leq i\leq k},(\chi_i)_{1\leq i\leq k})=((x_i\te \unr(\psi_i(\varpi_L)))_{1\leq i\leq k},(\chi_i\psi_i|_{\cO_L^\times})_{1\leq i\leq k}).\]
Combining Propositions \ref{basicpropertyeigen}(1) and \ref{propertyparavar}(2), the theory of $\omepik$-filtrations in families \cite[Appendix A.1]{Ding2021} and \cite[Corollary 3.1.11]{Ding2021}, and the uniqueness of the $\omepik$-filtration with parameter $(\underline{x}_{\beta,\bh},z_\beta^\bh)$, we obtain the following lemma.
\begin{lem}\label{detlameda}
For $(\pi',\chi)\in \sbanpik\times\rigch$, the following condition holds:
\[\homo_{\bL_{\beta}(L)}\Big(\pi'\te\chi_{\varpi_L}\te \lamoverx_{\Delta^{\beta}}, J_{\bP_{\beta}}\big(\Pi_{\infty}^{R_{\infty}-\ana}[\fm_y]\big)\Big)\neq0\]
if and only if $(\pi',\chi)$ belongs to the $\mu_{\omepik}$-orbit of $(\breve{\underline{x}}_\beta,\underline{\mathbf1})$.
\end{lem}

For $\alpha:=\alpha_x\in E^{\times}$, set $\pi_{\beta}(\alpha,\bm{\lambda}):=\unr(\alpha)\circ\mathrm{det}\te\pi_{\beta}\te \lamoverx_{\Delta^{\beta}}$ and $\pi_{\beta}(\alpha):=\pi_{\beta}(\alpha,\underline{0})$.\;For every $\psi\in \eta_{\Omega_r}$,
\begin{equation}
	\begin{aligned}
		&\pi_{\psi\cdot \breve{\underline{x}}_{\beta}}\te (\psi\cdot \underline{\mathbf{1}})_{\varpi_L}= \pi_{\beta}(\alpha,\bm{\lambda})\te \delta_{\bP_{\beta}},\;
	\end{aligned}
\end{equation}
 Therefore, Lemma \ref{detlameda} implies
\begin{equation}\label{jPinj}
\homo_{\bL_{\beta}(L)}\big(\pi_{\beta}(\alpha,\bm{\lambda})\te \delta_{\bP_{\beta}}, J_{\bP_{\beta}}\big(\Pi_{\infty}^{R_{\infty}-\ana}[\fm_y]\big)\big)\neq0.
\end{equation}
We have natural maps
\begin{equation}\label{adjuncitonformula}
\begin{aligned}
	\homo_{G}\big(\bI_{\emptyset}(\alpha,\bm\beta,{\bm\lambda}),\Pi_{\infty}^{R_{\infty}-\ana}[\fm_y]\big)&\longrightarrow\homo_{\bL_{\beta}(L)}\big(J_{\bP_{\beta}}\big(\bI_{\emptyset}(\alpha,\bm\beta,{\bm\lambda})\big), J_{\bP_{\beta}}\big(\Pi_{\infty}^{R_{\infty}-\ana}[\fm_y]\big)\big)\\
	&\longrightarrow \homo_{\bL_{\beta}(L)}\big(\pi_{\beta}(\alpha,\bm{\lambda})\te \delta_{\bP_{\beta}}, J_{\bP_{\beta}}\big(\Pi_{\infty}^{R_{\infty}-\ana}[\fm_y]\big)\big),
\end{aligned}
\end{equation}
where the first map is induced by applying the Jacquet-Emerton functor, and the second by the injections
\[\pi_{\beta}(\alpha,\bm{\lambda})\te \delta_{\bP_{\beta}}\hookrightarrow J_{\bP_{\beta}}\big(i_{\emptyset}(\alpha,\bm\beta,{\bm\lambda}\big)\big)\hookrightarrow J_{\bP_{\beta}}\big(\bI_{\emptyset}(\alpha,\bm\beta,{\bm\lambda}\big)\big)\]
(note that $J_{\bP_{\beta}}(i_{\emptyset}(\alpha,\bm\beta,{\bm\lambda}))$ is semisimple and $J_{\bP_{\beta}}\big(\st^{\infty}_{n}(\alpha,\bm\beta,{\bm\lambda})\big)=\pi_{\beta}(\alpha,\bm{\lambda})\te \delta_{\bP_{\beta}}$, by the same argument as in the proof of \cite[Corollary 3.6]{2022ext1hyq}).\;We first establish an adjunction formula, following the approach of \cite[Proposition 4.7]{2019DINGSimple} and \cite{bergdall2018adjunction}.

\begin{pro}\label{proadjuncitonformularight}
The natural map constructed in (\ref{adjuncitonformula})
\begin{equation}\label{adjuncitonformularight}
	\begin{aligned}
		\homo_{G}\Big(&\bI_{\emptyset}(\alpha,\bm\beta,{\bm\lambda}),\; \Pi_{\infty}^{R_{\infty}-\ana}[\fm_y]\Big)\lra \homo_{\bL_{\beta}(L)}\Big(\pi_{\beta}(\alpha,\bm{\lambda})\te \delta_{\bP_{\beta}}, J_{\bP_{\beta}}\big(\Pi_{\infty}^{R_{\infty}-\ana}[\fm_y]\big)\Big)
	\end{aligned}
\end{equation}
is bijective.\;The same assertion holds when $\bI_{\emptyset}(\alpha,\bm\beta,{\bm\lambda})$ is replaced by any subrepresentation $W$ satisfying $\st_n^\infty(\alpha,\bm\beta,{\bm\lambda})\subseteq W\subseteq\st_n^\ana(\alpha,\bm\beta,{\bm\lambda})$. In particular, there is an injection of $G$-representations
\begin{equation}\label{steinberginj}
	\st_n^{\infty}(\alpha,\bm\beta,{\bm\lambda})\hookrightarrow\Pi_{\infty}^{R_{\infty}-\ana}[\fm_y].\;
\end{equation}
\end{pro}
\begin{proof}We divide the proof into the following steps.\;\\
\textbf{(a)} Let $f\in \homo_{G}\big(\bI_{\emptyset}(\alpha,\bm\beta,{\bm\lambda}), \Pi_{\infty}^{R_{\infty}-\ana}[\fm_y]\big)$ be a nonzero map.\;Lemma \ref{detlameda}, \cite[Theorem 4.3]{breuil2015II}, and \cite[Corollary 3.4]{breuil2016socle} imply that $f$ factors through $\st^{\ana}_{n}(\alpha,\bm\beta,{\bm\lambda})$ and induces a nonzero map $\st^{\infty}_{n}(\alpha,\bm\beta,{\bm\lambda})\hookrightarrow \Pi_{\infty}^{R_{\infty}-\ana}[\fm_y]$.\;Indeed, suppose that $W'$ is an irreducible constituent of $\bI_I(\alpha,\bm\beta,{\bm\lambda})$ for some $\emptyset\neq I\subset \Delta_{\beta}$.\;If $W'$ is locally algebraic, then it is an irreducible constituent of $i_{I}(\alpha,\bm\beta,\bm{\lambda}_\bh)$.\;However,
\[\homo_{\bL_{\beta}(L)}\Big( J_{\bP_{\beta}}\big(i_I(\alpha,\bm\beta,{\bm\lambda})\big), J_{\bP_{\beta}}\big(\Pi_{\infty}^{R_{\infty}-\ana}[\fm_y]\big)\Big)=0\]
by Lemma \ref{detlameda}.\;Thus $W'$ is not locally algebraic, and \cite[Theorem]{orlik2014jordan} implies that it has the form $\mathcal{F}_{\op_{\beta,J}}^G(L(-s\cdot{\bm\lambda}),W'')$
with $s\in \sW_{n,\Sigma_{L}}$ satisfying $I\subseteq J$ and $s\cdot{\bm\lambda}\in X^+_{\Delta^\beta\cup J}$, where $W''$ is an irreducible constituent of $\big(\ind_{\op_{\beta,J}(L)\cap\bL_{\beta,J}(L) }^{\bL_{\beta,J}(L)}{\pi}_{\beta,I}\big)^{\infty}$.\;It follows from \cite[Th\'{e}or\`{e}me 4.3,\;Remarque 4.4 (i),\;Corollaire 4.5]{breuil2015II} that
\begin{equation*}
	\begin{aligned}
		0\neq \homo_{G}\Big(&\mathcal{F}_{\op_{\beta}}^G\big(\overline{M}(-s\cdot{\bm\lambda})^\vee,\pi_{\beta}(\alpha)\big),\;\Pi_{\infty}^{R_{\infty}-\ana}[\fm_y]\Big)\cong \homo_{\bL_{\beta}(L)}\Big(\pi_{\beta}(\alpha)\te L(s\cdot{\bm\lambda}), J_{\bP_{\beta}}\big(\Pi_{\infty}^{R_{\infty}-\ana}[\fm_y]\big)\Big),
	\end{aligned}
\end{equation*}
contradicting Lemma \ref{detlameda}.\;Similarly, $f(V_0)=0$ for every irreducible constituent $V_0\neq \st_n^\infty(\alpha,\bm\beta,{\bm\lambda})$ of $\st^{\ana}_{n}(\alpha,\bm\beta,{\bm\lambda})$.\;Thus $f$ factors through $\st^{\ana}_{n}(\alpha,\bm\beta,{\bm\lambda})$.\;By the left exactness of $J_{\bP_\beta}$, it therefore induces a nonzero map $\st_n^\infty(\alpha,\bm\beta,{\bm\lambda})\hookrightarrow\Pi_\infty^{R_\infty-\ana}[\fm_y]$, proving \eqref{steinberginj}.\;Moreover, the map in (\ref{adjuncitonformularight}) is injective.\\
\textbf{(b)}\;The arguments in the proofs of \cite[Theorem 4.8]{bergdall2018adjunction} and \cite[Proposition 4.7]{2019DINGSimple} apply in our case, although our input differs slightly from that in those results.\;We indicate the necessary modifications below.\;Recall that $M_{\infty}$ is finite projective over $S_{\infty}[[\GLN_n(\cO_L)]]$.\;It follows from \cite[Corollary 3.9]{breuil2017interpretation} that $\Pi_{\infty}^{R_{\infty}-\ana}$ is a direct summand of $\cC^{\bQ_p-\ana}\big(\BZ_p^s \times \GLN_n(\cO_L),E\big)$ as a $\GLN_n(\cO_L)$-representation, where $s=n^2(|S_p|+1)+q$.\;Let $\fm \subset S_{\infty}$ be the preimage of $\fm_y$ under the morphism $S_{\infty} \ra R_{\infty}$.\\
\textbf{(c)}\;Set $V:=\Pi_{\infty}^{R_{\infty}-\ana}[\fm]$.\;This is an admissible Banach representation of $G$ equipped with a continuous action of $R_{\infty}$.\;There exists a pro-$p$ uniform compact open subgroup $H$ of $\GLN_n(\cO_L)$ such that $V|_{H}\cong \cC^{\bQ_p-\ana}(H,E)^{\oplus r}$ for some $r\in \BZ_{\geq 1}$.\;By Lemma \ref{fpacycilic} below, $V$ satisfies the first hypothesis of \cite[Theorem 4.8]{bergdall2018adjunction}.\;On the other hand, Lemma \ref{detlameda} shows that $(U,\pi)=\big(\lamoverx,\pi_{\beta}(\alpha)\te \delta_{\bP_{\beta}}\big)$ is non-critical with respect to $\Pi_{\infty}^{R_{\infty}-\ana}[\fm_y]$ (in the terminology of \cite[Definition 4.4]{bergdall2018adjunction}).\;Recall that $\cC^{\lp}_c(\bN_{\beta}(L),\pi_{\beta}(\alpha,\bm{\lambda}))$ (resp., $\cC^{\infty}_c(\bN_{\beta}(L),\pi_{\beta}(\alpha))$) denotes the space of locally polynomial (resp., smooth) functions on $\bN_{\beta}(L)$ with values in $\pi_{\beta}(\alpha,\bm{\lambda})$ (resp., $\pi_{\beta}(\alpha)$) and compact support.\;These spaces carry the topologies defined in \cite[(2.5)]{emerton2007jacquet} (resp., see \cite[(2.2)]{emerton2007jacquet}) and are convex $E$-vector spaces of compact type.\;For brevity, write $A^{\lp}:=\cC^{\lp}_c(\bN_{\beta}(L),\pi_{\beta}(\alpha,\bm{\lambda}))$ and $A^{\infty}:=\cC^{\infty}_c(\bN_{\beta}(L),\pi_{\beta}(\alpha))$.\;We use a commutative diagram analogous to the one in \cite[proof of Theorem 4.8]{bergdall2018adjunction}, where $V$ is replaced by $V[\fm_y]$ and $(U,\pi)$ by $\big(\lamoverx,\pi_{\beta}(\alpha)\te \delta_{\bP_{\beta}}\big)$:
\[\xymatrix{
	\homo_{G}\big(\bI_{\emptyset}(\alpha,\bm\beta,{\bm\lambda}),V[\fm_y]\big) \ar[d]^{\simeq}_{(a)} \ar[r]^{(1)\hspace{40pt}} & \homo_{\bL_{\beta}(L)}\big(\pi_{\beta}(\alpha,\bm{\lambda})\te \delta_{\bP_{\beta}}, J_{\bP_{\beta}}\big(V[\fm_y]\big)\big) \ar[d]^{\simeq}_{(b)}  \\
	\homo^\#\big(A^{\lp},V[\fm_y]\big)  \ar[r]^{(2)\hspace{30pt}} \ar[d]^{\simeq}_{(c)} & \homo^\#\big(U(\mathfrak{g}_{L})\otimes_ {U(\mathfrak{p}_{L})}A^{\lp},V[\fm_y]\big)\ar[d]^{\simeq}_{(d)}\\
	\homo^\#\big(M({\bm\lambda})^\vee\otimes_EA^{\infty},V[\fm_y]\big)  \ar[r]^{(3)} \ar[d]^{\eta_1} & \homo^\#\big(M({\bm\lambda})\otimes_EA^{\infty},V[\fm_y]\big)\\
	\homo^\#\big(L({\bm\lambda})\otimes_EA^{\infty},V[\fm_y]\big),  \ar[ur]^{\eta_2} &
}\]
where we write $\homo^\#:=\homo_{(\mathfrak{g}_{L},\bP_{\beta}(L))}$ for brevity.\;The map $(1)$ is the map in (\ref{adjuncitonformularight}).
For explanations of the terms and maps $(1)$, (2), (a), (b), (c), and (d), see \cite[Theorem 4.8]{bergdall2018adjunction} or \cite[Proposition 4.7]{2019DINGSimple}.\;The maps $\eta_1$, $\eta_2$, and $(3)$ are induced by the natural morphisms $M({\bm\lambda})\twoheadrightarrow \lamoverx\hookrightarrow M({\bm\lambda})^\vee$.\;By Lemma \ref{detlameda} and arguments analogous to those in \cite[Proposition 4.9]{bergdall2018adjunction}, one can show that $\eta_2$ is bijective and $\eta_1$ is injective.

To prove that $\eta_1$ is surjective, we need a generalization of Step (c) in the proof of \cite[Proposition 4.7]{2019DINGSimple}.\;It suffices to show that, for any pair $(M,M')$ such that $M'/M=L(s\cdot{\bm\lambda})$ in $\cO_{\alge}^{\overline{\fp}_{\beta,L}}$ with $1\neq s\cdot{\bm\lambda}\in X^+_{\Delta}$, the restriction map
\begin{equation*}
	\begin{aligned}
		\homo_{(\mathfrak{g}_{L},\bP_{\beta}(L))}&\big(M'\otimes_EA^{\infty},V[\fm_y]\big)\longrightarrow \homo_{(\mathfrak{g}_{L},\bP_{\beta}(L))}\big(M\otimes_EA^{\infty},V[\fm_y]\big)
	\end{aligned}
\end{equation*}
is surjective. Given a $(\mathfrak{g}_{L},\bP_{\beta}(L))$-equivariant morphism $f$ in the target, we have
\begin{equation}\label{proofadjformla}
	M\otimes_EA^{\infty}\rightarrow V[\fm_y]\hookrightarrow V.
\end{equation}
The space $V$ has a natural $R_\infty[1/p]$-action.\;Equip the source of \eqref{proofadjformla} with the action induced by $R_\infty[1/p]\twoheadrightarrow R_\infty[1/p]/\fm_y\cong E$; then $f$ is also $R_\infty[1/p]$-equivariant.\;Let $V'$ be the pushout of $M'\otimes_EA^\infty$ along $f$.\;We obtain an exact sequence of $(\mathfrak g_{L},\bP_\beta(L))$-modules with continuous $R_\infty[1/p]$-actions:
\begin{equation}\label{adjuncitonexact1}
	0\rightarrow V \rightarrow  V' \xrightarrow{\beta_0} L(s\cdot{\bm\lambda})\otimes_EA^{\infty}
	\rightarrow0.\;
\end{equation}
It suffices to construct a section of $\beta_0$ that is a $(\mathfrak{g}_{L},\bP_{\beta}(L))\times R_{\infty}[\frac{1}{p}]$-equivariant morphism.\;Indeed, if we construct a section $s_0$ of $\beta_0$, then $s_0\circ f'$ gives the desired lift of $f$ to $M'\otimes_EA^{\infty}$ and hence a $(\mathfrak{g}_{L},\bP_{\beta}(L))$-morphism $M'\otimes_EA^{\infty}\rightarrow V[\fm_y]$.\;Pulling back the exact sequence (\ref{adjuncitonexact1}) along the quotient map $M(s\cdot{\bm\lambda})\twoheadrightarrow L(s\cdot{\bm\lambda})$, we obtain an exact sequence
\begin{equation}\label{adjuncitonexact2}
	0\rightarrow V \rightarrow  V'' \xrightarrow{\beta_1}  M(s\cdot{\bm\lambda})\otimes_EA^{\infty}
	\rightarrow0.\;
\end{equation}
By Lemma \ref{detlameda} and the argument following \cite[(80)]{2019DINGSimple} or \cite[(15)]{bergdall2018adjunction}, it suffices to construct a $(\mathfrak{g}_{L},\bP_{\beta}(L))\times R_{\infty}[\frac{1}{p}]$-equivariant map $M(s\cdot{\bm\lambda}) \otimes_EA^{\infty}\rightarrow V''$.\;It remains to construct a $\bP_{\beta}(L)\times R_{\infty}[\frac{1}{p}]$-equivariant map
\begin{equation}\label{desiredhomo1}
	L(s\cdot{\bm\lambda}) \otimes_EA^{\infty}\rightarrow  V''.
\end{equation}
By applying Proposition \ref{emertongenerpara} below and taking the $\fm_y$-generalized eigenspaces for $R_{\infty}[\frac{1}{p}]$, we obtain an $R_{\infty}[\frac{1}{p}]$-equivariant exact sequence of finite-dimensional $E$-vector spaces
\begin{equation}
	\begin{aligned}
		0\rightarrow B_{\omepik,{\bm\lambda}}(V)&[\fz_{\beta}=d\omega_{s\cdot{\bm\lambda}}][\fm^\infty,\fm_y^\infty][
		\fm_{\omega_{s\cdot{\bm\lambda}}}]\rightarrow B_{\omepik,{\bm\lambda}}(V'')[\fz_{\beta}=d\omega_{s\cdot{\bm\lambda}}][\fm^\infty,\fm_y^\infty][
		\fm_{\omega_{s\cdot{\bm\lambda}}}]\\
		&\rightarrow B_{\omepik,{\bm\lambda}}(L(s\cdot{\bm\lambda})\otimes_EA^{\infty})[\fz_{\beta}=d\omega_{s\cdot{\bm\lambda}}][\fm^\infty,\fm_y^\infty][
		\fm_{\omega_{s\cdot{\bm\lambda}}}]\rightarrow 0.
	\end{aligned}
\end{equation}
where $\omega_{s\cdot{\bm\lambda}}$ denotes the central character of $\pi_{\beta}\te L(s\cdot{\bm\lambda})$, and $\fm$ denotes the maximal ideal of $\mathfrak{Z}_{\omepik}$ associated with $\pi_{\beta}(\alpha)\te \delta_{\bP_{\beta}}$.\;By Lemma \ref{detlameda}, we deduce that $B_{\omepik,{\bm\lambda}}(V)[\fz_{\beta}=d\omega_{s\cdot{\bm\lambda}}][\fm^\infty,\fm_y^\infty][
\fm_{\omega_{s\cdot{\bm\lambda}}}]=0$.\;Hence we obtain an isomorphism
\begin{equation}\label{526}
	\begin{aligned}
		&B_{\omepik,{\bm\lambda}}(V'')[\fz_{\beta}=d\omega_{s\cdot{\bm\lambda}}][\fm^\infty,\fm_y^\infty][
		\fm_{\omega_{s\cdot{\bm\lambda}}}]\\
		\cong\;&B_{\omepik,{\bm\lambda}}(L(s\cdot{\bm\lambda})\otimes_EA^{\infty})[\fz_{\beta}=d\omega_{s\cdot{\bm\lambda}}][\fm^\infty,\fm_y^\infty][
		\fm_{\omega_{s\cdot{\bm\lambda}}}]\\ \cong\;&\pi_{\beta}(\alpha)\te \delta_{\bP_{\beta}}\te L(s\cdot{\bm\lambda}),
	\end{aligned}
\end{equation}
where the second isomorphism follows from \cite[Lemma 3.5.2]{emerton2006jacquet}.\;The map in (\ref{desiredhomo1}) is then induced by the inverse of (\ref{526}) and \cite[Theorem\;3.5.6]{emerton2006jacquet}.
\end{proof}

The argument in the proof of \cite[Example 4.2]{bergdall2018adjunction} yields the following lemma.
\begin{lem}\label{fpacycilic}Suppose that $V$ is an admissible continuous Banach representation of $G$ and that there exists a compact open subgroup $H$ such that $V|_H\xrightarrow{\sim}C^0(H,E)^{\oplus l}$. Then
\begin{equation}
	\hH^i\big(\overline{\fp}_{\beta,L},V^{\bQ_p-\ana}\otimes U^\vee\big)=0
\end{equation}
for every $i\in \BZ_{\geq 1}$ and every finite-dimensional $E$-linear locally analytic representation $U$ of $\bL_{\beta}(L)$.\;
\end{lem}

The following proposition generalizes \cite[Proposition 4.1]{breuil2007first}.\;Let $\mu$ be an integral weight of $\fz$.\;Let $\mathcal{E}_{\omepik,\fp,{\bm\lambda}}^\infty(\overline{\rho})_{\mu}$ be the fibre of $\bersteineigenvarpik$ at $\mu_{\beta}$ under the morphism
\[\bersteineigenvarpik\hookrightarrow \FX_{\infty}\times \sbanpik\times \rigch\rightarrow \rigch\rightarrow \fz_{\beta}^\vee,\;\]where $\rigch\rightarrow \fz_{\beta}^\vee$ is the differentiation map $\chi\mapsto d\chi$.\;For an $E$-algebra $A$, an ideal $\fm\subset A$, and an $A$-module $M$, let $M[\fm^{\infty}]$ denote the $A$-submodule of elements annihilated by $\fm^n$ for some $n\geq 0$.\;We use $\iota_0(\fz_{\beta})$ to distinguish the action on $B_{\omepik,{\bm\lambda}}(\Pi_\infty^{R_\infty-\ana})$ induced by $\cZ_0$.\;Then \cite[Lemma 3.1.4]{Ding2021} shows that
\begin{equation}
\begin{aligned}
	B_{\omepik,{\bm\lambda}}(\Pi_\infty^{R_\infty-\ana})[\iota_0(\fz_{\beta})=\mu]
	=&\;\bigoplus_{\delta,\chi}B_{\omepik,{\bm\lambda}}(\Pi_\infty^{R_\infty-\ana})[\iota_0(\fz_{\beta})=\mu][\fm_\chi][\fm_\delta^\infty]\\
	=&\;\bigoplus_{\fm\in \MSpec\mathfrak{Z}_{\omepik},\chi}B_{\omepik,{\bm\lambda}}(\Pi_\infty^{R_\infty-\ana})[\iota_0(\fz_{\beta})=\mu][\fm_\chi][\fm^\infty],\;
\end{aligned}
\end{equation}
where $\delta$ (resp. $\chi$) runs over the smooth characters of $\Delta_0$ (see \cite[(3.4)]{Ding2021}) (resp. the locally algebraic characters of $\bZ_{\beta}(\cO_L)$ of weight $\mu$), and $\fm_\delta\subset E[\Delta_0]$ (resp. $\fm_\chi\subset E[\bZ_{\beta}(\cO_L)]$) is the maximal ideal associated with $\delta$ (resp. $\chi$).\;Replacing the character $\delta$ of $\Delta_0$ with maximal ideals $\fm$ of $\mathfrak{Z}_{\omepik}$ gives the second equality, by the proof of \cite[Lemma 3.1.4]{Ding2021}.\;By \cite[Lemma 3.1.4]{Ding2021}, each term in the direct sums is finite-dimensional over $E$.\;
\begin{pro}\label{emertongenerpara}Suppose that $V$ is an admissible continuous Banach representation of $G$ and that there exists a compact open subgroup $H$ such that $V|_H\xrightarrow{\sim}C^0(H,E)^{\oplus l}$. Suppose there is an exact sequence of admissible locally $\bQ_p$-analytic representations of $G$:
\begin{equation}\label{exactanaylytic}
	0\rightarrow V^{\bQ_p-\ana}\rightarrow \Pi\rightarrow \Pi_1\rightarrow 0,
\end{equation}
Let $\chi:\bZ_{\beta}(\cO_L)\rightarrow E^\times$ be a continuous character of weight $\mu$, and let $\fm$ be a maximal ideal of $\mathfrak{Z}_{\omepik}$. Then there are short exact sequences
\begin{equation}\label{techexactseq1}
	\begin{aligned}
		0\rightarrow (V^{\bQ_p-\ana})^{\bN(L)_0}[\fz_{\beta}=\mu,\fd=0]\rightarrow& \Pi^{\bN(L)_0}[\fz_{\beta}=\mu,\fd=0] \rightarrow \Pi_1^{\bN(L)_0}[\fz_{\beta}=\mu,\fd=0]\rightarrow 0,
	\end{aligned}
\end{equation}
\begin{equation}\label{techexactseq2}
	\begin{aligned}
		0\rightarrow B_{\omepik,{\bm\lambda}}(V^{\bQ_p-\ana})[\fz_{\beta}=\mu][\fm^\infty][
		\fm_{\chi}]\rightarrow B_{\omepik,{\bm\lambda}}(\Pi)&[\fz_{\beta}=\mu][\fm^\infty][
		\fm_{\chi}]\\
		&\rightarrow B_{\omepik,{\bm\lambda}}(\Pi_1)[\fz_{\beta}=\mu][\fm^\infty][
		\fm_{\chi}]\rightarrow 0.
	\end{aligned}
\end{equation}
Moreover, all the vector spaces in the second exact sequence are finite-dimensional.\;
\end{pro}
\begin{proof}The proposition implies that taking $\overline{\fp}_{L}$-invariants is exact on the short exact sequence (\ref{exactanaylytic}). Since taking $\bN(L)_0$-invariants is exact in the category of smooth representations of $\bN(L)_0$, we obtain the first exact sequence (\ref{techexactseq1}).\;To prove the exactness of (\ref{techexactseq2}), we unwind the actions of $\Delta_0$, $\cZ_0$, and $\cZ_1(\bZ_{\beta}(\cO_L))$ on $B_{\omepik,{\bm\lambda}}(W)$ for $W\in\{V^{\bQ_p-\ana},\Pi,\Pi_1\}$.\;As in the proof of \cite[Lemma 3.1.4]{Ding2021}, we can choose a compact open subgroup of $\bZ_{\beta}(\cO_L)$ such that $B_{\omepik,{\bm\lambda}}(W)[\fz_{\beta}=\mu]$ is a $\Delta_0$-equivariant direct summand of
\begin{equation*}
	\begin{aligned}
		\big(J_{\bP_{\beta}}&(W)_{{\bm\lambda}}[\fz_{\beta}=\mu]\widehat{\otimes}_E\cC^{\bQ_p-\ana}(\bZ_{\beta}(\cO_L),E)[\fz_{\beta}=\mu]\otimes_E\sigma^\vee\big)^H\\
		\cong\;\bigoplus_{\delta'}&\big(\big(J_{\bP_{\beta}}(W)_{{\bm\lambda}}\otimes_E\sigma^\vee\big)^{H^D}[\fz_{\beta}=\mu][\fm_{\delta'}^\infty]\widehat{\otimes}_E\cC^{\bQ_p-\ana}(\bZ_{\beta}(\cO_L),E)[\fz_{\beta}=\mu]\big)^{Z_H},		
	\end{aligned}
\end{equation*}
where $H^D:=H\cap \bD(\cO_L)$ (resp., $Z_H:=H\cap\bZ_{\beta}(\cO_L)$), and $\delta'$ runs over the locally algebraic characters of $\cZ_1(\bZ_{\beta}(\cO_L))$ (with $[\fm_{\delta'}^\infty]$ denoting the generalized eigenspace for the corresponding $\cZ_1(\bZ_{\beta}(\cO_L))$-action).\;Following the proof of \cite[Proposition 4.1,\;(4.3),\;Pages 10521-10522]{breuil2020towards}, we obtain a short exact sequence of finite-dimensional $E$-vector spaces:
\begin{equation*}
	\begin{aligned}
		0 \rightarrow \big(J_{\bP_{\beta}}(V^{\bQ_p-\ana})_{{\bm\lambda}}\otimes_E\sigma^\vee\big)^{H^D}[\fz_{\beta}=\mu][\fm_{\delta'}^\infty] &\rightarrow \big(J_{\bP_{\beta}}(\Pi)_{{\bm\lambda}}\otimes_E\sigma^\vee\big)^{H^D}[\fz_{\beta}=\mu][\fm_{\delta'}^\infty]\\ & \rightarrow \big(J_{\bP_{\beta}}(\Pi_1)_{{\bm\lambda}}\otimes_E\sigma^\vee\big)^{H^D}[\fz_{\beta}=\mu][\fm_{\delta'}^\infty] \rightarrow 0.\;
	\end{aligned}
\end{equation*}
The result follows by taking the $[\fm_\chi]$-eigenspaces and replacing the character $\delta'$ of $\Delta_0$ with maximal ideals $\fm$ of $\mathfrak{Z}_{\omepik}$, as in the proof of \cite[Lemma 3.1.4]{Ding2021}.\;
\end{proof}


Let $\FI_y\subseteq \fm_y$ be a closed ideal of $R_{\infty}[1/p]$ such that $\dim_E (R_{\infty}[1/p]/\FI_y) < +\infty$ and $\fm_y$ is the unique closed maximal ideal containing $\FI_y$ (e.g., $\FI_y=\fm_y^k$).\;Following \cite[Corollary 4.9]{2015Ding}, we obtain
\begin{cor}\label{cor: lgln-bij3}
Let $\st_n^{\infty}(\alpha,\bm\beta,{\bm\lambda})\subseteq W\subseteq \st_n^{\ana}(\alpha,\bm\beta,{\bm\lambda})$.\;If a morphism $f: W\rightarrow \Pi_{\infty}^{R_{\infty}-\ana}[\FI_y]$ satisfies $f|_{\st_n^{\infty}(\alpha,\bm\beta,{\bm\lambda})}\subseteq \Pi_{\infty}^{R_{\infty}-\ana}[\fm_y]$, then the image of $f$ lies in $\Pi_{\infty}^{R_{\infty}-\ana}[\fm_y]$.
\end{cor}

\begin{rmk}
The central character of $\pi_{\beta}(\alpha,{\bm\lambda})$ is unitary (that is, it takes values in $\cO_E^{\times}$).\;By \cite[Proposition 1.18]{EmpLfun},\;$\pi_{\beta}(\ul{\lambda})$ admits a universal unitary completion $\widehat{V}$, which is non-admissible if $r>1$.\;By \cite[Proposition 3.1]{breuil2020towards}, the universal unitary completion of $i_{\emptyset}(\alpha,\bm\beta,{\bm\lambda})$ is $\widehat{i_{\emptyset}(\alpha,\bm\beta,{\bm\lambda})}:=\big(\ind_{\op_{\beta}(L)}^{G}\widehat{V}\big)^{\cC^0}$.\;Therefore,
\begin{equation}
	\begin{aligned}
		\homo_{G}\big(\widehat{i_{\emptyset}(\alpha,\bm\beta,{\bm\lambda})},\Pi_{\infty}[\fm_y]\big)\xrightarrow{\sim}	\homo_{G}\big(i_{\emptyset}(\alpha,\bm\beta,{\bm\lambda}),\Pi_{\infty}[\fm_y]\big).\;
	\end{aligned}
\end{equation}
Passing to the locally $\bQ_p$-analytic vectors yields
\[\xymatrix{ 
	 	\homo_{G}\big(\big(\widehat{i_{\emptyset}(\alpha,\bm\beta,{\bm\lambda})}\big)^{\bQ_p-\ana},\Pi^{R_{\infty}-\ana}_{\infty}[\fm_y]\big) \ar[r]^{\quad\sim} \ar[d]^{\sim} &  	\homo_{G}\big(i_{\emptyset}(\alpha,\bm\beta,{\bm\lambda}),\Pi^{R_{\infty}-\ana}_{\infty}[\fm_y]\big) \ar[d]^{\sim}   \\
\homo_{G}\big(\big(\widehat{i_{\emptyset}(\alpha,\bm\beta,{\bm\lambda})}\big)^{\bQ_p-\ana}/\bI^-_{\emptyset}(\alpha,\bm\beta,{\bm\lambda}),\Pi^{R_{\infty}-\ana}_{\infty}[\fm_y]\big) \ar[r]^{\hspace{40pt}\sim} & \homo_{G}\big(	\st_n^{\infty}(\alpha,\bm\beta,{\bm\lambda}),\Pi^{R_{\infty}-\ana}_{\infty}[\fm_y]\big), }\]
where $\bI^-_{\emptyset}(\alpha,\bm\beta,{\bm\lambda}):=\sum_{\emptyset\neq I\subseteq \Delta_{\beta}} \bI_I(\alpha,\bm\beta,{\bm\lambda})$.\;Since $\st_n^{\infty}(\alpha,\bm\beta,{\bm\lambda})$ is dense in the locally $\bQ_p$-analytic representation $\big(\widehat{i_{\emptyset}(\alpha,\bm\beta,{\bm\lambda})}\big)^{\bQ_p-\ana}/\bI^-_{\emptyset}(\alpha,\bm\beta,{\bm\lambda})$,\;we deduce that $\st_n^{\infty}(\alpha,\bm\beta,{\bm\lambda})$ is dense in $\st_{n}^{\ana}(\alpha,\bm\beta,{\bm\lambda})$.\;We expect that the injection $	\st_n^{\infty}(\alpha,\bm\beta,{\bm\lambda})\hookrightarrow\Pi_{\infty}^{R_{\infty}-\ana}[\fm_y]$ induces an injection  $	\st_n^{\ana}(\alpha,\bm\beta,{\bm\lambda})\hookrightarrow\Pi_{\infty}^{R_{\infty}-\ana}[\fm_y]$ of $G$-representations.\;
\end{rmk}

We are now ready to state the main theorem.\;With these preparations complete, its proof is almost entirely parallel to that of \cite[Theorem 4.10]{2019DINGSimple}.\;We include the proof for the reader's convenience.\;

\begin{thm}\label{thm: lgln-main}
(1) The following restriction map is bijective
\begin{equation}\label{equ: lgln-main}
	\begin{aligned}
		\homo_{G}\big(\Sigma(\alpha,\bm\beta,{\bm\lambda}, \sL(\rho_L)), \Pi_{\infty}^{R_{\infty}-\ana}[\fm_y]\big)\lra \homo_{G}\big(\st_n^{\infty}(\alpha,\bm\beta,{\bm\lambda}), \Pi_{\infty}^{R_{\infty}-\ana}[\fm_y]\big).
	\end{aligned}
\end{equation}

\noindent
(2) Let $0\neq\psi\in\homo(L^\times,E)$ and $1\leq i\leq k-1$. An injection $f:\st_n^\infty(\alpha,\bm\beta,{\bm\lambda})\hookrightarrow\Pi_\infty^{R_\infty-\ana}[\fm_y]$ extends to an injection $\Sigma_{ir}(\alpha,\bm\beta,{\bm\lambda},\psi)\hookrightarrow\Pi_\infty^{R_\infty-\ana}[\fm_y]$ if and only if $\psi\in\sL(\rho_L)_{ir}$.
\end{thm}

\begin{proof}
\textbf{(a)} The ``if'' direction of Theorem \ref{thm: lgln-main} (2) follows from (1).\;It remains to prove the ``only if'' direction.\;Suppose, to the contrary, that there is an injection $\Sigma_{ir}(\alpha,\bm\beta,{\bm\lambda}, \psi)\hookrightarrow \Pi_{\infty}^{R_{\infty}-\ana}[\fm_y]$ with $\psi\notin \sL(\rho_L)_{ir}$. Since $\sL(\rho_L)_{ir}$ has codimension $1$ in $\homo(L^{\times}, E)$, we have $\sL(\rho_L)_{ir}+ E\psi =\homo(L^{\times}, E)$.The resulting injection
\begin{equation*}
	\Sigma_{ir}(\alpha, \bm\beta,{\bm\lambda}, \sL(\rho_L)_{ir})\oplus_{\Sigma_{ir}(\alpha,\bm\beta,{\bm\lambda})} \Sigma_{ir}(\alpha,\bm\beta,{\bm\lambda}, \psi) \hooklongrightarrow \Pi_{\infty}^{R_{\infty}-\ana}[\fm_y]
\end{equation*}
induces $\Sigma_{ir}(\alpha, \bm\beta,{\bm\lambda}, \psi_{\infty}) \hooklongrightarrow\Sigma_{ir}(\alpha, \bm\beta,{\bm\lambda}, \sL(\rho_L)_{ir})\oplus_{\Sigma_{ir}(\alpha, \bm\beta,{\bm\lambda})} \Sigma_{ir}(\alpha, \bm\beta,{\bm\lambda}, \psi) \hooklongrightarrow \Pi_{\infty}^{R_{\infty}-\ana}[\fm_y]$
for every $0\neq \psi_{\infty}\in \homo_{\infty}(L^{\times},E)$. By \cite[Remark 5.11]{2022ext1hyq}, $\Sigma_{ir}(\alpha,\bm\beta,{\bm\lambda},\psi_{\infty})$ contains an extension of $v_{ir}^{\lalg}(\alpha,\bm\beta,{\bm\lambda})$ by $\st_n^{\infty}(\alpha,\bm\beta,{\bm\lambda})$. Applying the left-exact Jacquet--Emerton functor yields a contradiction to Lemma \ref{detlameda}. This proves (2).
\\
\noindent \textbf{(b)} Since $\soc_G\Sigma(\alpha,\bm\beta,{\bm\lambda})=\st_n^{\infty}(\alpha,\bm\beta,{\bm\lambda})$, Lemma \ref{detlameda} and Step (a) of the proof of Proposition \ref{proadjuncitonformularight} imply that the restriction map in (1) is injective.\;It remains to show that \eqref{equ: lgln-main} is surjective.\;The key input is the surjectivity of \eqref{eigentangentmap2}.\;By definition, it suffices to show that, for every $1\leq i\leq k-1$ and $\psi\in\sL(\rho_L)_{ir}$, the following restriction map is surjective:
\begin{equation}\label{equ: lgln-rest}
	\homo_{G}\big(\Sigma_{ir}(\alpha ,\bm\beta,{\bm\lambda}, \psi), \Pi_{\infty}^{R_{\infty}-\ana}[\fm_y]\big) \lra \homo_{G}\big(\st_n^{\infty}(\alpha,\bm\beta,{\bm\lambda}), \Pi_{\infty}^{R_{\infty}-\ana}[\fm_y]\big).
\end{equation}
By Proposition \ref{coreigentangentmap2}, there exists a tangent vector $t:\Spec(E[\epsilon]/\epsilon^2)\rightarrow\bersteineigenvarpik$ whose $i$-th component under \eqref{eigentangentmap2} is $\psi$ and whose other components vanish.\;Let $\Psi\in\homo(\bZ_\beta(L),E)$ be the image of $t$ under the first map in \eqref{tangentpatchvar1}.\;By Proposition \ref{basicpropertyeigen}(3), the coherent sheaf $\cM^{\infty}_{\omepik,{\bm\lambda}}$ is Cohen--Macaulay over $\mathcal{E}_{\omepik,\fp,{\bm\lambda}}^\infty(\overline{\rho})$. Since the non-critical special point $x$ is smooth, this sheaf is locally free near $x$. Let $\FI_t$ denote the kernel of the map $R_\infty[1/p]\rightarrow E[\epsilon]/\epsilon^2$ induced by $t$, and put \[\pi_{\beta}(\alpha,{\bm\lambda},\Psi):=\pi_{\beta}(\alpha,\bm{\lambda})\te\big(1+(\Psi\circ\det)\epsilon\big).\]
By construction,
	\begin{equation}
	\begin{aligned}
		(x^*\cM^{\infty}_{\omepik,{\bm\lambda}})^{\vee}\cong&\; \homo_{\bL_{\beta}(L)}\big(\pi_{\beta}(\alpha,\bm{\lambda})\te \delta_{\bP_{\beta}}, {J}_{\bP_{\beta}}\big(\Pi_{\infty}^{R_{\infty}-\ana}[\fm_y]\big)\big),\\
		\text{resp.,\;}(t^* \cM^{\infty}_{\omepik,{\bm\lambda}})^{\vee}\cong&\; \homo_{\bL_{\beta}(L)}\big(\pi_{\beta}(\alpha,{\bm\lambda},\Psi)\te \delta_{\bP_{\beta}}, J_{\bP_{\beta}}\big(\Pi_{\infty}^{R_{\infty}-\ana}[\FI_t]\big)\big)
	\end{aligned}
\end{equation}
Both are closed subrepresentations of  $J_{\bP_{\beta}}\big(\Pi_{\infty}^{R_{\infty}-\ana}[\fm_y])$ (resp., $J_{\bP_{\beta}}(\Pi_{\infty}^{R_{\infty}-\ana}[\FI_t])$).\;There are natural $\bL_{\beta}(L)$-equivariant injections
	\begin{equation}\label{equ: lgln-injf}(x^*\cM^{\infty}_{\omepik,{\bm\lambda}})^{\vee}\hooklongrightarrow (t^* \cM^{\infty}_{\omepik,{\bm\lambda}})^{\vee} \hooklongrightarrow J_{\bP_{\beta}}(\Pi_{\infty}^{R_{\infty}-\ana}[\FI_t]).\end{equation}
By Lemma \ref{detlameda} and an easy variation of the proof of
\cite[Lemma 4.11]{2019DINGSimple},\;we can show that
\begin{lem} The morphisms of $\bL_{\beta}(L)$-representations:
	\begin{equation}
		\begin{aligned}
			(x^*\cM_{\infty,\omepik,{\bm\lambda}})^{\vee} \hookrightarrow J_{\bP_{\beta}}(\Pi_{\infty}^{R_{\infty}-\ana}[\fm_y]),\quad(t^* \cM_{\infty,\omepik,{\bm\lambda}})^{\vee} \hookrightarrow
			J_{\bP_{\beta}}(\Pi_{\infty}^{R_{\infty}-\ana}[\FI_t])
		\end{aligned}
	\end{equation}
	are balanced (see \cite[Definition 0.8]{emerton2007jacquet},\;\cite[Definition 5.17]{Emerton2007summary}).
\end{lem}
\begin{proof}This lemma follows from a slight adaptation of the proof of \cite[Lemma 4.11]{2019DINGSimple}.\;We briefly indicate the necessary changes.\;In that proof, replace the notation ``$U(\fb_{L})$'' (resp., $\cC^{\sm}_c\big(\bN(L),-)$, $\cC^{\lp}_c\big(\bN(L),-)$, and $\cC^{\bQ_p-\pol}_c\big(\bN(L),-)$) with ${U}(\fp_{\beta,L})$ (resp., $\cC^{\sm}_c\big(\bN_{\beta}(L),-)$, $\cC^{\lp}_c\big(\bN_{\beta}(L),-)$, and $\cC^{\bQ_p-\pol}_c\big(\bN_{\beta}(L),-)$).\;Also, replace the short exact sequence $``0\rightarrow\chi\rightarrow\widetilde{\chi}\rightarrow\chi\rightarrow0"$ with the short exact sequence
\[ ``0 \rightarrow\pi_{\beta}(\alpha,\bm{\lambda})\te \delta_{\bP_{\beta}} \rightarrow \pi_{\beta}(\alpha,{\bm\lambda},\Psi)\te \delta_{\bP_{\beta}}\rightarrow \pi_{\beta}(\alpha,\bm{\lambda})\te \delta_{\bP_{\beta}}\rightarrow 0".\;\]Finally, replace the reference to \cite[Lemma 4.6]{2019DINGSimple} with Lemma \ref{detlameda}.\;
\end{proof}
This lemma shows that the adjunction formula in \cite{emerton2007jacquet} applies in our setting.\;For details on the locally $\bQ_p$-analytic representation $I_{\op_{\beta}(L)}^G(-)$ of $G$, see \cite{emerton2007jacquet}.\;By \cite[Theorem 0.13]{emerton2007jacquet}, there exists an integer $r$ such that the injections in (\ref{equ: lgln-injf}) induce
\begin{equation}\label{equ: lgln-keymap}
	I_{\op_{\beta}(L)}^G \big(\pi_{\beta}(\alpha,\bm{\lambda})\big)^{\oplus r} \hooklongrightarrow I_{\op_{\beta}(L)}^G \big(\pi_{\beta}(\alpha,{\bm\lambda},\Psi)\big)^{\oplus r} \longrightarrow \pi_{\infty}^{R_{\infty}-\ana}[\FI_t].
\end{equation}
By \cite[Proposition 2.8.10]{emerton2007jacquet}, $I_{\op_{\beta}(L)}^G(\pi_{\beta}(\alpha,\bm{\lambda}))\cong i_{\emptyset}(\alpha,\bm\beta,{\bm\lambda})$.\;Thus we obtain the following sequence, which need not be exact:
\begin{equation*}
	i_{\emptyset}(\alpha,\bm\beta,{\bm\lambda}) \hooklongrightarrow I_{\op_{\beta}(L)}^G(\pi_{\beta}(\alpha,{\bm\lambda},\Psi)) \twoheadlongrightarrow  i_{\emptyset}(\alpha,\bm\beta,{\bm\lambda}).
\end{equation*}
Then $v_{I}^{\infty}(\alpha,\bm\beta,{\bm\lambda})$ occurs with multiplicity $2$ as an irreducible constituent of $I_{\op_{\beta}(L)}^G(\pi_{\beta}(\alpha,{\bm\lambda},\Psi))$.\;Inside the representation $\big(\ind_{\op_{\beta}(L)}^{G} \pi_{\beta}(\alpha,{\bm\lambda},\Psi)\big)^{\bQ_p-\ana}$, define the following intersections:
\begin{equation}
	U:=I_{\op_{\beta}(L)}^G(\pi_{\beta}(\alpha,{\bm\lambda},\Psi)) \cap \bI^-_{\emptyset}(\alpha,\bm\beta,{\bm\lambda})\hookrightarrow
	W:=I_{\op_{\beta}(L)}^G(\pi_{\beta}(\alpha,{\bm\lambda},\Psi))\cap \bI_{\emptyset}(\alpha,\bm\beta,{\bm\lambda}),
\end{equation}
since $I_{\op_{\beta}(L)}^G(\pi_{\beta}(\alpha,{\bm\lambda},\Psi))  \hookrightarrow \big(\ind_{\op_{\beta}(L)}^{G} \pi_{\beta}(\alpha,{\bm\lambda},\Psi)\big)^{\bQ_p-\ana}$.\;Thus $\widetilde{\Sigma}(\alpha,\bm\beta,{\bm\lambda})':=W/U$ is a subrepresentation of $\st_n^{\ana}(\alpha,\bm\beta,{\bm\lambda})$, and $I_{\op_{\beta}(L)}^G(\pi_{\beta}(\alpha,{\bm\lambda},\Psi))/U$ is an extension of $i_{\emptyset}(\alpha,\bm\beta,{\bm\lambda})$ by it.\;As in Step (a) of Proposition \ref{proadjuncitonformularight}, Lemma \ref{detlameda} shows that no irreducible constituent of $U$ occurs in the socle of $\pi_{\infty}^{R_{\infty}-\ana}[\FI_t]$. Hence \eqref{equ: lgln-keymap} induces
\begin{equation}\label{equ: lgln-keymap2}
	\st_n^{\infty}(\alpha,\bm\beta,{\bm\lambda})^{\oplus r}\hooklongrightarrow \big(I_{\op_{\beta}(L)}^G (\pi_{\beta}(\alpha,{\bm\lambda},\Psi))/U\big)^{\oplus r} \lra \Pi_{\infty}^{R_{\infty}-\ana}[\FI_t].
\end{equation}
Note that $\soc_{G} \Sigma_{ir}(\alpha,\bm\beta,{\bm\lambda})'\cong \st_n^{\infty}(\alpha,\bm\beta,{\bm\lambda})$.\;This composition is injective.\;Since $(x^* \cM_{\infty})^{\vee}$ has image in $J_{\bP_{\beta}}(\pi_{\infty}^{R_{\infty}-\ana}[\fm_y])$ via (\ref{equ: lgln-injf}),\;we see that it factors through an injection $\Pi_{\infty}^{R_{\infty}-\ana}[\fm_y]\hookrightarrow \Pi_{\infty}^{R_{\infty}-\ana}[\FI_t]$ by Corollary \ref{cor: lgln-bij3}.\;Denote by $\Sigma_{ir}(\alpha,\bm\beta,{\bm\lambda})':=\widetilde{\Sigma}(\alpha,\bm\beta,{\bm\lambda})' \cap \Sigma_{ir}(\alpha,\bm\beta,{\bm\lambda})$.\;Thus by Corollary \ref{cor: lgln-bij3},\;the  composition
\begin{equation}\label{equ: compW}
	(\Sigma_{ir}(\alpha,\bm\beta,{\bm\lambda})')^{\oplus r} \hooklongrightarrow \big(I_{\op_{\beta}(L)}^G (\pi_{\beta}(\alpha,{\bm\lambda},\Psi))/U\big)^{\oplus r} \lra \pi_{\infty}^{R_{\infty}-\ana}[\FI_t]
\end{equation}
also has image in $\pi_{\infty}^{R_{\infty}-\ana}[\fm_y]$.\;By Lemma \ref{detlameda},\;it is also injective.\;By Proposition \ref{proadjuncitonformularight}, we see that (\ref{equ: compW}) extends uniquely to   an injection
\begin{equation}\label{equ: compnoL}
	\Sigma_{ir}(\alpha,\bm\beta,{\bm\lambda})^{\oplus r} \lra \pi_{\infty}^{R_{\infty}-\ana}[\fm_y] \hooklongrightarrow \pi_{\infty}^{R_{\infty}-\ana}[\FI_t].
	\end{equation}
Combining (\ref{equ: compW}) with  (\ref{equ: compnoL}),\;we put \[V^+:=\big(I_{\op_{\beta}(L)}^G (\pi_{\beta}(\alpha,{\bm\lambda},\Psi))/U\big)\oplus_{\Sigma_{ir}(\alpha,\bm\beta,{\bm\lambda})'} \Sigma_{ir}(\alpha,\bm\beta,{\bm\lambda}).\]
Therefore,\;(\ref{equ: compW}) and  (\ref{equ: compnoL}) give
\begin{equation}\label{equ: lgL1}
	\begin{aligned}
		\Sigma_{ir}(\alpha,\bm\beta,{\bm\lambda})^{\oplus r}\hooklongrightarrow \big(V^+\big)^{\oplus r}\lra \pi_{\infty}^{R_{\infty}-\ana}[\FI_t].
	\end{aligned}
\end{equation}
The image of this map lies in $\pi_{\infty}^{R_{\infty}-\ana}[\fm_y]$.\;We claim that
$\Sigma_{ir}(\alpha,\bm\beta,{\bm\lambda}, \psi)$ is a subrepresentation of $V^+$.\;Indeed, let $V$ be the pullback of $I_{\op_{\beta}(L)}^G (\pi_{\beta}(\alpha,{\bm\lambda},\Psi))/U$ along the injection $i_{ir}(\alpha,\bm\beta,{\bm\lambda}) \hookrightarrow i_{\emptyset}(\alpha,\bm\beta,{\bm\lambda})$. Then $V$ fits into a commutative diagram:
\begin{equation}\label{commrevi}
	\xymatrix{0 \ar[r] & \widetilde{\Sigma}(\alpha,\bm\beta,{\bm\lambda})' \ar@{_(->}[d] \ar[r] & \hspace{20pt}V\hspace{20pt} \ar@{_(->}[d] \ar[r] & i_{ir}(\alpha,\bm\beta,{\bm\lambda}) \ar@{=}[d] \ar[r] & 0  \\
		0 \ar[r] &\st_n^{\ana}(\alpha,\bm\beta,{\bm\lambda}) \ar[r] & \sE_{ir}^\emptyset(\alpha,\bm\beta,{\bm\lambda},\Psi)^0/\bI^-_{\emptyset}(\alpha,\bm\beta,{\bm\lambda}) \ar[r] & i_{ir}(\alpha,\bm\beta,{\bm\lambda})\ar[r]&0.}
\end{equation}
The representation $\sE_{ir}^\emptyset(\alpha,\bm\beta,{\bm\lambda},\Psi)^0$ was defined in the discussion following (\ref{explicitexactseq}). The pullback of the bottom exact sequence along the injection $w_{ir}(\alpha,\bm\beta,{\bm\lambda})\hookrightarrow i_{ir}(\alpha,\bm\beta,{\bm\lambda})$ splits (see the argument around (\ref{dfnforsmoothanaStein}) or \cite[Theorem 5.9, Remark 5.11]{2022ext1hyq}).\;Since $\homo_G(w_{ir}(\alpha,\bm\beta,{\bm\lambda}), \st_n^{\ana}(\alpha,\bm\beta,{\bm\lambda})/\widetilde{\Sigma}(\alpha,\bm\beta,{\bm\lambda})')=0$, it follows that
\begin{equation*}\ext^1_G(w_{ir}(\alpha,\bm\beta,{\bm\lambda}), \widetilde{\Sigma}(\alpha,\bm\beta,{\bm\lambda})')\hooklongrightarrow \ext^1_G(w_{ir}(\alpha,\bm\beta,{\bm\lambda}),\st_n^{\ana}(\alpha,\bm\beta,{\bm\lambda}))
\end{equation*}
Thus the pullback of the top exact sequence in (\ref{commrevi}) along $w_{ir}(\alpha,\bm\beta,{\bm\lambda})\hookrightarrow i_{ir}(\alpha,\bm\beta,{\bm\lambda})$ also splits.\;Consequently, the representation
$I_{\op_{\beta}(L)}^G (\pi_{\beta}(\alpha,{\bm\lambda},\Psi))/U$ contains a subrepresentation $\widetilde{\Sigma}_{ir}(\alpha,\bm\beta,{\bm\lambda}, \psi)'$, which is an extension of $v_{ir}^{\lalg}(\alpha,\bm\beta,{\bm\lambda})$ by $\widetilde{\Sigma}(\alpha,\bm\beta,{\bm\lambda})'$. We have an isomorphism
\begin{equation*}
	\ext^1_G(v_{ir}^{\lalg}(\alpha,\bm\beta,{\bm\lambda}), \Sigma_{ir}(\alpha,\bm\beta,{\bm\lambda})') \xrightarrow{\sim}  \ext^1_G(v_{ir}^{\lalg}(\alpha,\bm\beta,{\bm\lambda}), \widetilde{\Sigma}(\alpha,\bm\beta,{\bm\lambda})'),
\end{equation*}
by an argument similar to those in \cite[Proposition 5.13]{2022ext1hyq} and \cite[Proposition 5.17]{2022ext1hyq}.\;Thus
$\widetilde{\Sigma}_{ir}(\alpha,\bm\beta,{\bm\lambda}, \psi)'$ comes from some $\Sigma_{ir}(\alpha,\bm\beta,{\bm\lambda}, \psi)'$,\;which is
an extension of $v_{ir}^{\lalg}(\alpha,\bm\beta,{\bm\lambda})$ by the representation $\Sigma_{ir}(\alpha,\bm\beta,{\bm\lambda})'$.\;The pushforward of $\Sigma_{ir}(\alpha,\bm\beta,{\bm\lambda}, \psi)'$ along the injection $\Sigma_{ir}(\alpha,\bm\beta,{\bm\lambda})'\hookrightarrow \Sigma_{ir}(\alpha,\bm\beta,{\bm\lambda})$ is isomorphic to $\Sigma_{ir}(\alpha,\bm\beta,{\bm\lambda}, \psi)$. This proves the claim.\;

The composition in (\ref{equ: lgL1}) therefore induces $\Sigma_{ir}(\alpha,\bm\beta,{\bm\lambda})^{\oplus r}\hooklongrightarrow \Sigma_{ir}(\alpha,\bm\beta,{\bm\lambda}, \psi)^{\oplus r} \lra \Pi_{\infty}^{R_{\infty}-\ana}[\FI_t]$.\;By an argument as in \cite[Page 8040]{2015Ding}, the image of the second morphism also lies in $ \pi_{\infty}^{R_{\infty}-\ana}[\fm_y] $.\;The surjectivity of (\ref{equ: lgln-rest}) follows, completing the proof of Theorem \ref{thm: lgln-main}.
\end{proof}

\section{Appendix}
\subsection{Proof of the parabolic Colmez--Greenberg--Stevens formula}
This subsection proves Theorem \ref{defor1thdefor}.\;We follow the strategy of \cite[Theorem 2.7]{HigherLinvariantsGL3Qp}, although the computations and d\'{e}vissage required here are more involved.
\begin{proof}[Proof of Theorem \ref{defor1thdefor}]Replacing $\widetilde{\Dpik}$ and $\widetilde{\Dpik}_1^{k-1}$ by \[\widetilde{\Dpik}\otimes_{\cR_{E[\epsilon]/\epsilon^2,L}}\cR_{E[\epsilon]/\epsilon^2,L}(\bm{\delta}_{E[\epsilon]/\epsilon^2,k}^{-1}) \text{\;and\;} \widetilde{\Dpik}_1^{k-1}\otimes_{\cR_{E[\epsilon]/\epsilon^2,L}}\cR_{E[\epsilon]/\epsilon^2,L}(\bm{\delta}_{E[\epsilon]/\epsilon^2,k}^{-1})\] respectively,\;we may assume that ${\delta}_{E[\epsilon]/\epsilon^2,k}=z^{\bh}_{\beta,k}$.\;By twisting by $\cR_{E,L}((z^{\bh}_{\beta,k})^{-1})$,\;we may assume that $z^{\bh}_{\beta,k}=1$ and $z^{\bh}_{\beta,k-1}=z^{\underline{h}}$ for some $\underline{h}\in \BZ_{>0}^{|\Sigma_L|}$.\;We put $ \widetilde\Delta_x:={\Delta_x}\otimes_{\cR_{E,L}}\cR_{E[\epsilon]/\epsilon^2,L}$.\;Therefore,\;we have  injections $\widetilde\Dpik_{k}\hookrightarrow \widetilde\Delta_x$ and $\Dpik_{k}\hookrightarrow {\Delta_x}$.\;By definition,\;we have two commutative diagrams:
\begin{align}\label{extksequence}
	\xymatrix{0 \ar[r] & \Dpik_{k} \ar[d]^{}  \ar[r] & \widetilde\Dpik_{k} \ar[d]^{} \ar[r] & \Dpik_{k} \ar[d]^{} \ar[r] & 0  \\
		0 \ar[r] & \Delta_x  \ar[r] & \widetilde\Delta_x \ar[r] & \Delta_x\ar[r] & 0
	},	\quad \xymatrix{0 \ar[r] & {\Dpik}_1^{k-1} \ar[d]^{\pr}  \ar[r] & \widetilde{\Dpik}_1^{k-1}\ar[r] \ar[d]^{\pr}  &{\Dpik}_1^{k-1} \ar[r] \ar[d]^{\pr}  & 0\\
	0\ar[r] &\Dpik_{k-1}\ar[r] &\widetilde\Dpik_{k-1}  \ar[r] &\Dpik_{k-1}\ar[r] & 0  .}
\end{align}
Combining the first rows of the two diagrams in (\ref{extksequence}) yields the following commutative diagram ($-\otimes-:=-\tee-$):
\[\xymatrix{
	& 0          \ar[d]^{}        & 0    \ar[d]^{}    &  0           \ar[d]^{}                              &                               & \\
	0 \ar[r] & {\Dpik}_1^{k-1}\otimes\Dpik_{k}^\vee \ar[d]^{}  \ar[r] &  {\Dpik}_1^{k-1}\otimes \widetilde{\Dpik}_{k}^\vee \ar[d]^{}  \ar[r]   &  {\Dpik}_1^{k-1}\otimes \Dpik_{k}^\vee    \ar[d]^{}  \ar[r]  & 0\\
	0 \ar[r] & \widetilde{\Dpik}_1^{k-1}\otimes \Dpik_{k}^\vee \ar[d]^{}  \ar[r] &  \widetilde{\Dpik}_1^{k-1}\otimes \widetilde{\Dpik}_{k}^\vee  \ar[d]^{}  \ar[r]  &  \widetilde{\Dpik}_1^{k-1}\otimes \Dpik_{k}^\vee  \ar[d]^{}  \ar[r]     & 0\\
	0 \ar[r] & {\Dpik}_1^{k-1}\otimes \Dpik_{k}^\vee \ar[d]^{}  \ar[r] &  {\Dpik}_1^{k-1}\otimes \widetilde{\Dpik}_{k}^\vee \ar[d]^{}  \ar[r]   &  {\Dpik}_1^{k-1}\otimes \Dpik_{k}^\vee  \ar[d]^{}  \ar[r]    & 0\\
	& 0                 & 0       &  0                                        &                       &
	,}\]
which induces a surjection $\widetilde{\Dpik}_1^{k-1}\otimes \widetilde{\Dpik}_{k}^\vee \rightarrow {\Dpik}_1^{k-1}\otimes \Dpik_{k}^\vee$.\;Let $\Dpik_{1,2}$ denote its kernel.\;We define the morphisms
\begin{equation}
	\begin{aligned}
		u:\;&{\Dpik}_1^{k-1}\otimes \Dpik_{k}^\vee \rightarrow ({\Dpik}_1^{k-1}\otimes \widetilde{\Dpik}_{k}^\vee)\oplus(\widetilde{\Dpik}_1^{k-1}\otimes \Dpik_{k}^\vee)\quad a\mapsto (a',b'),
	\end{aligned}
\end{equation}
where $a'$ (resp.,\;$b'$) is the image of $a$ in $({\Dpik}_1^{k-1}\otimes \widetilde{\Dpik}_{k}^\vee)$ (resp.,\;in $\widetilde{\Dpik}_1^{k-1}\otimes \Dpik_{k}^\vee$),\;and
\begin{equation}
	\begin{aligned}
		v:\;&({\Dpik}_1^{k-1}\otimes \widetilde{\Dpik}_{k}^\vee)\oplus(\widetilde{\Dpik}_1^{k-1}\otimes \Dpik_{k}^\vee) \rightarrow \Dpik_{1,2}\quad(a',b')\mapsto a''-b'',
	\end{aligned}
\end{equation}
where $a''$ (resp.,\;$b''$) is the image of $a'$ (resp.,\;$b'$) in $\widetilde{\Dpik}_1^{k-1}\otimes \widetilde{\Dpik}_{k}^\vee$.\;It is straightforward to verify that these definitions yield the following exact sequences:
\begin{equation}\label{extksequence2}
	\begin{aligned}
		&0\rightarrow \Dpik_{1,2}\rightarrow \widetilde{\Dpik}_1^{k-1}\otimes \widetilde{\Dpik}_{k}^\vee \rightarrow {\Dpik}_1^{k-1}\otimes \Dpik_{k}^\vee\rightarrow 0,\\
		&0\rightarrow {\Dpik}_1^{k-1}\otimes \Dpik_{k}^\vee \xrightarrow{u}  ({\Dpik}_1^{k-1}\otimes \widetilde{\Dpik}_{k}^\vee)\oplus(\widetilde{\Dpik}_1^{k-1}\otimes \Dpik_{k}^\vee) \xrightarrow{v} \Dpik_{1,2}\rightarrow 0,\;
	\end{aligned}
\end{equation}
are exact.\;Similar diagrams and short exact sequences are obtained using the first row of the first diagram in (\ref{extksequence}) and the second row of the second diagram; namely,
\begin{equation}\label{extksequence34}
	\begin{aligned}
		&0\rightarrow \Dpik_{1,2}'\rightarrow \widetilde{\Dpik}_{k-1}\otimes \widetilde{\Dpik}_{k}^\vee \rightarrow {\Dpik}_{k-1}\otimes \Dpik_{k}^\vee\rightarrow 0,\\
		&0\rightarrow {\Dpik}_{k-1}\otimes \Dpik_{k}^\vee \rightarrow  ({\Dpik}_{k-1}\otimes \widetilde{\Dpik}_{k}^\vee)\oplus(\widetilde{\Dpik}_{k-1}\otimes \Dpik_{k}^\vee) \rightarrow \Dpik_{1,2}'\rightarrow 0.
	\end{aligned}
\end{equation}
The second row of the first diagram in (\ref{extksequence}), together with the second row of the second diagram, gives
\begin{equation}\label{extksequence345}
	\begin{aligned}
		&0\rightarrow \Dpik_{1,2}''\rightarrow \widetilde{\Dpik}_{k-1}\otimes \widetilde\Delta_x^\vee \rightarrow {\Dpik}_{k-1}\otimes \Delta_x^\vee\rightarrow 0,\\
		&0\rightarrow {\Dpik}_{k-1}\otimes \Delta_x^\vee \rightarrow  ({\Dpik}_{k-1}\otimes \widetilde\Delta_x^\vee)\oplus(\widetilde{\Dpik}_{k-1}\otimes \Delta_x^\vee) \rightarrow \Dpik_{1,2}''\rightarrow 0.\;
	\end{aligned}
\end{equation}

Taking cohomology of the first rows of (\ref{extksequence2}) and (\ref{extksequence34}) yields the following commutative diagram. For simplicity, we omit the subscript $(\varphi,\Gamma)$ from $\hH^{\bullet}_{(\varphi,\Gamma)}$.
\begin{equation}\label{cgs1}
	\xymatrix{ \ar[r] & \hH^1(\Dpik_{1,2}) \ar[d]^{u_1}  \ar[r] & \ext^1(\widetilde{\Dpik}_{k},\widetilde{\Dpik}_1^{k-1}) \ar[d]^{v_1}  \ar[r]^{p}  & \ext^1(\Dpik_{k},\Dpik_{1}^{k-1}) \ar[d]^{w_1} \ar[r]^{c} & \hH^2(\Dpik_{1,2}) \ar[d]^{u_2}\\
		\ar[r] & \hH^1(\Dpik_{1,2}')   \ar[r]
		& \ext^1(\widetilde{\Dpik}_{k},\widetilde{\Dpik}_{k-1})  \ar[r]^{p'} & \ext^1(\Dpik_{k},\Dpik_{k-1})  \ar[r]^{c'}  & \hH^2(\Dpik_{1,2}')
		.}
\end{equation}
By Lemma \ref{phigammacohononcrtialspecial} and a straightforward d\'{e}vissage argument, we have $\ext^2(D_k,\Dpik_{1}^{k-2})=0$. It follows that $w_1$ is surjective and that $u_2':\ext^2(\Dpik_{k},\Dpik_{1}^{k-1})\rightarrow \ext^2(\Dpik_{k},\Dpik_{k-1})$ is an isomorphism.\;We now show that $u_1$ is surjective and $u_2$ is an isomorphism.\;Taking cohomology of the second rows of (\ref{extksequence2}) and (\ref{extksequence34}) yields the following commutative diagrams:
\begin{equation}
	\begin{aligned}
		& \xymatrix@C=2.2ex{  \ext^1(\Dpik_{k},\Dpik_1^{k-1}) \ar[d]^{u_1'}  \ar[r] & \ext^1(\widetilde{\Dpik}_{k},{\Dpik}_1^{k-1})\oplus \ext^1(\Dpik_{k},\widetilde\Dpik_{1}^{k-1}) \ar[d]^{u_1''}  \ar[r]   & \hH^1(\Dpik_{1,2}) \ar[d]^{u_1}\\
			 \ext^1(\Dpik_{k},\Dpik_{k-1})  \ar[r] & \ext^1(\widetilde{\Dpik}_{k},{\Dpik}_{k-1})\oplus \ext^1(\Dpik_{k},\widetilde\Dpik_{k-1})   \ar[r]   & \hH^1(\Dpik_{1,2}')
	}\end{aligned}
\end{equation}
\begin{equation}
	\begin{aligned}
		& \xymatrix@C=2.2ex{ \ar[r] & \ext^2(\Dpik_{k},\Dpik_1^{k-1}) \ar[d]^{u_2'}  \ar[r] & \ext^2(\widetilde{\Dpik}_{k},{\Dpik}_1^{k-1})\oplus \ext^2(\Dpik_{k},\widetilde\Dpik_{1}^{k-1}) \ar[d]^{u_2''}  \ar[r] & \hH^2(\Dpik_{1,2}) \ar[d]^{u_2}  \ar[r] & 0\\
			\ar[r] & \ext^2(\Dpik_{k-1},\Dpik_{k-1})   \ar[r] & \ext^2(\widetilde{\Dpik}_{k},{\Dpik}_{k-1})\oplus \ext^2(\Dpik_{k},\widetilde\Dpik_{k-1})   \ar[r]   & \hH^2(\Dpik_{1,2}') \ar[r] & 0
		.}
	\end{aligned}
\end{equation}
By Lemma \ref{phigammacohononcrtialspecial} and d\'{e}vissage, $u'_1$ is surjective and $u'_2$ is an isomorphism. The four lemma therefore implies that
\[\ext^2(\widetilde{\Dpik}_{k},{\Dpik}_1^{k-1}) \rightarrow \ext^2(\widetilde{\Dpik}_{k},{\Dpik}_{k-1}),\qquad \ext^2(\Dpik_k,\widetilde\Dpik_1^{k-1}) \rightarrow \ext^2(\Dpik_k,\widetilde\Dpik_{k-1})\]
are isomorphisms. Hence $u_2''$, and consequently $u_2$, are isomorphisms. Similarly, the corresponding maps in degree one are surjective; therefore, $u_1$ and then $v_1$ are surjective.

Note that $\widetilde{\Dpik}$ admits an $\omepik$-filtration with parameter $((\underline{x}_{\beta,i})_{1\leq i\leq k},(\bm{\delta}_{E[\epsilon]/\epsilon^2,i})_{1\leq i\leq k})$ for which the graded piece $\gr^{\widetilde{\cF}}_{i}\widetilde{\Dpik}=\widetilde\Dpik_i$ occurs precisely when $[\Dpik]\in \ext^1(\Dpik_{k},{\Dpik}_1^{k-1})$ lies in $\mathrm{Im}(p)$ (see (\ref{cgs1})); equivalently, when $c([\Dpik])=0$. The preceding discussion shows that $\widetilde{\Dpik}$ admits the desired $\omepik$-filtration precisely when $c'([\Dpik^k_{k-1}])=0$, where
$[\Dpik^k_{k-1}]=w_1([\Dpik])\in  \ext^1(\Dpik_k,\Dpik_{k-1})$.\;

We now translate the preceding discussion to the setting of $p$-adic differential equations. The first rows of (\ref{extksequence34}) and (\ref{extksequence345}) yield the following commutative diagram:
\begin{equation}\label{D12D12}
	\xymatrix@C=3ex{ \ar[r] & \hH^1(\Dpik_{1,2}'') \ar[d]^{x_1}  \ar[r] & \ext^1(\widetilde\Delta_x,\widetilde{\Dpik}_{k-1}) \ar[d]^{v_1'}  \ar[r]^{\;\;\;\;\;p''}  & \ext^1({\Delta_x},\Dpik_{k-1}) \ar[d]^{w_1'}_{\simeq} \ar[r]^{\;\;\;\;c''} & \hH^2(\Dpik_{1,2}'') \ar[d]^{x_2}\\
		\ar[r] & \hH^1(\Dpik_{1,2}')   \ar[r]
		& \ext^1(\widetilde{\Dpik}_{k},\widetilde{\Dpik}_{k-1})  \ar[r]^{p'}  & \ext^1(\Dpik_{k},\Dpik_{k-1})  \ar[r]^{\;\;\;\;c'}  & \hH^2(\Dpik_{1,2}')
		.}
\end{equation}
By Lemma \ref{pullbackforiso}, $w'_1$ is an isomorphism. Replacing $\Dpik_{k}^\vee$ (resp.\;$\widetilde{\Dpik}_{k}^\vee$) in the second diagrams of (\ref{extksequence}) and (\ref{extksequence2}) by $\Delta_x^\vee$ (resp.\;$\widetilde\Delta_x^\vee$) yields the following diagram:
\[\xymatrix{0 \ar[r] & \Delta_x^\vee\otimes \Dpik_{k-1}\ar@{_(->}[d]  \ar[r] & \Delta_x^\vee\otimes \widetilde{\Dpik}_{k-1}\ar[r] \ar@{_(->}[d]  & \Delta_x^\vee\otimes \Dpik_{k-1}\ar[r] \ar@{_(->}[d]  & 0\\
	0 \ar[r] & \Dpik_{1,2}'' \ar[r] & \widetilde\Delta_x^\vee\otimes \widetilde{\Dpik}_{k-1}  \ar[r] & \Delta_x^\vee\otimes \Dpik_{k-1} \ar[r] & 0  .}\]
Combining the commutative diagram (\ref{D12D12}) with the cohomology sequence of the diagram above yields
\begin{equation}
	\begin{aligned}
		\xymatrix{
			  \ext^1(\Delta_x,\Dpik_{k-1})  \ar[r] \ar[d]^{x_1'}
			& \ext^1(\Delta_x,\widetilde{\Dpik}_{k-1})   \ar[d]^{v_1''}_{\simeq} \ar[r]^{p'''}  & \ext^1({\Delta_x},\Dpik_{k-1})  \ar[r]^{c'''} \ar@{=}[d]^{w_1''}_{\simeq}  &\ext^2(\Delta_x,\Dpik_{k-1}) \ar[d]^{x_2'}\\
			\hH^1(\Dpik_{1,2}'')   \ar[r] \ar[d]^{x_1}
			& \ext^1(\widetilde\Delta_x,\widetilde{\Dpik}_{k-1})   \ar[d]^{v_1'} \ar[r]^{p''}  & \ext^1({\Delta_x},\Dpik_{k-1})  \ar[r]^{c''} \ar[d]^{w_1'}_{\simeq} &\hH^2(\Dpik_{1,2}'') \ar[d]^{x_2}\\
			 \hH^1(\Dpik_{1,2}')   \ar[r]
			& \ext^1(\widetilde{\Dpik}_{k},\widetilde{\Dpik}_{k-1}) \ar[r]^{p'}  & \ext^1(\Dpik_{k},\Dpik_{k-1})  \ar[r]^{c'}  & \hH^2(\Dpik_{1,2}').}
	\end{aligned}
\end{equation}
\text{\textbf{Claim.}} The composite map $x_3=x_2\circ x_2':\ext^2({\Delta_x},\Dpik_{k-1})\rightarrow\ext^2(\Dpik_{1,2}') $ is injective.\;\\
\text{\textbf{Proof of the claim.}}
Using the second rows of (\ref{extksequence34}) and (\ref{extksequence345}) and chasing the resulting diagram, we obtain the following commutative diagram:
\begin{equation}
	\xymatrix{  \ext^2({\Delta_x},\Dpik_{k-1}) \ar@{=}[d]  \ar[r] & \ext^2({\Delta_x},\widetilde{\Dpik}_{k-1}) \ar[d]\ar[r] &\ext^2({\Delta_x},\Dpik_{k-1}) \ar[d]^{x_2} \ar[r] &0 \\
		 \ext^2({\Delta_x},\Dpik_{k-1}) \ar@{=}[d]^{\simeq}_{(\ast\ast)} \ar[r]
		& \ext^2({\Delta_x},\widetilde{\Dpik}_{k-1}) \oplus\ext^2(\widetilde\Delta_x,\Dpik_{k-1}) \ar[d]  	\ar[r]  & \hH^2(\Dpik_{1,2}'') \ar[d]^{x_2'} \ar[r] &0\\
		{\ext}^2({\Dpik}_{k},\Dpik_{k-1}) \ar[r]			& {\ext}^2({\Dpik}_{k},\widetilde{\Dpik}_{k-1})\oplus{\ext}^2(\widetilde{\mathbf{D}}_{k},\Dpik_{k-1})\ar[r]  & \hH^2(\Dpik_{1,2}')\ar[r] &0 
		.}	
\end{equation}

The isomorphism in $(\ast\ast)$ follows from Lemma \ref{pullbackforiso}. By that lemma, a straightforward d\'{e}vissage argument, and the $4$-Lemma, the following morphisms are isomorphisms:
\begin{equation}
	\begin{aligned}
		&\ext^2({\Delta_x},\widetilde{\Dpik}_{k-1}) \rightarrow \ext^2(\Dpik_{k},\widetilde{\Dpik}_{k-1}),\;\quad\ext^2(\widetilde\Delta_x,\Dpik_{k-1})  \rightarrow \ext^2(\Dpik_{k},\widetilde{\Dpik}_{k-1})
	\end{aligned}
\end{equation}
This implies, by a straightforward diagram chase, that $x_3$ is injective. This proves the claim.

Therefore, $[\Dpik^k_{k-1}]\in \ext^1(\Dpik_k,\Dpik_{k-1})$ lies in $\mathrm{Im}(p')$ if and only if $c'''((w_1'\circ w_1'')^{-1}([\Dpik^k_{k-1}]))=0$. Recall that
\[\mathbf{I}'_{k-1}([\widetilde{\Dpik}_{k-1}])\in \ext^1(\Dpik_{k-1},\Delta_x(\unr(q_L^{-1})z^{\bh}_{\beta,k-1}))\] belongs to the image of $\hH^1(\cR_{E,L})$ via the injection \[ {\mathbf{I}}''_{k-1}\circ  \widetilde{\mathbf{I}}'_{k-1}: \hH^1(\cR_{E,L}) \rightarrow  \ext^1(\Dpik_{k-1},\Delta_x(\unr(q_L^{-1})z^{\bh}_{\beta,k-1})).\;\]
By the naturality of cup products, we obtain
\[\xymatrix@C=3.2ex{
	\ext^0(\Delta_x,\Delta_x) \ar[d]_{c_0}  & \times \hspace{10pt}\ext^1(\Delta_x,\Dpik_{k-1}) \ar@{=}[d]  \ar[r]^{\cup} & \hspace{10pt}\ext^1(\Delta_x,\Dpik_{k-1}) \ar[d]_{c'''}\\
	\ext^1(\Delta_x,\Delta_x)   & \times \hspace{10pt} \ext^1(\Delta_x,\Dpik_{k-1})  \ar[r]^{\cup}_{\hspace{5pt}\text{perfect}} & \hspace{10pt}\ext^2(\Delta_x,\Dpik_{k-1}) .}\]
This shows that the connecting map $c'''$ is given by the cup product $\langle [\widetilde{\Dpik}_{k-1}],-\rangle$; that is,
\[\langle-,-\rangle:\ext^1(\Delta_x,\Dpik_{k-1})\times \ext^1({\Delta_x},{\Delta_x})\xrightarrow{\cup}\ext^2(\Delta_x,\Dpik_{k-1})\cong E.\;\]
Thus, $c'''((w_1'\circ w_1'')^{-1}([\Dpik^k_{k-1}]))=0$ if and only if $\langle(w_1'\circ w_1'')^{-1}([\Dpik^k_{k-1}]),\mathbf{I}'_{i}([\widetilde{\Dpik}_{k-1}])\rangle=0$. The conclusion now follows from the definition of parabolic Fontaine--Mazur simple $\sL$-invariants.
\end{proof}

\subsection{A picture for the exceptional-zero formula for $p$-adic adjoint $L$-functions}

Retain the notation of \cite[Section 2.1]{Benois}. Set $\gal_{\bQ}:=\gal(\overline{\bQ}/\bQ)$, and let $V:\gal_{\bQ}\rightarrow \GLN_n(E)$ be a pseudo-geometric representation of $\gal_{\bQ}$. Assume that the adjoint Bloch--Kato Selmer groups $\hH^1_f(\mathrm{Ad}(V))$ and $\hH^1_f(\mathrm{Ad}(V)^\ast(1))$ vanish. By Poitou--Tate duality, their dimensions are equal.

For simplicity, consider the representative case $n=4$ with partition $(2,2)$, which may be viewed as a generalization of \cite{hida2004greenberg}. Although $\mathrm{Ad}(V)$ is only potentially semistable at $p$, our construction of parabolic Fontaine--Mazur simple $\sL$-invariants leads us to study $D_{\mathrm{st}}(\mathrm{Ad}(V))$.

Let $\Dpik$ be the $(\varphi,\Gamma)$-module over $\cR_E:=\cR_{E,\bQ_p}$ of rank $4$ associated to $V$.\;Then $\Dpik$ is an extension of $D_2$ by $D_1$ such that $D_1\hookrightarrow \Delta(z^{\bh_2})$ and $D_2\hookrightarrow \Delta(\unr(p)z^{\bh_4})$.\;Therefore,\;
\begin{equation}
	\mathrm{Ad}(\Dpik)=\Big[D_2^\vee\otimes_{\cR_E}D_1-\big(\mathrm{Ad}(D_1)\oplus \cR_E\oplus \mathrm{Ad}(D_2)\big)-D_1^\vee\otimes_{\cR_E}D_2\Big],
\end{equation}
where $\cR_E$ denotes the line $
\cR_E(\mathrm{id}_{D_1}+\mathrm{id}_{D_2})$ in $\EndO(D_1)\otimes \EndO(D_2)\cong \cR_E\otimes \cR_E$. Thus, $\mathrm{Ad}(\Dpik)$ admits the following three-step filtration:
\begin{equation}
	0\subset F^+\mathrm{Ad}(\Dpik)\subset F^-\mathrm{Ad}(\Dpik)\subset \mathrm{Ad}(\Dpik),
\end{equation}
where $F^+\mathrm{Ad}(\Dpik)=\{f\in \mathrm{Ad}(\Dpik):f(D_1)=0\}\cong D_2^\vee\otimes_{\cR_E}D_1$ and $F^-\mathrm{Ad}(\Dpik)=\{f\in \mathrm{Ad}(\Dpik):f(D_1)\subset D_1\}$. Note that $F^-\mathrm{Ad}(\Dpik)/F^+\mathrm{Ad}(\Dpik)\cong \mathrm{Ad}(D_1)\oplus\cR_E\oplus\mathrm{Ad}(D_2)$.

Let $\df_{k}$ be the following absolutely indecomposable Deligne--Fontaine module over $E$:
\begin{equation}\label{dfnindecompDFmod}
	(\varphi,N,\df_{k})=\bigoplus_{i=1}^k\;(p^{i-1},0,Ee_i)
\end{equation}
such that the monodromy operator $N$ sends $(p^{i},0,Ee_{i+1})$ to $(p^{i-1},0,Ee_i)$ for $i\geq 1$ and sends $(1,0,Ee_1)$ to zero. Then $\df_{\Dpik}\cong \df(\Delta)\otimes_E \df_{k}$. It follows that
\[\df(\mathrm{Ad}(\Dpik))=\df_{3}\oplus \mathrm{Ad}(\df(\Delta))\oplus (\mathrm{Ad}(\df(\Delta))\otimes_E \df_{3}). \]
In particular,\;$D_{\mathrm{st}}(\mathrm{Ad}(\Dpik))\cong \df_{3}$.\;The proposed regular subspace of $D_{\mathrm{st}}(\mathrm{Ad}(\Dpik))$ is $D:=(1,0,Ee_1)$.\;This leads to the situation considered in \cite[Section 2.3.8(c)]{Benois} and motivates the following conjecture.

\begin{conjecture} [Trivial zero conjecture] $L_p^{\ana}(\mathrm{Ad}(V),D,s)$ has a zero of order $1$ at $s=0$ and 
	\begin{equation}
		\lim_{s\rightarrow 0}\frac{L_p^{\ana}(\mathrm{Ad}(V),D,s)}{s}=\cL(\mathrm{Ad}(V),D)\frac{L(\mathrm{Ad}(V),0)}{\Omega_{\infty}(\mathrm{Ad}(V))}.\;
	\end{equation}
\end{conjecture}
\begin{rmk}
The construction of $L_p^{\ana}(\mathrm{Ad}(V),D,s)$ remains to be investigated. A natural starting point is the case in which $V$ arises from an irreducible cuspidal automorphic representation of $\GL_{2n}$ of symplectic type (equivalently, one with a nonzero Shalika integral). This question has been studied extensively in the $\GLN_2$ case.
\end{rmk}

\begin{rmk}
$\cL(\mathrm{Ad}(V),D)$ should be determined explicitly from our $\sL(\Dpik)$. Note that
\begin{equation}
	\EndO(\Dpik)=\Big[D_2^\vee\otimes_{\cR_E}D_1-\big(\EndO(D_1)\otimes \EndO(D_2)\big)-D_1^\vee\otimes_{\cR_E}D_2\Big],
\end{equation}
All extension classes between its graded pieces are determined by $[\Dpik]\in\ext^1(D_2,D_1)$.\;The following diagram commutes:
\begin{equation}
	\xymatrix@C=2.5ex{
		\ext^1_{(\varphi,\Gamma)}(D_2,D_1) \ar@{_(->}[d]                       & \times & \ext^1_{(\varphi,\Gamma)}(D_1,D_1) \ar@{_(->}[d]     \ar[r]^{\cup}_{c_1}
		& \ext^2_{(\varphi,\Gamma)}(D_2,D_1)\cong E \ar[d]_{\backsimeq} \\
		\ext^1_{(\varphi,\Gamma)}(D_1^\vee\otimes_{\cR_E}D_2,\EndO(D_1))           & \times & \ext^1_{(\varphi,\Gamma)}(\EndO(D_1),\EndO(D_1))\hspace{5pt}         \ar[r]^{\cup}_{c_2}
		& \ext^2_{(\varphi,\Gamma)}(D_1^\vee\otimes_{\cR_E}D_2,\EndO(D_1))          }.
\end{equation}
Here, the first vertical inclusion is given by $\ext^1_{(\varphi,\Gamma)}(D_2,D_1)\hookrightarrow \ext^1_{(\varphi,\Gamma)}(D_2,\EndO(D_1)\otimes D_1)\cong \ext^1_{(\varphi,\Gamma)}(D_1^\vee\otimes_{\cR_E}D_2,\EndO(D_1)) $ and the second vertical inclusion is given by $\ext^1_{(\varphi,\Gamma)}(D_1,D_1)\hookrightarrow \ext^1_{(\varphi,\Gamma)}(D_1,\EndO(D_1)\otimes D_1)\cong \ext^1_{(\varphi,\Gamma)}(\EndO(D_1),\EndO(D_1))$.\;Our constructions of $\sL(\Dpik)$ indicate the correspondence between $\cL(\mathrm{Ad}(V),D)$ and $\sL(\Dpik)$.
\end{rmk}

\end{document}